\documentclass[12pt,a4paper]{amsart}
\usepackage{amssymb,amsxtra,eucal}
\usepackage[cmtip,arrow]{xy}
\usepackage{pb-diagram,pb-xy}

\calclayout

\newcommand{\+}{\nobreakdash-}
\renewcommand{\:}{\colon}

\newcommand{\rarrow}{\longrightarrow}
\newcommand{\ot}{\otimes}
\newcommand{\ocn}{\odot} 

\newcommand{\bu}{{\text{\smaller\smaller$\scriptstyle\bullet$}}}
\newcommand{\lrarrow}{\mskip.5\thinmuskip\relbar\joinrel\relbar\joinrel
 \rightarrow\mskip.5\thinmuskip\relax} 

\DeclareMathOperator{\Hom}{Hom}
\DeclareMathOperator{\Tor}{Tor}
\DeclareMathOperator{\Ext}{Ext}
\DeclareMathOperator{\PExt}{PExt}
\DeclareMathOperator{\coker}{coker}
\DeclareMathOperator{\pd}{pd}
\DeclareMathOperator{\fd}{fd}
\newcommand{\Id}{\mathrm{Id}}
\newcommand{\id}{\mathrm{id}}

\newcommand{\fA}{\mathfrak A}
\newcommand{\C}{\mathfrak C}
\newcommand{\D}{\mathfrak D}
\newcommand{\fH}{\mathfrak H}
\newcommand{\J}{\mathfrak J}
\newcommand{\fK}{\mathfrak K}
\newcommand{\fP}{\mathfrak P}
\newcommand{\R}{\mathfrak R}
\renewcommand{\S}{\mathfrak S}
\newcommand{\T}{\mathfrak T}
\newcommand{\fU}{\mathfrak U}

\newcommand{\sA}{\mathsf A}
\newcommand{\sB}{\mathsf B}
\newcommand{\sC}{\mathsf C}
\newcommand{\sD}{\mathsf D}
\newcommand{\sE}{\mathsf E}
\newcommand{\sF}{\mathsf F}
\newcommand{\sG}{\mathsf G}
\newcommand{\sH}{\mathsf H}
\newcommand{\sL}{\mathsf L}
\newcommand{\sM}{\mathsf M}
\newcommand{\sN}{\mathsf N}
\newcommand{\sP}{\mathsf P}

\newcommand{\sV}{\mathsf V}

\newcommand{\Add}{\mathsf{Add}}
\newcommand{\Prod}{\mathsf{Prod}}
\newcommand{\Ab}{\mathsf{Ab}}
\newcommand{\Sets}{\mathsf{Sets}}

\newcommand{\proj}{\mathsf{proj}}
\newcommand{\inj}{\mathsf{inj}}
\newcommand{\bb}{\mathsf{b}}

\newcommand{\rop}{{\mathrm{op}}}
\newcommand{\sop}{{\mathsf{op}}}

\newcommand{\boZ}{\mathbb Z}
\newcommand{\boQ}{\mathbb Q} 
\newcommand{\boT}{\mathbb T}
\newcommand{\boL}{\mathbb L}
\newcommand{\boR}{\mathbb R}

\newcommand{\modl}{{\operatorname{\mathsf{--mod}}}}
\newcommand{\modr}{{\operatorname{\mathsf{mod--}}}}
\newcommand{\contra}{{\operatorname{\mathsf{--contra}}}}
\newcommand{\discr}{{\operatorname{\mathsf{discr--}}}}
\newcommand{\co}{{\operatorname{\mathsf{-co}}}}
\newcommand{\ctra}{{\operatorname{\mathsf{-ctra}}}}

\newcommand{\Section}[1]{\bigskip\section{#1}\medskip}
\theoremstyle{plain}
\newtheorem{thm}{Theorem}[section]
\newtheorem{conj}[thm]{Conjecture}
\newtheorem{mc}[thm]{Main Conjecture}
\newtheorem{lem}[thm]{Lemma}
\newtheorem{prop}[thm]{Proposition}
\newtheorem{cor}[thm]{Corollary}
\theoremstyle{definition}
\newtheorem{ex}[thm]{Example}
\newtheorem{exs}[thm]{Examples}
\newtheorem{rem}[thm]{Remark}
\newtheorem{rems}[thm]{Remarks}
\newtheorem{dfn}[thm]{Definition}

\numberwithin{equation}{section}

\begin{document}

\title{Covers and direct limits: \\ A contramodule-based approach}

\author[S.~Bazzoni]{Silvana Bazzoni}

\address[Silvana Bazzoni]{%
Dipartimento di Matematica ``Tullio Levi-Civita'' \\
Universit\`a di Padova \\
Via Trieste 63, 35121 Padova (Italy)}

\email{bazzoni@math.unipd.it}

\author[L.~Positselski]{Leonid Positselski}

\address[Leonid Positselski]{%
Institute of Mathematics of the Czech Academy of Sciences \\
\v Zitn\'a~25, 115~67 Praha~1 (Czech Republic); and
\newline\indent
Laboratory of Algebra and Number Theory \\
Institute for Information Transmission Problems \\
Moscow 127051 (Russia)}

\email{positselski@yandex.ru}

\begin{abstract}
 We present applications of contramodule techniques to the Enochs
conjecture about covers and direct limits, both in the categorical
tilting context and beyond.
 In the $n$\+tilting-cotilting correspondence situation,
if $\sA$ is a Grothendieck abelian category and the related abelian
category $\sB$ is equivalent to the category of contramodules over
a topological ring $\R$ belonging to one of certain four classes of
topological rings (e.~g., $\R$ is commutative), then the left tilting
class is covering in $\sA$ if and only if it is closed under direct
limits in $\sA$, and if and only if all the discrete quotient rings of
the topological ring $\R$ are perfect.
 More generally, if $M$ is a module satisfying a certain telescope Hom
exactness condition (e.~g., $M$ is
$\Sigma$\+pure-$\Ext^1$-self-orthogonal) and the topological ring $\R$
of endomorphisms of $M$ belongs to one of certain seven classes of
topological rings, then the class $\Add(M)$ is closed under direct
limits if and only if every countable direct limit of copies of $M$
has an $\Add(M)$\+cover, and if and only if $M$ has
perfect decomposition.
 In full generality, for an additive category $\sA$ with (co)kernels
and a precovering class $\sL\subset\sA$ closed under summands, an object
$N\in\sA$ has an $\sL$\+cover if and only if a certain object $\Psi(N)$
in an abelian category $\sB$ with enough projectives has
a projective cover.
 The $1$\+tilting modules and objects arising from injective ring
epimorphisms of projective dimension~$1$ form a class of examples
which we discuss.
\end{abstract}

\maketitle

\tableofcontents

\section*{Introduction}
\medskip

\setcounter{subsection}{-1}
\subsection{{}}
 The main result (or one of the main results) of Bass'
1960~paper~\cite{Bas} can be stated as follows: given an associative
ring $R$, every left $R$\+module has a projective cover if and only if
the class of projective modules is closed under direct limits in
the category of left $R$\+modules.
 Subsequently, in 1981 Enochs proved that any precovering class of
modules closed under direct limits is
covering~\cite[Theorems~2.1 and~3.1]{Eno}, and in the late 1990s
he asked the question whether every covering class of modules is closed
under direct limits (see~\cite[Section~5.4]{GT};
cf.~\cite[Section~5]{AST}).

 A hypothetical general positive answer to this question is sometimes
called ``the Enochs conjecture''.
 A positive answer in many particular cases was recently obtained
by Angeleri H\"ugel, \v Saroch, and Trlifaj~\cite[Theorem~5.2 and
Corollary~5.5]{AST}, based on set-theoretical tools developed
by \v Saroch in~\cite{Sar}.
 (An alternative elementary proof of a part of the results
of~\cite{AST} is suggested in the preprint~\cite{BPS}.)
 The aim of this paper is to offer a new approach to proving
particular cases of the Enochs conjecture, based on the recently
developed techniques of contramodules and categorical
tilting theory~\cite{PR,PS1,PS2,Pcoun,Pproperf,PS3,BPS}.

\subsection{{}}
 The general idea of our approach can be explained as follows.
 Firstly, we extend Bass' theorem about projective covers from
the categories of modules over associative rings to some other
abelian categories $\sB$ with enough projective objects.
 This is the subject of the paper~\cite{Pproperf}.
 
 Secondly, let $A$ be an associative ring and $M$ be a left $A$\+module.
 More generally, $M$ could be an object of a good enough
additive/abelian category $\sA$ in lieu of $A\modl$.
 We consider the full subcategory $\Add(M)\subset\sA$ consisting
of all the direct summands of coproducts of copies of $M$ in~$\sA$.
 The aim is to prove the Enochs conjecture for the class of objects
$\Add(M)$ in~$\sA$.

 For this purpose, we find an abelian category $\sB$ such that
the full subcategory $\sB_\proj\subset\sB$ of projective objects in
$\sB$ is equivalent to the full subcategory $\Add(M)\subset\sA$.
 Then we transfer our knowledge about the Enochs conjecture for
the class of projective objects $\sB_\proj$ in $\sB$ to
the class of objects $\Add(M)$ in~$\sA$.

 In fact, we do more.
 Extending the discussion in~\cite{AST} to the category-theoretic
context, we consider covers in cotorsion pairs, self-pure-projective
and $\varinjlim$\+pure-rigid objects, and objects with perfect
decomposition.
 Under certain assumptions, we prove that the class $\Add(M)$ is
covering in $\sA$ if and only if the object $M\in\Add(M)$ has
a perfect decomposition.
 This is based on some results of the papers~\cite{PS3}
and~\cite{BPS}.

 One specific feature of our approach is that we consider topologies
on (the opposite ring to) the ring of endomorphisms
$\R=\Hom_\sA(M,M)^\rop$ of the object~$M$.
 In particular, the endomorphism ring of a module $M$ over
an associative ring $A$ always has the so-called \emph{finite topology}.
 Under certain assumptions, we prove that the class $\Add(M)$ is
covering in $\sA$ if and only if all the discrete quotient rings of
the topological ring $\R$ are left perfect.

\subsection{{}} \label{introd-assumptions-subsecn}
 The time has come to explain what our assumptions are.
 There are three kinds of assumptions.
 Firstly, given an object $M$ in a category $\sA$, there should exist
a topology on the ring $\R$ of endomorphisms of $M$ for which
the abelian category $\sB$ could be described as the category of
left $\R$\+contramodules.
 This always holds when $\sA=A\modl$ is the category of modules over
an associative ring, and more generally, when $\sA$ is a locally
finitely generated abelian category (and in some other cases, too).

 Secondly, the topological ring $\R$ has to satisfy one of the technical
assumptions~(a), (b), (c), or~(d) under which the main results of
the paper~\cite{Pproperf} are proved.
 In particular, the condition~(a) says that the ring $\R$ is commutative
(and when it is not, there are three other alternatives~(b), (c), or~(d)
which may happen to hold for~$\R$).

 Alternatively, there are three conditions~(e), (f), and~(g), under any
one of which some of our results in this paper can be proved using
the main results of the papers~\cite{PS3,Ro}.
 In particular, (e)~says that $\R$ has a countable base of neighborhoods
of zero.

 Thirdly, there is a more conceptual assumption which we call
``telescope Hom exactness condition'', abbreviated as~THEC\@.
 This condition is not very restrictive.
 It says that right exactness of the telescope sequences computing
countable direct limits of copies of the object $M$ in $\sA$
is preserved by the functor $\Hom_\sA(M,{-})$.
 All $\Sigma$\+pure-rigid and all self-pure-projective objects (hence,
in particular, all $n$\+tilting objects) in abelian categories with
exact countable direct limits satisfy THEC\@.

\subsection{{}}
 Having mentioned the assumptions, we can now formulate our main result.
 
\begin{thm} \label{introd-THEC-main-theorem}
 Let\/ $\sA$ be a locally presentable additive category and $M\in\sA$
be an object satisfying THEC\@.
 Denote by\/ $\sB$ the abelian category with enough projective objects
such that the full subcategory\/ $\Add(M)\subset\sA$ is equivalent
to the full subcategory of projective objects\/ $\sB_\proj\subset\sB$.
 Assume that there exists a (complete, separated, right linear)
topological ring structure on the ring\/ $\R=\Hom_\sA(M,M)^\rop$
such that the abelian category\/ $\sB$ is equivalent to the abelian
category of left\/ $\R$\+contramodules\/ $\R\contra$ (this always holds
for\/ $\sA=A\modl$).
 Finally, assume that the topological ring\/ $\R$ satisfies one of
the conditions (a), (b), (c), or~(d) of the paper~\cite{Pproperf}
(e.~g., this holds if\/ $\R$ is commutative).
 Then the following conditions are equivalent:
\begin{enumerate}
\item the class of objects\/ $\Add(M)\subset\sA$ is covering;
\item every countable direct limit of copies of $M$ has
an\/ $\Add(M)$\+cover in\/~$\sA$;
\item the class of objects\/ $\Add(M)$ is closed under direct limits
in\/~$\sA$;
\item the class\/ $\sB_\proj$ is covering in\/~$\sB$;
\item any countable direct limit of copies of the projective
generator\/ $\R\in\sB$ has a projective cover in\/~$\sB$;
\item the class\/ $\sB_\proj$ is closed under direct limits in\/~$\sB$;
\item the object\/ $M\in\sA$ has a perfect decomposition;
\item all descending chains of cyclic discrete right $\R$\+modules
terminate;
\item all the discrete quotient rings of the topological ring\/ $\R$
are left perfect.
\end{enumerate}
 Replacing the assumption of one of the conditions~(a\+-d) with that
of one of the conditions~(e), (f), or~(g) (e.~g., if\/ $\R$ has
a countable base of neighborhoods of zero), the eight
conditions~\textup{(1\+-8)} are equivalent.
\end{thm}

 Notice that, even in the case of the category of modules $\sA=A\modl$,
one can sometimes choose between several topologies on the ring $\R$
for which the category $\sB$ in Theorem~\ref{introd-THEC-main-theorem}
is equivalent to $\R\contra$.
 In particular, when the $A$\+module $M$ is \emph{self-small}, i.~e.,
the natural map of abelian groups $\bigoplus_{i=0}^\infty
\Hom_A(M,M)\rarrow\Hom_A(M,\>\bigoplus_{i=0}^\infty M)$ is
an isomorphism, it suffices to endow the ring $\R$ with
the \emph{discrete} topology.
 Then the condition~(b) is satisfied.

 Furthermore, suppose that a left $A$\+module $M=\sum_{i=1}^\infty E_i$
is the sum of a countable family of its submodules $E_i\subset M$
such that the $A$\+modules $E_i$ are \emph{weakly finitely generated}
(known also as ``small'' or ``dually slender'').
 This means that for any family of left $A$\+modules $(N_x)_{x\in X}$,
the natural map $\bigoplus_x\Hom_A(E_i,N_x)\rarrow
\Hom_A(E_i,\>\bigoplus_xN_x)$ is an isomorphism for every
$i=1$, $2$,~\dots{}
 Then one can endow the ring $\R$ with the \emph{weakly finite}
topology, and the condition~(e) is satisfied
(cf.~\cite[Section~7.2]{BPS}).

\subsection{{}}
 Specializing to the tilting context, we prove the following theorem
with our methods (cf.~\cite[Theorem~5.2 and Corollary~5.5]{AST}).

\begin{thm} \label{introd-tilting-main-theorem}
 Let\/ $\sA$ be a Grothendieck abelian category and $T\in\sA$ be
an $n$\+tilting object.
 Let $(\sL,\sE)$ denote the induced $n$\+tilting cotorsion pair
in\/~$\sA$, and let\/ $\sB$ denote the heart of the related $n$\+tilting
t\+structure on\/ $\sD(\sA)$.
 Assume that there exists a (complete, separated, right linear)
topology on the ring\/ $\R=\Hom_\sA(T,T)^\rop$ such that the abelian
category\/ $\sB$ is equivalent to the abelian category of left\/
$\R$\+contramodules\/ $\R\contra$ (this always holds when\/ $\sA$ is
a locally weakly finitely generated abelian category).
 Finally, assume that the topological ring\/ $\R$ satisfies one of
the conditions (a), (b), (c), or~(d). 
 Then the following conditions are equivalent:
\begin{enumerate}
\item the class\/ $\sL$ is covering in\/~$\sA$;
\item any countable direct limit of copies of $T$ has
an\/ $\sL$\+cover in\/~$\sA$;
\item the class\/ $\sL$ is closed under direct limits in\/~$\sA$;
\item the class\/ $\Add(T)$ is covering in\/~$\sA$;
\item any countable direct limit of copies of $T$ has
an\/ $\Add(T)$\+cover in\/~$\sA$;
\item the class\/ $\Add(T)$ is closed under direct limits in\/~$\sA$;
\item any or all of the equivalent conditions (4\+-6) of
Theorem~\ref{introd-THEC-main-theorem} hold for the category\/
$\sB=\R\contra$;
\item the object\/ $T\in\sA$ has a perfect decomposition;
\item all descending chains of cyclic discrete right\/ $\R$\+modules
terminate;
\item all the discrete quotient rings of the topological ring\/ $\R$
are left perfect.
\end{enumerate}
 Replacing the assumption of one of the conditions~(a\+-d) with that
of one of the conditions~(e), (f), or~(g), the nine
conditions~\textup{(1\+-9)} are equivalent.
\end{thm}

\subsection{{}}
 In the full generality (without any of the assumptions mentioned in
Section~\ref{introd-assumptions-subsecn}), we make the following
simple observations.

 Let $\sA$ be an additive category with cokernels and (weak) kernels,
and $\sL\subset\sA$ be a precovering class of objects closed under
direct summands.
 Viewing $\sL$ as a full subcategory in $\sA$, we notice that $\sL$
has weak kernels, too.
 So there exists a unique abelian category $\sB$ with enough projectives
such that the full subcategory of projectives in $\sB$ is equivalent
to~$\sL$ \cite[Corollary~1.5]{Fr}, \cite[Proposition~2.3]{Kra}.
 Furthermore, the equivalence of categories $\sB_\proj\cong\sL$
can be naturally extended to a pair of adjoint functors
$\Phi\:\sB\rarrow\sA$ and $\Psi\:\sA\rarrow\sB$ (where $\Psi$ is
the right adjoint).

 Let $N\in\sA$ be an object.
 Then $N$ has an $\sL$\+cover in $\sA$ if and only if the object
$\Psi(N)\in\sB$ has a projective cover.
 More specifically, given an object $L\in\sL$ and the related
object $\Psi(L)=P\in\sB_\proj$, a morphism $l\:L\rarrow N$ is
an $\sL$\+cover if and only if the morphism
$\Psi(l)\:P\rarrow\Psi(N)$ is a projective cover.
 Hence the class $\sL$ is covering in $\sA$ if and only if
all the objects in the essential image of the functor $\Psi$ have
projective covers in~$\sB$.

\subsection{{}}
 In the final sections of the paper, we discuss the class of examples
for Theorem~\ref{introd-tilting-main-theorem} provided by the tilting
modules and objects arising from injective homological ring epimorphisms
of projective dimension~$1$.
 Here our discussion is based on the paper~\cite{BP2}.

 In fact, there are two classes of examples.
 Let $u\:R\rarrow U$ be an injective homological epimorphism of
associative rings such that the projective dimension of the left
$R$\+module $U$ does not exceed~$1$.
 Then the left $R$\+module $U\oplus U/R$ is $1$\+tilting.
 If the ring $R$ is commutative, then the condition~(d) is satisfied for
the topological ring $\S$ of endomorphisms of the $R$\+module
$U\oplus U/R$, and Theorem~\ref{introd-tilting-main-theorem} is
applicable for $\sA=R\modl$ and $T=U\oplus U/R$.

 Assume additionally that the flat dimension of the right $R$\+module
$U$ does not exceed~$1$.
 Then we consider the full subcategory $\sA=R\modl_{u\co}$ of what
we call \emph{left $u$\+comodules} in the category of left
$R$\+modules $R\modl$.
 The category $\sA$ is a Grothendieck abelian category, and the left
$R$\+module $U/R$ is a $1$\+tilting object in~$\sA$.
 If the ring $R$ is commutative, then so is the topological ring
$\R=\Hom_R(U/R,U/R)^\rop$, and
Theorem~\ref{introd-tilting-main-theorem} is applicable for
$\sA=R\modl_{u\co}$ and $T=U/R$.

\subsection{{}}
 In conclusion, let us say a few words about how our results compare
to those of the paper~\cite{AST}.
 Our results are both more and less general than the results
of~\cite{AST}.
 On the one hand, the paper~\cite{AST} only deals with cotorsion pairs
in module categories, while we work in more general additive and
abelian categories.
 On the other hand, the main results of the present paper require
one of the rather restrictive conditions~(a), (b), (c), (d), (e), (f),
or~(g), while there are no comparable assumptions in~\cite{AST}.

 Even for module categories $\sA=A\modl$, our
Theorem~\ref{introd-THEC-main-theorem} is both stronger and weaker
than the results of~\cite{AST}.
 On the one hand, we do not assume that the object $M$ belongs to
the kernel of a cotorsion pair.
 The running assumption in~\cite{AST} is that of a cotorsion pair
$(\sL,\sE)$ in $A\modl$ such that the right-hand class $\sE$ is
closed under direct limits.
 Under this assumption, any module $M\in\sL\cap\sE$ satisfies our
telescope Hom exactness condition (in fact, it is enough that
$\sE$ be closed under countable coproducts).
 So, in this respect, our setting is more general.

 On the other hand, the assertions of~\cite[Theorem~5.2 and
Corollary~5.5]{AST} tell more than those of our theorems.
 In particular, \cite[Corollary~5.5\,(5)]{AST} allows to conclude
that the module in the kernel of the cotorsion pair is
$\Sigma$\+pure-split, while we only prove that our object $M$
has a perfect decomposition.

\subsection{Acknowledgement}
 The authors are grateful to Jan \v St\!'ov\'\i\v cek, Michal Hrbek,
Rosanna Laking, and Jan \v Saroch for very helpful discussions.
 We wish to thank an anonymous referee for careful reading of
the manuscript and for several helpful suggestions.
 In particular, the argument in the proof of
Corollary~\ref{uncountable-direct-limit-closedness-cor} is due to
the referee.
 The first-named author was partially supported by MIUR-PRIN
(Categories, Algebras: Ring-Theoretical and Homological
Approaches-CARTHA) and DOR1828909 of Padova University.
 The second-named author is supported by the GA\v CR project 20-13778S
and research plan RVO:~67985840.

\Section{Contramodules over Topological Rings}
\label{contramodules-secn}

 Cocomplete abelian categories with enough projective objects,
and more specifically contramodule categories, play a key role in
this paper.
 In this section, we briefly recall the basic material related to
contramodules over complete, separated topological rings with
right linear topologies.
 More details can be found in~\cite[Section~1]{Pproperf},
\cite[Section~2]{Pcoun}, \cite[Introduction and Section~5]{PR},
\cite[Section~6]{PS1}, and~\cite[Section~1]{Pper}.

\subsection{Linear topological abelian groups}
 A topological abelian group $A$ is said to have a \emph{linear
topology} if open subgroups form a base of neighborhoods of zero in~$A$.
 A topological abelian group $A$ with a linear topology (a ``linear
topological abelian group'', for brevity) is \emph{separated} if
the natural map $\lambda_A\:A\rarrow\varprojlim_{U\subset A}A/U$, where
$U$ ranges over the open subgroups of $A$, is injective, and $A$ is
\emph{complete} if the map~$\lambda_A$ is surjective.
 Obviously, $A$ is separated if and only if the intersection of all its
open subgroups is zero.

 For any abelian group $A$ and a set $X$, we use the notation
$A[X]=A^{(X)}$ for the coproduct of $X$ copies of~$A$.
 The elements of $A[X]$ are interpreted as finite formal linear
combinations of elements of $X$ with the coefficients in~$A$.

 Let $\fA$ be a complete, separated linear topological abelian group.
 For any set $X$, we denote by $\fA[[X]]$ the projective limit
$$
 \fA[[X]]=\varprojlim\nolimits_{\fU\subset\fA}(\fA/\fU)[X],
$$
where $\fU$ ranges over all the open subgroups of~$\fA$.
 Equivalently, $\fA[[X]]$ is the group of all infinite formal linear
combinations $\sum_{x\in X} a_xx$ of elements of the set $X$ with
the coefficients $a_x\in\fA$ such that the family of coefficients
$(a_x)_{x\in X}$ converges to zero in $\fA$ in the following sense:
for any open subgroup $\fU\subset\fA$, the set of all indices $x\in X$
for which $a_x\notin\fU$ must be finite.

 For any complete, separated linear topological abelian group $\fA$
and any map of sets $f\:X\rarrow Y$ there is a naturally induced
``push-forward'' map $\fA[[f]]\:\fA[[X]]\rarrow\fA[[Y]]$ taking
a formal linear combination $\sum_{x\in X}a_xx$ to the formal linear
combination $\sum_{y\in Y}b_yy$ with the coefficients
$b_y=\sum_{x:f(x)=y}a_x$.
 Here the latter sum is understood as the limit of finite partial
sums in the topology of $\fA$; the convergence condition on
the family of elements $(a_x)_{x\in X}$ together with the conditions
of separatedness and completeness of $\fA$ guarantee that
the coefficients~$b_y$ are well-defined (and form a family of elements
$(b_y)_{y\in Y}$ which again converges to zero in~$\fA$).
 This construction shows that the assignment $X\longmapsto \fA[[X]]$
is a functor from the category of sets to the category of sets or even
abelian groups.

\subsection{Monads on $\Sets$} \label{monads-subsecn}
 A \emph{monad} $\boT$ on the category of sets is a functor
$\boT\:\Sets\rarrow\Sets$ endowed with natural transformations of
\emph{monad unit} $\epsilon\:\Id_\Sets\rarrow\boT$ and \emph{monad
multiplication} $\phi\:\boT\circ\boT\rarrow\boT$ satisfying
the following associativity and unitality equations.
 The two natural maps $\boT(\phi_X)\:\boT(\boT(\boT(X)))
\rarrow\boT(\boT(X))$ and $\phi_{\boT(X)}\:\boT(\boT(\boT(X)))\rarrow
\boT(\boT(X))$ should have equal compositions with the map
$\phi_X\:\boT(\boT(X))\rarrow\boT(X)$ for any set $X$,
$$
 \boT\circ\boT\circ\boT\rightrightarrows\boT\circ\boT\rarrow\boT,
$$
and both the natural maps $\boT(\epsilon_X)\:\boT(X)\rarrow
\boT(\boT(X))$ and $\epsilon_{\boT(X)}\:\boT(X)\rarrow\boT(\boT(X))$
composed with the natural map~$\phi_X$ should be equal to the identity
endomorphism of the set $\boT(X)$,
$$
 \boT\rightrightarrows\boT\circ\boT\rarrow\boT.
$$
 Here $\epsilon_X$ denotes the map $X\rarrow\boT(X)$ assigned to
an object $X\in\Sets$ by the natural transformation~$\epsilon$,
and similarly, $\phi_X\:\boT(\boT(X))\rarrow\boT(X)$ is the map
assigned to $X$ by the natural transformation~$\phi$.

 A \emph{module} (or, in a more standard terminology, an algebra)
over a monad $\boT\:\Sets\rarrow\Sets$ is a set $C$ endowed with
a map of sets $\pi_C\:\boT(C)\rarrow C$, called the \emph{monad action}
map, satisfying the following associativity and unitality equations.
 The compositions of the two maps $\phi_C$ and $\boT(\pi_C)\:
\boT(\boT(C))\rarrow\boT(C)$ with the map~$\pi_C$ should be equal to
each other,
$$
 \boT(\boT(C))\rightrightarrows\boT(C)\rarrow C,
$$
and the composition of the map $\epsilon_C\:C\rarrow\boT(C)$ with
the map $\pi_C\:\boT(C)\rarrow C$ should be equal to the identity
map~$\id_C$,
$$
 C\rarrow\boT(C)\rarrow C.
$$

 A \emph{morphism of\/ $\boT$\+modules} $f\:B\rarrow C$ is a map of
sets for which the following square diagram is commutative:
$$
\begin{diagram}
\node{\boT(B)} \arrow{e,t}{\pi_B}\arrow{s,l}{\boT(f)} 
\node{B} \arrow{s,l}{f} \\
\node{\boT(C)} \arrow{e,t}{\pi_C} \node{C}
\end{diagram}
$$
 The composition of morphisms of $\boT$\+modules is defined in
the obvious way.

 For any monad $\boT\:\Sets\rarrow\Sets$, the category
of $\boT$\+modules $\boT\modl$ is complete and cocomplete.
 For any set $X$, the set $\boT(X)$ with the action map
$\pi_{\boT(X)}=\phi_X$ is a $\boT$\+module; such $\boT$\+modules
are called the \emph{free} $\boT$\+modules.
 For any $\boT$\+module $C$, morphisms of $\boT$\+modules $\boT(X)
\rarrow C$ are in bijective correspondence with maps of sets
$X\rarrow C$.

 A monad $\boT\:\Sets\rarrow\Sets$ is said to be \emph{additive}
if the category of $\boT$\+modules $\boT\modl$ is additive.
 In this case, the underlying set of every $\boT$\+module has
a natural abelian group structure; so the forgetful functor
$\boT\modl\rarrow\Sets$ lifts naturally to a forgetful functor
$\boT\modl\rarrow\Ab$.
 For any additive monad $\boT$, the category $\boT\modl$ is abelian;
the forgetful functor $\boT\modl\rarrow\Ab$ is faithful, exact,
and preserves all limits~\cite[Lemma~1.1]{Pper}.
 For any additive monad $\boT$, the abelian category of $\boT$\+modules
$\boT\modl$ has enough projective objects.
 A $\boT$\+module is projective if and only if it is a direct summand
of a free $\boT$\+module.

\subsection{Right linear topological rings}
 All \emph{rings} in this paper are presumed to be associative
and unital.
 A topological ring $R$ is said to have a \emph{right linear topology}
if open right ideals form a base of neighborhoods of zero in~$R$.
 A \emph{two-sided linear topology} on $R$ is a topology in which
open two-sided ideals form a base of neighborhoods of zero.
 When the ring $R$ is commutative, one simply says that ``$R$ has
a linear topology'' if open ideals form a base of neighborhoods of zero.
 A topological ring with a right (resp., two-sided) linear topology is
called \emph{right} (resp., \emph{two-sided}) \emph{linear topological}
(or just ``linear topological'', if the ring is commutative).

 Let $\R$ be a complete, separated right linear topological ring.
 Then the functor $\boT_\R\:X\longmapsto\R[[X]]$ has a natural structure
of a monad on the category of sets.
 By the definition (see Section~\ref{monads-subsecn}), this means that
there are natural transformations of monad unit
$\epsilon\:\Id_{\Sets}\rarrow\boT_\R$ and monad multiplication
$\phi\:\boT_\R\circ\boT_\R\rarrow\boT_\R$ satisfying
the associativity and unitality equations.

 For any set $X$, the natural ``point measure'' map
$\epsilon_X\:X\rarrow\R[[X]]$ assigns to an element $x\in X$
the formal linear combination $\sum_{z\in X}r_zz$ with the coefficients
$r_x=1$ and $r_z=0$ for $z\ne x$.
 The natural ``opening of parentheses'' map $\phi_X\:\R[[\R[[X]]]]
\rarrow\R[[X]]$ assigns to a formal linear combination
$\sum_{y\in\R[[X]]}r_yy$, where $y=\sum_{x\in X}s_{y,x}x\in\R[[X]]$
and $r_y$, $s_{y,x}\in\R$, the formal linear combination
$\sum_{x\in X}t_xx\in\R[[X]]$ with the coefficients
$t_x=\sum_{y\in\R[[X]]}r_ys_{y,x}\in\R$.
 Here the infinite sum in the construction of the coefficient~$t_x$
is understood as the limit of finite partial sums in the topology
of $\R$, and the conditions of right linear topology, completeness,
and separatedness imposed on the ring $\R$ guarantee the convergence.

\subsection{Contramodules} \label{introd-contramodules-subsecn}
 A \emph{left contramodule} over a complete, separated right linear
topological ring $\R$ is a module (or, in the more standard terminology,
an algebra) over the monad~$\boT_\R$.
 In other words, a left $\R$\+contramodule $\C$ is a set endowed with
a \emph{left contraaction} map $\pi_\C\:\R[[\C]]\rarrow\C$ satisfying
the associativity and unitality equations written down in
Section~\ref{monads-subsecn}.
 
 Restricting the map~$\pi_\C$ to the subset of finite formal
linear combinations $\R[X]\subset\R[[X]]$, one obtains the structure
of a module over the monad $X\longmapsto\R[X]$ on the underlying set
of every left $\R$\+contramodule, which is the same as a left
$\R$\+module structure.
 This construction defines a natural forgetful functor $\R\contra
\rarrow\R\modl$ from the category of left $\R$\+contramodules to
the category of left $\R$\+modules.
 The monad $\boT_\R$ is additive, the category $\R\contra$ is abelian,
and the forgetful functor $\R\contra\rarrow\R\modl$ is exact and
preserves infinite products (but not coproducts).
 
 For any set $X$, the free $\boT_\R$\+module $\boT_\R(X)=\R[[X]]$
(with the contraaction map $\pi_{\R[[X]]}=\phi_X$) is called
the \emph{free} left $\R$\+contramodule generated by~$X$.
 Following the discussion in Section~\ref{monads-subsecn}, for every
left $\R$\+contramodule $\C$, left $\R$\+contramodule morphisms
$\R[[X]]\rarrow\C$ are in bijective correspondence with maps of
sets $X\rarrow\C$,
$$
 \Hom^\R(\R[[X]],\C)\cong\Hom_\Sets(X,\C),
$$
where we denote by $\Hom^\R(\C,\D)$ the group of morphisms between
any two objects $\C$ and $\D$ in the category $\R\contra$.
 There are enough projective objects in the abelian category
$\R\contra$; a left $\R$\+contramodule is projective if and only if
it is a direct summand of a free left $\R$\+contramodule.

\subsection{Discrete modules}
 Let $R$ be a right linear topological ring.
 A right $R$\+module $N$ is said to be \emph{discrete} if, for every
element $x\in N$, the annihilator of~$x$ in $R$ is an open right ideal.
 Equivalently, this means that the action map $N\times R\rarrow N$ is
continuous in the given topology on $R$ and the discrete topology
on~$N$.
 The full subcategory of discrete right $R$\+modules $\discr R$ is
closed under subobjects, quotient objects, and infinite direct sums
in the abelian category of right $R$\+modules $\modr R$ (in other words,
$\discr R\subset\modr R$ is a hereditary pretorsion class).
 It follows that $\discr R$ is a locally finitely generated
Grothendieck abelian category.

\Section{Generalized Tilting Theory}

 Let $\sA$ be an additive category with set-indexed coproducts, and let
$\sB$ be an additive category with set-indexed products.
 For any object $T\in\sA$ and any set $X$, we denote by $T^{(X)}\in\sA$
the coproduct of $X$ copies of $T$ in~$\sA$.
 For any object $W\in\sB$ and any set $X$, we denote by $W^X\in\sB$
the product of $X$ copies of $W$ in~$\sB$.

 Furthermore, we denote by $\Add(T)=\Add_\sA(T)\subset\sA$ the class of
all direct summands of the coproducts $T^{(X)}$ of copies of the object
$T$ in the category~$\sA$.
 Similarly, we denote by $\Prod(W)=\Prod_\sB(W)\subset\sB$ the class of
all direct summands of the products $W^X$ of copies of the object $W$
in~$\sB$.

 Given an exact category $\sE$ (in Quillen's sense), we denote by
$\sE_\inj$ and $\sE_\proj\subset\sE$ the classes of all injective
and projective objects in $\sA$, respectively.
 In particular, this notation applies to abelian categories.

 Let $\sA$ be an idempotent-complete additive category with set-indexed
coproducts, and let $M\in\sA$ be an object.
 In this section we recall the description of the category $\Add(M)$ as
the category $\sB_\proj$ of projective objects in a certain abelian
category~$\sB$.
 This material first appeared in~\cite[Section~6]{PS1}
and~\cite[Section~1]{PS2}.

\begin{rem}
 The latter two references are papers in tilting theory.
 So let us briefly explain the connection, which will also explain
the title of this section and its first subsection, following below.
 In the infinitely generated tilting theory, one assigns to
a cocomplete abelian category $\sA$ with an $n$\+tilting object $T$
another abelian category $\sB$, which is constructed as the heart of
the tilting t\+structure on the derived category $\sD(\sA)$.
 One observes that the abelian category $\sB$ has enough projective
objects, and the full subcategory of projective objects in $\sB$
is equivalent to the full subcategory $\Add(T)\subset\sA$.
 (See Section~\ref{tilting-cotilting-secn} for a detailed discussion.)
 The next observation is that one does not need a tilting object to
perform such a construction: for any object $M\in\sA$, there exists
a unique abelian category $\sB$ with enough projective objects
such that $\sB_\proj\cong\Add(M)$.
 Hence the name ``generalized tilting theory'' which we give to this
categorical construction and its basic properties.
\end{rem}

\subsection{Generalized tilting theory}
\label{generalized-tilting-subsecn}
 Let $\sA$ be a category with coproducts and $M\in\sA$ be an object.
 Consider the pair of adjoint functors
$$
 \Phi\:\Sets\rightleftarrows\sA\,:\!\Psi
$$
defined as follows.
 For any set $X$, the object $\Phi(X)=M^{(X)}$ is the coproduct of
$X$ copies of $M$ in~$\sA$.
 For any object $N\in\sA$, the set $\Psi(N)=\Hom_\sA(M,N)$ is
the set of all morphisms $M\rarrow N$ in the category~$\sA$.
 The composition of the two adjoint functors $\boT_M=\Psi\circ\Phi\:
\Sets\rarrow\Sets$, taking a set $X$ to the set $\boT_M(X)=
\Hom_\sA(M,M^{(X)})$, acquires a natural structure of a monad on
the category of sets (see Section~\ref{monads-subsecn}).
 According to~\cite[Proposition~6.2]{PS1}, the full subcategory formed
by the objects $M^{(X)}$, \ $X\in\Sets$, in the category $\sA$ is
equivalent to the full subcategory of free $\boT_M$\+modules $\boT_M(X)$
in $\boT\modl$.

 Let $\sB$ be a cocomplete abelian category with a projective
generator~$P$.
 Then the related monad $\boT_P\:X\longmapsto\Hom_\sB(P,P^{(X)})$ is
additive, and the abelian category $\sB$ is equivalent to the abelian
category of $\boT_P$\+modules~\cite[Corollary~6.3]{PS1}:
\begin{equation} \label{cocomplete-abelian-proj-gen-monad}
 \sB\cong\boT_P\modl.
\end{equation}
 The equivalence of categories~\eqref{cocomplete-abelian-proj-gen-monad}
takes the projective generator $P\in\sB$ to the free $\boT_P$\+module
with one generator $\boT_P(*)$.

 Let $\sA$ be an idempotent-complete additive category with coproducts
and $M\in\sA$ be an object.
 Then $\boT_M\:\Sets\rarrow\Sets$ is an additive monad, and
$\sB=\boT_M\modl$ is a complete, cocomplete abelian category with enough
projective objects.
 The full subcategory of projective objects $\sB_\proj\subset\sB$ is
equivalent to the full subcategory $\Add(M)\subset\sA$
\cite[Theorem~1.1(a)]{PS2}, \cite[Theorem~3.13]{PS3}:
\begin{equation} \label{generalized-tilting-equivalence}
 \sB\supset\sB_\proj\cong\Add(M)\subset\sA.
\end{equation}
 The equivalence of categories~\eqref{generalized-tilting-equivalence}
takes the object $M\in\Add(M)$ to the free $\boT_M$\+module with one
generator $P=\boT_M(*)\in\boT_M\modl=\sB$, which is a projective
generator of~$\sB$.

 Assume that $\sA$ is a cocomplete additive category.
 Then the equivalence of full 
subcategories~\eqref{generalized-tilting-equivalence}
extends naturally to a pair of adjoint functors between the ambient
additive/abelian categories \cite[Section~1]{PS2}:
\begin{equation}
 \Phi_M\:\sB\rightleftarrows\sA\,:\!\Psi_M.
\end{equation}
 The right adjoint functor $\Psi_M\:\sA\rarrow\boT_M\modl$ takes
an object $N\in\sA$ to the set $\Hom_\sA(M,N)$ endowed with
the $\boT_M$\+module structure provided by the map
$$
 \pi_{\Psi_M(N)}\:\boT_M(\Hom_\sA(M,N))=
 \Hom_\sA(M,M^{(\Hom_\sA(M,N))})\rarrow\Hom_\sA(M,N)
$$
of composition with the natural morphism
$M^{(\Hom_\sA(M,N))}\rarrow N$ in the category~$\sA$
(cf.~\cite[Remark~6.4]{PS1}).
 The left adjoint functor $\Phi_M\:\sB\rarrow\sA$ can be obtained as
the extension of the fully faithful embedding $\sB_\proj\cong\Add(M)
\rarrow\sA$ to a right exact functor $\sB\rarrow\sA$.
 The restrictions of the functors $\Phi_M$ and $\Psi_M$ to the full
subcategories $\sB_\proj\subset\sB$ and $\Add(M)\subset\sA$ take these
two full subcategories into each other, providing the equivalence
$\sB_\proj\cong\Add(M)$.

\subsection{Contramodules in generalized tilting theory}
 For many additive categories $\sA$ with coproducts, the monads
$\boT_M$ associated with objects $M\in\sA$ have the form
$\boT_M\cong\boT_\R$ for certain complete, separated, right linear
topological rings~$\R$.
 In particular, this is the case for the categories $\sA=A\modl$ of
modules over associative rings~$A$.

 The first related observation is that, for every monad $\boT$ on
the category of sets, the set $\boT(*)$ assigned by the functor $\boT$
to a one-element set~$*$ has a natural monoid structure.
 In fact, the set $\boT(*)=\Hom_{\boT\modl}(\boT(*),\boT(*))$ is the set
of $\boT$\+module endomorphisms of the free $\boT$\+module with one
generator $\boT(*)$.
 We will follow the convention that the multiplication in $\boT(*)$ is
opposite to the composition of endomorphisms (so the monoid $\boT(*)$
acts in the object $\boT(*)\in\boT\modl$ on the right).

 For every additive monad $\boT$, the set $\boT(*)$ has a natural
structure of associative ring.
 In the case of the monad $\boT_M$ for an object $M\in\sA$, the related
ring $\boT_M(*)=\Hom_\sA(M,M)^\rop$ is the opposite ring to the ring of
endomorphisms of the object~$M$.
 In the case of the monad $\boT_\R$ for a topological ring $\R$,
the related ring is $\boT_\R(*)=\R$.
 Thus, given an object $M\in\sA$, in order to find a topological ring
$\R$ for which the monad $\boT_M$ is isomorphic to the monad $\boT_\R$,
one has to endow the endomorphism ring $\R=\Hom_\sA(M,M)^\rop$ with
an appropriate complete, separated right linear topology.

 Additive categories $\sA$ with set-indexed coproducts in which
the groups of morphisms $\Hom_\sA(M,N)$ carry topologies appropriate for
the task are called \emph{topologically agreeable categories}
in~\cite{PS3}.
 In fact, it often happens that a given category $\sA$ can be endowed
with several topologically agreeable structures, differing slightly
from one another.
  
\begin{exs} \label{module-endomorphism-topologies}
 (1)~Let $A$ be an associative ring and $\sA=A\modl$ be the category of
left $A$\+modules.
 Then, for $M$, $N\in\sA$, the abelian group $\Hom_A(M,N)$ can be
endowed with what is known as the \emph{finite topology}, in which
annihilators of finite subsets (or equivalently, of finitely generated
submodules) $E\subset M$ form a base of neighborhoods of zero
in $\Hom_A(M,N)$.

 The finite topology on $\Hom_A(M,N)$ is complete and separated; and
the ring $\Hom_A(M,M)$ is a \emph{left} linear topological ring in
the finite topology.
 So the ring $\R=\Hom_A(M,M)^\rop$ is a right linear topological ring.
 The monad $\boT_M\:X\longmapsto\Hom_A(M,M^{(X)})$ is isomorphic to
the monad $\boT_\R\:X\longmapsto\R[[X]]$.
 Thus the abelian category $\sB=\boT_M\modl$ is equivalent to
$\R\contra$ \cite[Theorem~7.1]{PS1}.

\smallskip
 (2)~In the setting of~(1), we say that a left $A$\+module $E$ is
\emph{weakly finitely generated} if, for any family of left $A$\+modules
$(N_x)_{x\in X}$, the natural map $\bigoplus_x\Hom_A(E,N_x)\rarrow
\Hom_A(E,\>\bigoplus_xN_x)$ is an isomorphism.
 Equivalently, this means that every $A$\+module morphism $E\rarrow
\bigoplus_xN_x$ factorizes through the direct sum of the modules $N_x$
over a finite subset of indices $x\in Z\subset X$, \,$|Z|<\infty$.
 Such modules $E$ are known in the literature as ``dually slender''
or ``small''.

 For any left $A$\+modules $M$ and $N$, the \emph{weakly finite topology}
on the abelian group $\Hom_A(M,N)$ has a base of neighborhoods of zero
consisting of the annihilators of weakly finitely generated submodules
$E\subset M$.
 The weakly finite topology on $\Hom_A(M,N)$ is complete and separated,
and once again, the ring $\Hom_A(M,M)$ is a left linear topological ring
in the weakly finite topology.
 Denoting by $\R'$ the ring $\Hom_A(M,M)^\rop$ with the weakly finite
topology on it, we once again obtain a complete, separated right linear
topological ring.
 The monad $\boT_{\R'}\:X\longmapsto\R'[[X]]$ is still isomorphic
to the monad $\boT_M\:X\longmapsto\Hom_A(M,M^{(X)})$
\cite[Theorem~9.9]{PS1}.

 In fact, while the finite topology and the weakly finite topology on
the endomorphism ring of a module may well differ, the sets
$\R[[X]]$ and $\R'[[X]]$ are the same for any set $X$, as a family of
$A$\+module morphisms $r_x\:M\rarrow M$ converges to zero in the finite
topology if and only if it converges to zero in the weakly finite one,
and if and only if the morphism $r\:M\rarrow M^X$ with
the components~$r_x$ factorizes though the submodule
$M^{(X)}\subset M^X$.
 Thus the abelian category $\sB=\boT_M\modl$ can be alternatively
described as the category $\R'\contra$.
 (Cf.~\cite[Examples~3.10]{PS3}.)
 
\smallskip
 (3)~A left $A$\+module $M$ is said to be \emph{self-small} if any
$A$\+module morphism $M\rarrow M^{(X)}$ factorizes though the coproduct
$M^{(Z)}\subset M^{(X)}$ of copies of $M$ indexed over a finite subset
$Z\subset X$.
 Equivalently, $M$ is self-small if and only if the natural map
of abelian groups $\bigoplus_{i=0}^\infty\Hom_A(M,M)\rarrow
\Hom_A(M,\>\bigoplus_{i=0}^\infty M)$ is an isomorphism.

 For a self-small left $A$\+module $M$, the ring $R=\Hom_A(M,M)^\rop$
endowed with the \emph{discrete} topology has the property that
the monad $\boT_M\:X\longmapsto\Hom_A(M,M^{(X)})$ is isomorphic to
the monad $\boT_R\:X\longmapsto R[X]$.
 So the abelian category $\sB$ is equivalent to the category of left
$R$\+modules, $\sB\cong R\modl$.

\smallskip
 (4)~Let $N$ be a fixed left $A$\+module.
 We say that a left $A$\+module $E$ is \emph{$N$\+small} if the natural
map of abelian groups $\bigoplus_{i=0}^\infty\Hom_A(E,N)\rarrow
\Hom_A(E,\>\bigoplus_{i=0}^\infty N)$ is an isomorphism.
 Equivalently, $E$ is $N$\+small if and only if, for any set $X$, any
$A$\+module morphism $E\rarrow N^{(X)}$ factorizes through
the subcoproduct $N^{(Z)}\subset N^{(X)}$ indexed over a finite subset
$Z\subset X$.
 An $A$\+module $E$ is weakly finitely generated or ``small''
(as defined in~(2)) if and only if it is $N$\+small for all left
$A$\+modules~$N$.
 It is important for the present example that the class of all
$N$\+small $A$\+modules is closed under quotients (while, e.~g.,
the class of all self-small $A$\+modules is \emph{not},
generally speaking).
 In addition, the class of $N$\+small $A$\+modules is closed under
extensions.

 For any left $A$\+modules $M$ and $N$, the \emph{$N$\+small topology}
on the abelian group $\Hom_A(M,N)$ has a base of neighborhoods of zero
consisting of the annihilators of $N$\+small submodules $E\subset M$.
 The $N$\+small topology on $\Hom_A(M,N)$ is complete and separated.
 The ring $\Hom_A(M,M)$ is a left linear topological ring in
the $M$\+small topology.
 Denoting by $\R''$ the ring $\Hom_A(M,M)^\rop$ with the $M$\+small
topology, we obtain yet another complete, separated right linear
topological ring structure for which the monad
$\boT_{\R''}\:X\longmapsto\R''[[X]]$ coincides with the monads
$\boT_\R$ and $\boT_{\R'}$ from (1) and~(2).
 Consequently, the monad $\boT_{\R''}$ is also isomorphic to
the monad~$\boT_M$, and the abelian category $\sB$ can be described
as the category $\R''\contra$.
\end{exs}

\begin{exs} \label{further-top-agreeable-examples}
 Further examples of topologically agreeable additive/abelian categories
include:

 (1)~all the locally finitely generated abelian categories $\sA$ (in
particular, all the locally finitely presentable abelian categories),
endowed with the finite topology~\cite[Example~3.7\,(2)]{PS3};

 (2)~all the locally weakly finitely generated abelian categories $\sA$,
endowed with the weakly finite topology~\cite[Section~9.2]{PS1};

 (3)~all the additive categories $\sA$ with set-indexed coproducts
admitting a \emph{closed functor} $F\:\sA\rarrow\sC$ into a locally
weakly finitely generated abelian category~$\sC$
\cite[Section~9.3]{PS1}, or more generally, into a topologically
agreeable additive category~$\sC$ \cite[Example~3.9\,(3)]{PS3}.
 In particular, the additive/abelian categories of comodules over
corings and semimodules over semialgebras belong to the class~(3)
\cite[Section~10.3]{PS1}.

 So, for any object $M$ in an additive category $\sA$ satisfying~(1),
(2), or~(3), the monad $\boT_M\:X\longmapsto\Hom_\sA(M,M^{(X)})$ is
isomorphic to the monad $\boT_\R\:X\longmapsto\R[[X]]$ for a certain
complete, separated right linear topology on the ring
$\R=\Hom_\sA(M,M)^\rop$.
 The abelian category $\sB=\boT_M\modl$ is equivalent to $\R\contra$.
\end{exs}

\subsection{Accessible monads and locally presentable categories}
 We refer to the book~\cite{AR} for the definitions and general
discussion of accessible and locally presentable categories, and
only recall here that a category is called \emph{locally presentable}
if it is accessible and cocomplete~\cite[Corollary~2.47]{AR}.
 A monad $\boT\:\Sets\rarrow\Sets$ is said to be \emph{accessible} if
its underlying functor $\boT$ is accessible, i.~e., there exists
a cardinal~$\kappa$ such that $\boT$ preserves $\kappa$\+directed
colimits.

 The category $\boT\modl$ is locally presentable if and only if
the monad $\boT$ is accessible.
 For any accessible category $\sA$ with coproducts and an object
$M\in\sA$, the monad $\boT_M$ is accessible (so the category
$\sB=\boT_M\modl$ is locally presentable).
 For any complete, separated right linear topological ring $\R$,
the monad $\boT_\R$ is accessible and the category $\R\contra$ is
locally presentable.
 We refer to the paper~\cite[Introduction and Section~5]{PR} for
the details.

\Section{Seven Classes of Topological Rings}
\label{seven-conditions-secn}

 Let $\fA$ be a complete, separated linear topological group.
 A closed subgroup $\fK\subset\fA$ is said to be \emph{strongly closed}
if the quotient group $\fA/\fK$ is complete in the quotient topology
and, for every set $X$, the induced map of sets/abelian groups
$\fA[[X]]\rarrow(\fA/\fK)[[X]]$ is surjective.
 We refer to~\cite[Sections~1.11\+-12]{Pproperf} for a discussion of
strongly closed subgroups in topological groups and strongly closed
ideals in topological rings.

 Let $R$ be a separated topological ring.
 A subset $K\subset R$ is said to be \emph{topologically left
T\+nilpotent} if, for every sequence of elements $a_1$, $a_2$,
$a_3$,~\dots~$\in K$, the sequence of products $a_1$, $a_1a_2$,~\dots,
$a_1a_2\dotsm a_n$,~\dots\ converges to zero in the topology of~$R$.
 We refer to~\cite[Section~7]{PS3} and~\cite[Section~5]{Pproperf} for
a discussion of topologically left T\+nilpotent subsets and
topologically left T\+nilpotent ideals in right linear
topological rings.

 Let $\R$ be a complete, separated right linear topological ring.
 The following four classes of such topological rings~$\R$ are
 considered in~\cite[Sections~10 and~12]{Pproperf}:
\begin{enumerate}
\renewcommand{\theenumi}{\alph{enumi}}
\item the ring $\R$ is commutative; or
\item $\R$ has a countable base of neighborhoods of zero consisting
of open two-sided ideals; or
\item $\R$ is a two-sided linear topological ring having only a finite
number of classically semisimple (semisimple Artinian) discrete
quotient rings; or
\item there is a topologically left T\+nilpotent strongly closed
two-sided ideal $\fK\subset\R$ such that the quotient ring $\R/\fK$
is isomorphic, as a topological ring, to the product
$\prod_{\delta\in\Delta}\T_\delta$ of a family of two-sided linear
topological rings $\T_\delta$, each of which satisfies one of
the conditions~(a), (b), or~(c). 
\end{enumerate} 
 Note that all the topological rings satisfying~(a), (b), or~(c) must
be two-sided linear, while a topological ring satisfying~(d) can well
be only right linear.

 Furthermore, our discussion of the following two classes of (right
linear) topological rings $\R$ is based on the results of
the paper~\cite[Sections~12 and~13]{PS3}:
\begin{enumerate}
\renewcommand{\theenumi}{\alph{enumi}}
\setcounter{enumi}{4}
\item $\R$ has a countable base of neighborhoods of zero; or
\item the abelian category $\discr\R$ is locally coherent.
\end{enumerate}
 We refer to the papers~\cite{Ro,PS3} for the definition of a locally
coherent abelian category.
 Note that (d) and~(e) are two different generalizations of~(b), while
(d)~is also a common generalization of~(a), (b), and~(c).

 Finally, we consider the following common generalization of all
the previous six conditions~(a\+-f):
\begin{enumerate}
\renewcommand{\theenumi}{\alph{enumi}}
\setcounter{enumi}{6}
\item there is a topologically left T\+nilpotent strongly closed
two-sided ideal $\fK\subset\R$ such that the quotient ring $\R/\fK$
is isomorphic, as a topological ring, to the product
$\prod_{\delta\in\Delta}\T_\delta$ of a family of right linear
topological rings $\T_\delta$, each of which satisfies one of
the conditions~(a), (c), (e), or~(f). 
\end{enumerate}

\begin{lem} \label{class-g-closure-properties}
 \textup{(i)} Let $(\R_\gamma)_{\gamma\in\Gamma}$ be a family of
topological rings each of which satisfies one of the conditions~(a),
(b), (c), (d), (e), (f), or~(g).
 Then the topological ring\/ $\R=\prod_{\gamma\in\Gamma}\R_\gamma$
satisfies~(g). \par
\textup{(ii)} Let\/ $\R$ be a complete, separated right linear
topological ring, and let\/ $\J\subset\R$ be a topologically
left T\+nilpotent strongly closed two-sided ideal.
 Assume that the topological quotient ring\/ $\R/\J$ satisfies~(g).
 Then the topological ring\/ $\R$ satisfies~(g).
\end{lem}

\begin{proof}
 Part~(i) is similar to~\cite[Lemma~12.6(a)]{Pproperf}.
 Condition~(b) is a particular case of~(e), and therefore (d)~is
a particular case of~(g).
 Conditions~(a), (c), (e), and~(f) are also particular cases of~(g).
 Hence without loss of generality we can assume that $\R_\gamma$
satisfies~(g) for every $\gamma\in\Gamma$.

 Let $\fK_\gamma\subset\R_\gamma$ be the related topologically left
T\+nilpotent strongly closed two-sided ideal.
 Then, in view of the discussion in~\cite[beginning of
Section~7]{Pproperf}, \,$\fK=\prod_\gamma\fK_\gamma$ is
a topologically left T\+nilpotent strongly closed two-sided ideal in
the topological ring $\R=\prod_\gamma\R_\gamma$, and the topological
quotient ring $\R/\fK\cong\prod_\gamma\R_\gamma/\fK_\gamma$ is
isomorphic to the topological product of topological rings, each of
which satisfies one of the conditions~(a), (c), (e), or~(f). 

 Part~(ii) is similar to~\cite[Lemma~12.6(b)]{Pproperf}.
 Let $\fK\subset\R/\J$ be a two-sided ideal witnessing that
the topological ring $\R/\J$ satisfies~(g), and let $\fH\subset\R$
be the full preimage of $\fK$ under the topological ring
homomorphism $\R\rarrow\R/\J$.
 Then the ideal $\fH$ is strongly closed in $\R$
by~\cite[Lemma~1.4(b)]{Pproperf} and topologically left T\+nilpotent
by~\cite[Lemma~5.3]{Pproperf}.
 In view of the natural isomorphism of topological rings
$\R/\fH\cong(\R/\J)/\fK$, the ideal $\fH\subset\R$ witnesses that
the topological ring $\R$ satisfies~(g).
\end{proof}
 
 The following definition was given in the paper~\cite[Section~10]{PS3}.
 A complete, separated right linear topological ring $\R$ is called
\emph{topologically left perfect} if there is a topologically left
T\+nilpotent strongly closed two-sided ideal\/ $\fH\subset\R$ such that
the quotient ring\/ $\R/\fH$ is isomorphic, as a topological ring,
to the product\/
$\S=\prod_{\gamma\in\Gamma}\Hom_{D_\gamma}(D_\gamma^{(\Upsilon_\gamma)},
D_\gamma^{(\Upsilon_\gamma)})^\rop$ of the endomorphism rings of vector
spaces over skew-fields (division rings)~$D_\gamma$.
 Here $\Gamma$ is a set, $\Upsilon_\gamma$~are nonempty sets,
the endomorphism ring of the vector space $D_\gamma^{(\Upsilon_\gamma)}$
is endowed with the finite topology, and the product of such
endomorphism rings is endowed with the product topology.
 Right linear topological rings $\S$ of the above form are called
\emph{topologically semisimple}~\cite[Section~6]{PS3}.

\begin{lem} \label{top-perfect-implies-g}
 All topologically left perfect topological rings\/ $\R$ satisfy
condition~(g).
\end{lem}

\begin{proof}
 The assertion holds because all topologically semisimple right linear
topological rings $\S$ satisfy condition~(f).
 Indeed, $\S$ is topologically semisimple if and only if the category
of discrete right $\S$\+modules $\discr\S$ is
semisimple~\cite[Theorem~6.2\,(2)]{PS3}.
 Any semisimple Grothendieck abelian category is locally Noetherian
(with simple objects forming a set of Noetherian generators); hence
it is locally coherent.
\end{proof}

\begin{lem} \label{topologically-perfect-quotient-by-T-nilpotent}
 Let\/ $\R$ be a complete, separated right linear topological ring,
and let\/ $\fK\subset\R$ be a topologically left T\+nilpotent
strongly closed two-sided ideal.
 Then the quotient ring\/ $\R/\fK$, endowed with the quotient topology,
is topologically left perfect if and only if the topological ring\/
$\R$~is.
\end{lem}

\begin{proof}
 ``If'': assume that $\R$ is topologically left perfect, and let
$\fH\subset\R$ be the related two-sided ideal, as per the definition.
 Then $\fH$ is the (topological) Jacobson radical of~$\R$\,
\cite[Lemma~10.3]{PS3}, and any topologically left T\+nilpotent
ideal in $\R$ is contained in $\fH$\, \cite[Lemma~6.6(a)]{Pproperf}.
 Hence we have $\fK\subset\fH$.
 The two-sided ideal $\fH/\fK\subset\R/\fK$ is topologically left
T\+nilpotent, since the ideal $\fH\subset\R$ is.
 By~\cite[Lemma~1.4(c)]{PS3}, \,$\fH/\fK$ is strongly closed in
$\R/\fK$.
 Finally, we have an isomorphism of topological rings
$(\R/\fK)/(\fH/\fK)\cong\R/\fH$, and the topological ring $\R/\fH$
is topologically semisimple by assumption.
 Hence the topological ring $\R/\fK$ is topologically left perfect.

 ``Only if'': assuming that $\R/\fK$ is topologically left perfect,
an argument based on~\cite[Lemmas~1.4(b) and~5.3]{PS3} and similar
to the proof of Lemma~\ref{class-g-closure-properties}(ii) shows
that $\R$ is topologically left perfect as well.
\end{proof}

\begin{lem} \label{topologically-perfect-closed-under-products}
 The class of all topologically left perfect topological rings is
closed under (infinite) topological products.
\end{lem}

\begin{proof}
 The argument is based on the discussion in~\cite[beginning of
Section~7]{Pproperf} and similar to the proof of
Lemma~\ref{class-g-closure-properties}(i).
\end{proof}

\begin{thm} \label{g-implies-as}
 Let\/ $\R$ be a complete, separated right linear topological ring.
 Then the topological ring\/ $\R$ is topologically left perfect if and
only if it satisfies one of the conditions~(a), (b), (c), (d), (e), (f),
or~(g) \emph{and} every descending chain of cyclic discrete right\/
$\R$\+modules terminates.
\end{thm}

\begin{proof}
 ``Only if'': for any topologically left perfect topological ring $\R$,
any descending chain of cyclic discrete right $\R$\+modules
terminates by~\cite[Theorem~14.4\,(iv)\,$\Rightarrow$\,(v)]{PS3},
and condition~(g) is satisfied by Lemma~\ref{top-perfect-implies-g}.

 ``If'': cases~(a\+-c) are covered
by~\cite[Theorem~10.1\,(v)\,$\Rightarrow$\,(iv)]{Pproperf},
and case~(d) is~\cite[Theorem~12.4\,(v)\,$\Rightarrow$\,(iv)]{Pproperf}
(see also~\cite[Remark~14.6 and Corollary~14.7]{PS3}).
 Case~(e) is~\cite[Theorem~12.4 or~14.8]{PS3}, and case~(f)
is~\cite[Theorem~13.3 or~14.12]{PS3}.

 To prove case~(g), assume that $\fK\subset\R$ is a topologically left
T-nilpotent strongly closed two-sided ideal for which the topological
ring $\R/\fK$ is isomorphic to the topological product
$\prod_{\delta\in\Delta}\T_\delta$, where each topological ring
$\T_\delta$ satisfies one of the conditions~(a), (c), (e), or~(f).
 Then $\T_\delta$ is a topological quotient ring of the topological
ring $\R$ for every $\delta\in\Delta$.
 Hence $\discr\T_\delta$ is the full subcategory in $\discr\R$
consisting of all the modules annihilated by the kernel ideal of
the surjective continuous ring homomorphism $\R\rarrow\T_\delta$.
 Since every descending chain of cyclic discrete right $\R$\+modules
terminates, so does every descending chain of cyclic discrete
right $\T_\delta$\+modules.
 According to the previous paragraph, in each of the cases~(a), (c),
(e), or~(f) it follows that $\T_\delta$ is a topologically left
perfect topological ring.
 Using Lemmas~\ref{topologically-perfect-quotient-by-T-nilpotent}
and~\ref{topologically-perfect-closed-under-products}, we can conclude
that $\R$ is a topologically left perfect topological ring.
\end{proof}

\Section{The Enochs Conjecture} \label{enochs-conjecture-secn}

 Throughout this paper, by ``direct limits'' in a category we mean
inductive limits indexed by directed posets.
 Otherwise, these are known as the directed or filtered colimits.
 For any class of objects $\sM$ in a cocomplete category $\sA$,
we denote by $\varinjlim\sM=\varinjlim^\sA\sM\subset\sA$ the class of
all direct limits of objects from $\sM$ in $\sA$.
 This means the direct limits of diagrams $A\:\Theta\rarrow\sA$ indexed
by directed posets $\Theta$ and such that $A(\theta)\in\sM$ for
all $\theta\in\Theta$.

 Let $\sA$ be a category and $\sL\subset\sA$ be a class of objects.
 A morphism $l\:L\rarrow C$ in $\sA$ is called an \emph{$\sL$\+precover}
(of the object~$C$) if $L\in\sL$ and all the morphisms from objects of
$\sL$ to the object $C$ factorize through the morphism~$l$ in
the category~$\sA$, that is, for every morphism $l'\:L'\rarrow C$ with
$L'\in\sL$ there exists a morphism $f\:L'\rarrow L$ such that $l'=lf$.
 A morphism $l\:L\rarrow C$ in $\sA$ is called an \emph{$\sL$\+cover} if
it is an $\sL$\+precover and, for any endomorphism $e\:L\rarrow L$,
the equation $le=l$ implies that $e$ is an automorphism of~$L$.
 We will say that a class of objects $\sL$ in a category $\sA$ is
\emph{precovering} if every object of $\sA$ has an $\sL$\+precover.
 Similarly, the class $\sL$ is said to be \emph{covering} if every
object of $\sA$ has an $\sL$\+cover.

 Given another class of objects $\sE\subset\sA$, the definitions of
an \emph{$\sE$\+preeenvelope} and an \emph{$\sE$\+envelope} of an object
$C\in\sA$ are dual to the above definitions of an $\sL$\+precover
and an $\sL$\+cover.
 These notions are due to Enochs~\cite{Eno}; a detailed discussion
of their properties in a relevant context can be found in
the book~\cite{Xu}.

\begin{ex} \label{Add(M)-precovers}
 If $\sA$ is an additive category with coproducts and
$M\in\sA$ is an object, then the class of objects $\Add(M)\subset\sA$
is precovering.
 Indeed, for any object $N\in\sA$, the obvious morphism
$M^{(\Hom_\sA(M,N))}\rarrow N$ is an $\Add(M)$\+precover of~$N$.
\end{ex}

\begin{ex} \label{projective-precovers}
 Let $\sB$ be an abelian category with enough projective objects and
$\sL=\sB_\proj\subset\sB$ be the class of all projective objects.
 Then a morphism $L\rarrow C$ in $\sB$ with $L\in\sL$ is
an $\sL$\+precover if and only if it is an epimorphism.
 So the class of all projective objects in an abelian category with
enough projective objects is always precovering; but it is rarely
covering, as we will see.
 A $\sB_\proj$\+cover in $\sB$ is called a \emph{projective cover}.
\end{ex}

 The first assertion of the following theorem is one of the main results
of Bass' paper~\cite{Bas}.
 In fact, it is a part of the famous~\cite[Theorem~P]{Bas}.

\begin{thm}
 Let $\sB=R\modl$ be the category of modules over an associative ring,
and let\/ $\sL=R\modl_\proj\subset R\modl$ be the class of projective
left $R$\+modules.
 Then the class\/ $\sL$ is covering in $R\modl$ if and only if\/ $\sL$
is closed under direct limits in $R\modl$.

 Moreover, if every countable direct limit of copies of the free left
$R$\+module $R$ has a projective cover in $R\modl$, then all flat left
$R$\+modules are projective and all left $R$\+modules have
projective covers.
\end{thm}

\begin{proof}
 The first assertion
is~\cite[Theorem~P\,(2)\,$\Leftrightarrow$\,(5)]{Bas}.
 The second assertion stems from the proof of
the implication~\cite[Theorem~P\,(5)\,$\Rightarrow$\,(6)]{Bas}, which
only uses projectivity of the countable direct limits of copies of
the $R$\+module~$R$.
 Such direct limits are now known as \emph{Bass flat $R$\+modules}.
 Associative rings $R$ satisfying the equivalent conditions
of~\cite[Theorem~P]{Bas} are called \emph{left perfect}.
 So it is shown in~\cite{Bas} that a ring $R$ is left perfect whenever
all Bass flat left $R$\+modules are projective.

 A proof of the assertion that any flat module having a projective
cover is projective can be found in~\cite[Section~36.3]{Wis}.
\end{proof}

 The idea of the proof of the following result goes back to Enochs'
paper~\cite[Theorems~2.1 and~3.1]{Eno}.

\begin{thm} \label{direct-limits-imply-covers}
 In a locally presentable category\/ $\sA$, any precovering class
closed under direct limits is covering.
\end{thm}

\begin{proof}
 For module categories, this was established by Enochs in~\cite{Eno}.
 For Grothendieck abelian categories, a proof of this assertion can be
found in~\cite[Theorem~1.2]{Bash}; and for locally presentable
categories, in~\cite[Theorem~2.7 or Corollary~4.17]{PR}. \hbadness=1300
\end{proof}

 It is easy to prove that, in any category $\sA$, any covering class
$\sL\subset\sA$ is closed under retracts, and any precovering class that
is closed under retracts is also closed under coproducts
(cf.~\cite[Proposition~2.1]{Eno} or~\cite[Theorem~2.5.1]{Xu}).
 Hence any covering class is closed under coproducts.
 The following inverse assertion to
Theorem~\ref{direct-limits-imply-covers} (for module categories)
is known as ``the Enochs conjecture''
(see~\cite[Section~5.4]{GT}; cf.~\cite[Section~5]{AST}).

\begin{conj}
 Let\/ $\sA=A\modl$ be the category of modules over an associative ring
$A$, and let\/ $\sL\subset A\modl$ be a covering class.
 Then\/ $\sL$ is closed under direct limits in $A\modl$.
\end{conj}

 Far-reaching results confirming particular cases of the Enochs
conjecture were obtained in the paper~\cite{AST}, based on
the tools developed in~\cite{Sar}.
 (See also the preprint~\cite{BPS} for an alternative elementary proof
of some of the results of~\cite{AST}.)
 The idea of our categorical approach to the Enochs conjecture is
expressed in the following conjectural extension of Bass' theorem.

\begin{mc} \label{main-conjecture}
 Let\/ $\sB$ be a locally presentable abelian category with a projective
generator~$P$.
 Then the following conditions are equivalent.
\begin{enumerate}
\item the class\/ $\sB_\proj$ is covering in\/~$\sB$;
\item any direct limit of projective objects has a projective
cover in\/~$\sB$;
\item any countable direct limit of copies of $P$ has a projective
cover in\/~$\sB$;
\item any countable direct limit of copies of $P$ is a projective
object in\/~$\sB$;
\item the class\/ $\sB_\proj$ is closed under direct limits in\/~$\sB$.
\end{enumerate}
\end{mc}

 Notice that the implications
(1)\,$\Longrightarrow$\,(2)\,$\Longrightarrow$\,(3) and
(5)\,$\Longrightarrow$\,(4)\,$\Longrightarrow$\,(3) in
the Main Conjecture are obvious, while the implication
(5)\,$\Longrightarrow$\,(1) holds by Example~\ref{projective-precovers}
and Theorem~\ref{direct-limits-imply-covers}.
 The implications (3)\,$\Longrightarrow$\,(2)\,$\Longrightarrow$\,(1)%
\,$\Longrightarrow$\,(5)
and (3)\,$\Longrightarrow$\,(4)\,$\Longrightarrow$\,(5) are nontrivial
(and unknown).

 For the categories of contramodules over topological rings, some of
the equivalences in Conjecture~\ref{main-conjecture} are provided by
the results of the paper~\cite{PS3}.

\begin{thm} \label{main-conjecture-for-contramodules}
 Let\/ $\R$ be a complete, separated right linear topological ring.
 Then the following equivalences of conditions in
Main Conjecture~\ref{main-conjecture} hold for the abelian category\/
$\sB=\R\contra$ with the projective generator $P=\R[[*]]]=\R$:
$$
 \mathrm{(1)} \Longleftrightarrow \mathrm{(2)}
 \Longleftrightarrow \mathrm{(5)} \quad\text{and}\quad
 \mathrm{(3)} \Longleftrightarrow \mathrm{(4)}.
$$
\end{thm}

\begin{proof}
 The equivalences
(1)\,$\Longleftrightarrow$\,(2)\,$\Longleftrightarrow$\,(5) are
\cite[Theorem~14.1\,(ii)\,$\Leftrightarrow$\,(i$'$)\,%
$\Leftrightarrow$\,(iii$'$)]{PS3}.
 The equivalence (3)\,$\Longleftrightarrow$\,(4) is
\cite[Theorem~14.4\,(i$^\flat$)\,$\Leftrightarrow$\,%
(iii$^\flat$)]{PS3}.
\end{proof}

 A more refined version of Main Conjecture~\ref{main-conjecture} in
the particular case of contramodules over topological rings can be found
in~\cite[Conjecture~14.3]{PS3}.

 The following special cases of Main Conjecture~\ref{main-conjecture}
for the categories of contramodules over topological rings are provable
with our methods.

\begin{thm} \label{a-b-c-d-imply-main-conjecture}
 Let\/ $\R$ be a complete, separated right linear topological ring
satisfying one of the conditions~(a), (b), (c), (d), (e), (f), or~(g).
 Then Main Conjecture~\ref{main-conjecture} holds for the abelian
category\/ $\R\contra$ with the projective generator\/ $\R$, i.~e.,
the conditions~(1), (2), (3), (4), and~(5)
are equivalent for\/ $\sB=\R\contra$ and $P=\R$.
\end{thm}

\begin{proof}
 Follows from Theorem~\ref{g-implies-as} and~\cite[Theorem~14.4]{PS3}.
\end{proof}

 The next lemma, generalizing the ``if'' assertion of
Lemma~\ref{topologically-perfect-quotient-by-T-nilpotent},
is an application of projective covers to topological algebra.

\begin{lem} \label{topologically-perfect-closed-under-quotients}
 The class of topologically left perfect complete, separated right
linear topological rings is closed under the passages to topological
quotient rings by strongly closed two-sided ideals.
\end{lem}

\begin{proof}
 We use the following characterization: a right linear topological
ring $\R$ is topologically left perfect if and only if all left
$\R$\+contramodules have projective covers~\cite[Theorem~14.1%
\,(ii)\,$\Leftrightarrow$\,(iv)]{PS3}.
 Assume that $\R$ is topologically left perfect, $\J\subset\R$ is
a strongly closed two-sided ideal, and $\T=\R/\J$ is the topological
quotient ring.
 Let us show that every left $\T$\+contramodule $\C$ has a projective
cover.
 Using the contrarestriction of scalars~\cite[Section~1.9]{Pproperf},
one can consider $\C$ as a left $\R$\+contramodule.
 As such, $\C$ has a projective cover $p\:\fP\rarrow\C$ in
$\R\contra$.
 Then the reduction construction of~\cite[Lemma~3.3]{Pproperf} produces
a projective cover of $\C$ in $\T\contra$.
\end{proof}

\Section{Covers Reduced to Projective Covers}

 Let $\sA$ be an additive category and $f\:A\rarrow B$ be a morphism
in~$\sA$.
 A morphism $k\:K\rarrow A$ is said to be a \emph{weak kernel} of~$f$
if $fk=0$ and for any object $C\in\sA$ and any morphism $c\:C\rarrow A$
such that $fc=0$ there exists a (not necessarily unique) morphism
$h\:C\rarrow K$ such that $c=kh$.
 A morphism~$k$ is a kernel of~$f$ if and only if $k$~is a weak kernel
of~$f$ and $k$~is a monomorphism.

 Let $\sL\subset\sA$ be a precovering class of objects.
 We are interested in conditions under which $\sL$ is a covering class.
 First of all, if $\sL$ is covering, then $\sL$ is closed under
direct summands in~$\sA$.
 If $\sL$ is precovering and closed under direct summands, then $\sL$
is closed under coproducts (see the discussion in the previous section).
 In particular, the full subcategory $\sL\subset\sA$ is additive.

\begin{lem} \label{precovers-weak-kernels}
 Let\/ $\sA$ be an additive category with weak kernels and\/ $\sL\subset
\sA$ be an additive full subcategory.
 Assume that the class of objects\/ $\sL$ is precovering in\/~$\sA$.
 Then the category\/ $\sL$ also has weak kernels.
\end{lem}

\begin{proof}
 Let $f\:L\rarrow M$ be a morphism in $\sL$ and $a\:A\rarrow L$ be
a weak kernel of~$f$ in~$\sA$.
 Let $p\:K\rarrow A$ be an $\sL$\+precover of the object $A\in\sA$.
 Then the composition $k=ap\:K\rarrow L$ is a weak kernel of~$f$
in~$\sL$.
\end{proof}

\begin{lem} \label{freyd}
 Let\/ $\sA$ be an idempotent-complete additive category with weak
kernels and\/ $\sL\subset\sA$ be an additive full subcategory closed
under direct summands.
 Assume that the class of objects\/ $\sL$ is precovering in\/~$\sA$.
 Then there exists a unique abelian category\/ $\sB$ with enough
projectives such that the full subcategory of projective objects\/
$\sB_\proj\subset\sB$ is equivalent to the full subcategory\/
$\sL\subset\sA$.
\end{lem}

\begin{proof}
 By Lemma~\ref{precovers-weak-kernels}, the category $\sL$ has weak
kernels.
 Hence the category $\sB$ can be constructed as the category of
finitely presented (or ``coherent'') functors $\sL^\sop\rarrow\Ab$
\,\cite[Corollary~1.5]{Fr}, \cite[Lemma~2.2 and Proposition~2.3]{Kra}
(see also~\cite[proof of Theorem~1.1(a)]{PS2} for a discussion with
further references).
\end{proof}

\begin{prop} \label{phi-psi-for-precovering-class}
 Let\/ $\sA$ be an additive category with cokernels and weak kernels
and\/ $\sL\subset\sA$ be an additive full subcategory closed under
direct summands.
 Assume that the class of objects\/ $\sL$ is precovering in\/~$\sA$.
 Let\/ $\sB$ be the abelian category from Lemma~\ref{freyd}.
 Then the equivalence of full subcategories\/ $\sB\supset\sB_\proj
\cong\sL\subset\sA$ can be extended, in a unique way, to a pair of
adjoint functors
$$
 \Phi_\sL\:\sB\rightleftarrows\sA\,:\!\Psi_\sL,
$$
where the functor\/ $\Phi_\sL$ is the left adjoint and the functor\/
$\Psi_\sL$ is the right adjoint.
\end{prop}

\begin{proof}
 The inclusion functor $\sB_\proj\cong\sL\rarrow\sA$ extends uniquely
to a right exact functor $\Phi_\sL\:\sB\rarrow\sA$.
 This suffices to prove uniqueness of the desired adjoint pair.

 To construct the functor $\Psi_\sL$, we assign to every object
$N\in\sA$ the functor $\Hom_\sA({-},N)|_\sL\:\sL^\sop\rarrow\Ab$.
 Let us check that the functor $\Hom_\sA({-},N)|_\sL$ is finitely
presented.
 Choose an $\sL$\+precover $l\:L\rarrow N$ of the object $N$.
 Let $a\:A\rarrow L$ be a weak kernel of the morphism~$l$ in
the category $\sA$, and let $p\:K\rarrow A$ be an $\sL$\+precover
of the object $A\in\sA$.
 Consider the morphism $m=pa\:K\rarrow L$ in the category~$\sL$.
 Then the functor $\Hom_\sA({-},N)|_\sL$ is the cokernel of the morphism
of representable functors $\Hom_\sL({-},m)\:\Hom_\sL({-},K)\rarrow
\Hom_\sL({-},L)$.

 So the functor $\Psi_\sL\:\sA\rarrow\sB$ assigning to an object $N$
the functor $\Hom_\sA({-},N)|_\sL$ is well-defined.
 By the Yoneda lemma, a natural isomorphism $\Hom_\sB(\Psi_\sL(M),
\Psi_\sL(N))\cong\Hom_\sA(M,N)$ holds for all objects $M\in\sL$ and
$N\in\sA$.
 Hence we have an adjunction isomorphism $\Hom_\sB(P,\Psi_\sL(N))
\cong\Hom_\sA(\Phi_\sL(P),N)$ for all objects $P\in\sB_\proj$
and $N\in\sA$.
 The latter isomorphism extends uniquely to a functorial isomorphism
$\Hom_\sB(B,\Psi_\sL(N))\cong\Hom_\sA(\Phi_\sL(B),N)$ for all objects
$B\in\sB$ and $N\in\sA$ by right exactness of the functor~$\Phi_\sL$.
\end{proof}

\begin{prop} \label{cover-adjunction}
 Let\/ $\sA$ and\/ $\sB$ be two categories, and let\/
$\Phi\:\sB\rightleftarrows\sA\,:\!\Psi$ be a pair of adjoint functors,
where\/ $\Psi$ is the right adjoint, such that the restrictions of\/
$\Phi$ and\/ $\Psi$ are mutually inverse equivalences between a full
subcategory\/ $\sL\subset \sA$ and a full subcategory\/ $\sP\subset\sB$.
 Then \par
\textup{(a)} a morphism $l\:L\rarrow N$ in\/ $\sA$ with $L\in\sL$ is
an\/ $\sL$\+precover if and only if the morphism $\Psi(l)\:\Psi(L)
\rarrow\Psi(N)$ is a\/ $\sP$\+precover; \par
\textup{(b)} a morphism $l\:L\rarrow N$ in\/ $\sA$ with $L\in\sL$ is
an\/ $\sL$\+cover if and only if the morphism $\Psi(l)\:\Psi(L)
\rarrow\Psi(N)$ is a\/ $\sP$\+cover; \par
\textup{(c)} an object $N\in\sA$ has an\/ $\sL$\+precover if and only if
the object\/ $\Psi(N)\in\sB$ has a\/ $\sP$\+precover; \par
\textup{(d)} an object $N\in\sA$ has an\/ $\sL$\+cover if and only if
the object\/ $\Psi(N)\in\sB$ has a\/ $\sP$\+cover.
\end{prop}

\begin{proof}
 Part~(a): given an object $P\in\sP$, the map of sets
$$
 \Hom_\sA(\Phi(P),l)\:\Hom_\sA(\Phi(P),L)\lrarrow\Hom_\sA(\Phi(P),N)
$$
is isomorphic to the map of sets
$$
 \Hom_\sB(P,\Psi(l))\:\Hom_\sB(P,\Psi(L))\lrarrow\Hom_\sB(P,\Psi(N)).
$$
 Hence former map is surjective if and only if the latter map is.
 Since one has $\Phi(P)\in\sL$ for all $P\in\sP$, and every object
$L'\in\sL$ is isomorphic to an object $\Phi(P)$ for some $P\in\sP$,
the assertion follows.

 Part~(b): given an endomorphism $e\:L\rarrow L$, one has
$le=l$ if and only if $\Psi(l)\Psi(e)=\Psi(l)$, since the map
$$
 \Hom_\sA(L,N)\,\cong\,\Hom_\sA(\Phi\Psi(L),N)
 \lrarrow\Hom_\sB(\Psi(L),\Psi(N))
$$
is bijective.
 Since the map $\Hom_\sA(L,L)\rarrow\Hom_\sB(\Psi(L),\Psi(L))$ is
bijective, too, the assertion follows in view of part~(a).

 Finally, part~(a) implies~(c), and part~(b) implies~(d), because
any morphism $p\:P\rarrow\Psi(N)$ in $\sB$ with $P\in\sP$ has the form
$p=\Psi(l)$ for a (uniquely defined) morphism $l\:L=\Phi(P)\rarrow N$
in~$\sA$.
\end{proof}

 Proposition~\ref{phi-psi-for-precovering-class} describes one situation
in which Proposition~\ref{cover-adjunction} is applicable.
 Let $\sA$ be an additive category with cokernels and (weak) kernels,
and let $\sL\subset\sA$ be a precovering class closed under direct
summands.
 Consider the related abelian category $\sB$, and put
$\sP=\sB_\proj\subset\sB$.
 Then an object $N\in\sA$ has an $\sL$\+cover if and only if the object
$\Psi_\sL(N)\in\sB$ has a projective cover.
 Hence the class $\sL\subset\sA$ is covering if and only if all objects
of the form $\Psi_\sL(N)$, \,$N\in\sA$, have projective covers in~$\sB$.

 Another such situation is described in
Section~\ref{generalized-tilting-subsecn}.
 Let $\sA$ be a cocomplete additive category and $M\in\sA$ be an object.
 Consider the related abelian category $\sB=\boT_M\modl$, and put
$\sL=\Add(M)\subset\sA$ and $\sP=\sB_\proj\subset\sB$.
 By Proposition~\ref{cover-adjunction}(d), an object $N\in\sA$ has
an $\Add(M)$\+cover if and only if the object $\Psi_M(N)\in\sB$ has
a projective cover.
 Once again, we conclude that the class $\Add(M)\subset\sA$ is covering
if and only if all objects of the form $\Psi_M(N)$, \,$N\in\sA$, have
projective covers in~$\sB$.

\begin{rem}
 In both contexts above, the existence of cokernels in the category
$\sA$ was used in order to extend the equivalence $\sB_\proj\cong\sL
\subset\sA$ to a right exact functor $\Phi\:\sB\rarrow\sA$.
 However, looking into the proof of Proposition~\ref{cover-adjunction},
one can observe that the functor $\Phi$ is never applied to any objects
outside of the full subcategory $\sP\subset\sB$.
 So one can relax the assumptions of that proposition by requiring
the functor $\Phi$ to be defined on the full subcategory $\sP\subset
\sB$ only.
 For this reason, the assumption of existence of cokernels in
the category $\sA$ can be replaced by the weaker assumption of
idempotent-completeness.
 Then, in the first of the above two settings (based on
Proposition~\ref{phi-psi-for-precovering-class}), the existence of weak
kernels in $\sA$ is sufficient; and in the second one (based on
Section~\ref{generalized-tilting-subsecn}), it suffices to assume
that $\sA$ has coproducts.
\end{rem}

\Section{Telescope Hom Exactness Condition}  \label{THEC-secn}

 In this section we introduce the most general setting in which we
can show that Main Conjecture~\ref{main-conjecture} implies some
instances of the Enochs conjecture.

\begin{dfn}
 Let $\sA$ be an additive category with countable direct limits, and
let $M\in\sA$ be an object.
 Given a sequence of endomorphisms $f_1$, $f_2$, $f_3$,~\dots\
$\in\Hom_\sA(M,M)$, we form the inductive system
$$
 M\overset{f_1}\lrarrow M\overset{f_2}\lrarrow M
 \overset{f_3}\lrarrow\dotsb
$$
and consider the related telescope sequence
\begin{equation} \label{categorical-telescope-sequence}
 \coprod\nolimits_{n=1}^\infty M\lrarrow
 \coprod\nolimits_{n=1}^\infty M\lrarrow
 \varinjlim\nolimits_{n\ge1}M\lrarrow0.
\end{equation}
 The short sequence~\eqref{categorical-telescope-sequence} is
always right exact, i.~e., the direct limit $\varinjlim_{n\ge1} M$
is the cokernel of the morphism $\id-\mathit{shift}\:
\coprod_{n=1}^\infty M\rarrow \coprod_{n=1}^\infty M$.

 We will say that the object $M\in\sA$ satisfies the \emph{telescope Hom
exactness condition} (\emph{THEC}) if, for any sequence of endomorphisms
$(f_n\in\Hom_\sA(M,M))_{n\ge1}$ of the object $M$, the short
sequence~\eqref{categorical-telescope-sequence} remains right
exact after applying the functor $\Hom_\sA(M,{-})$, that is,
the short sequence of abelian groups
\begin{equation} \label{hom-into-telescope-sequence}
\textstyle
 \Hom_\sA\left(M,\>\coprod\nolimits_{n=1}^\infty M\right)\rarrow
 \Hom_\sA\left(M,\>\coprod\nolimits_{n=1}^\infty M\right)\rarrow
 \Hom_\sA(M,\>\varinjlim\nolimits_{n\ge1}M)\rarrow0
\end{equation}
is right exact.
\end{dfn}

\begin{ex}
 Let $\sA$ be an abelian category with exact functors of countable
direct limit.
 Then the telescope sequence~\eqref{categorical-telescope-sequence}
is exact at its leftmost term, too,
$$
 0\lrarrow\coprod\nolimits_{n=1}^\infty M\lrarrow
 \coprod\nolimits_{n=1}^\infty M\lrarrow
 \varinjlim\nolimits_{n\ge1}M\lrarrow0,
$$
as it is a countable direct limit of the split exact sequences
$$
 0\lrarrow\coprod\nolimits_{i=1}^{n-1} M\lrarrow
 \coprod\nolimits_{i=1}^n M\lrarrow M\lrarrow0.
$$
 In this case, the exactness of the short seqeunce of Hom
groups~\eqref{hom-into-telescope-sequence} at the middle term is
obvious, and the telescope Hom exactness condition simply means
exactness of the sequence~\eqref{hom-into-telescope-sequence} at
the rightmost term.
 In other words, this means that any morphism
$M\rarrow\varinjlim_{n\ge1}M$ in the category $\sA$ can be lifted
to a morphism $M\rarrow\coprod_{n=1}^\infty M$.
 This is equivalent to the condition that the morphism
$\coprod_{n=1}^\infty M\rarrow\varinjlim_{n\ge1}M$
in~\eqref{categorical-telescope-sequence} is an $\Add(M)$\+precover.
\end{ex}

\begin{exs} \label{sigma-pure-rigid-example}
 (1)~Let $\sA$ be an abelian category with exact countable direct
limits.
 Then the telescope Hom exactness condition holds for any
\emph{$\Sigma$\+rigid} (or \emph{$\Sigma$\+$\Ext^1$-self-orthogonal})
object $M\in\sA$, that is, any object such that
$\Ext^1_\sA(M,M^{(\omega)})=0$.

\smallskip
 (2)~More generally, if there is a notion of \emph{purity} in
the abelian category $\sA$, then for any two objects $M$, $N\in\sA$
one can consider the group $\PExt_\sA^1(M,N)$ of equivalence classes
of pure short exact sequences $0\rarrow N\rarrow A\rarrow M\rarrow 0$.
 An object $M\in\sA$ is called \emph{$\Sigma$\+pure-rigid}
(or \emph{$\Sigma$\+pure-$\Ext^1$-self-orthogonal}) if
$\PExt^1_\sA(M,M^{(\omega)})=0$.

 For any meaningful notion of purity, one expects that split short
exact sequences should be pure exact.
 It is also reasonable to assume that the class of pure short exact
sequences in $\sA$ is closed under countable direct limits and
pullbacks, among other things.
 If this is the case, then any $\Sigma$\+pure-rigid object in $\sA$
satisfies THEC\@.
 In particular, this applies to the module categories $\sA=A\modl$
over associative rings~$A$.

 One specific notion of purity in abelian categories, called
the \emph{functor purity}, will be discussed below in
Section~\ref{functor-purity-secn}.
 It has the above-mentioned properties.
\end{exs}

\begin{ex}
 Let $\sA$ be an abelian category with exact countable direct limits
and a class of pure short exact sequences satisfying the conditions of
Example~\ref{sigma-pure-rigid-example}\,(2).
 We will say that an object $M\in\sA$ is
\emph{$\omega$\+self-pure-projective} if for any pure short exact
sequence $0\rarrow K\rarrow M^{(\omega)}\rarrow L\rarrow0$ in $\sA$
the induced morphism of abelian groups $\Hom_\sA(M,M^{(\omega)})
\rarrow\Hom_\sA(M,L)$ is surjective.
 Any $\omega$\+self-pure-projective object $M\in\sA$ satisfies
the telescope Hom exactness condition.
\end{ex}

 For the rest of this section, we are working with a fixed object $M$
in a cocomplete additive category~$\sA$.
 We consider the related abelian category $\sB=\boT_M\modl$ and
the pair of adjoint functors $\Psi\:\sA\rarrow\sB$ and $\Phi\:\sB
\rarrow\sA$, as in Section~\ref{generalized-tilting-subsecn}.

 Furthermore, we denote by $\sG\subset\sA$ the full subcategory formed
by all the objects $G\in\sA$ for which the adjunction morphism
$\Phi(\Psi(G))\rarrow G$ is an isomorphism, and by $\sH\subset\sB$
the full subcategory of all the objects $H\in\sB$ for which
the adjunction morphism $H\rarrow\Psi(\Phi(H))$ is an isomorphism.
 One has $\Psi(\sG)\subset\sH$ and $\Phi(\sH)\subset\sG$, and
the restrictions of the functors $\Psi$ and $\Phi$ to the full
subcategories $\sG$ and $\sH$ are mutually inverse equivalences
between them~\cite[Theorem~1.1]{FJ},
\begin{equation} \label{G-H-equivalence}
 \Psi|_\sG\:\sG\,\cong\,\sH\,:\!\Phi|_\sH.
\end{equation}
 By construction, we have $\Add(M)\subset\sG$ and $\sB_\proj\subset\sH$,
since $\Psi|_{\Add(M)}\:\Add(M)\rarrow\sB_\proj$ and $\Phi|_{\sB_\proj}
\:\sB_\proj\rarrow\Add(M)$ are mutually inverse equivalences.

\begin{lem} \label{direct-limit-closure-B-implies-A}
 Let\/ $\sA$ be a cocomplete additive category, $M\in\sA$ be an object,
and\/ $\sB=\boT_M\modl$ be the related abelian category.
 Suppose that the class of all projective objects in\/ $\sB$ is closed
under (arbitrary or countable) direct limits.
 Then the class of objects\/ $\Add(M)\subset\sA$ is also closed under
(arbitrary or countable, resp.)\ direct limits.
 More specifically, if every countable direct limit of copies of
the projective generator $P=T_M(*)$ is projective in\/ $\sB$, then every
countable direct limit of copies of $M$ in $\sA$ belongs to\/ $\Add(M)$.
\end{lem}

\begin{proof}
 Let $\Theta$ be a directed poset and $A\:\Theta\rarrow\sA$ be
a diagram such that the object $A(\theta)$ belongs to the class
$\Add(M)$ for all $\theta\in\Theta$.
 Applying the functor $\Psi$, we obtain a diagram $B=\Psi\circ A\:
\Theta\rarrow\sB$ such that $B(\theta)$ is a projective object in
$\sB$ for all $\theta\in\Theta$.
 Applying the functor $\Phi$ to get back to the category $\sA$, we
come to the original diagram $A\cong\Phi\circ B$.
 Now the functor $\Phi$, being a left adjoint, preserves all colimits,
so the natural morphism $\varinjlim_{\theta\in\Theta}A(\theta)\cong
\varinjlim_{\theta\in\Theta}\Phi(B(\theta))\rarrow\Phi
\bigl(\varinjlim_{\theta\in\Theta}B(\theta)\bigr)$ is an isomorphism
in~$\sA$.
 Since $\varinjlim_{\theta\in\Theta}B(\theta)$ is a projective object
in $\sB$ by assumption and $\Phi(\sB_\proj)=\Add(M)$, the desired
conclusion follows.
\end{proof}

\begin{prop} \label{THEC-Bass-objects-in-A-and-B}
 Let\/ $\sA$ be a cocomplete additive category and $M\in\sA$ be
an object satisfying THEC\@.
 Let $\sB=\boT_M\modl$ be the related cocomplete abelian category
with a projective generator $P=\boT_M(*)\in\sB$ corresponding to
the object $M\in\sA$, and let\/ $\sG\subset\sA$ and\/ $\sH\subset\sB$
be the related two full subcategories.

 Then all countable direct limits of copies of the object $M$ in\/ $\sA$
belong to the class\/ $\sG$, and all the countable direct limits of
copies of the object $P$ in\/ $\sB$ belong to the class\/~$\sH$.
 The functor\/ $\Psi$ preserves countable direct limits of copies of
the object $M\in\sA$ (taking them to countable direct limits of
copies of the object $P\in\sB$).
\end{prop}

\begin{proof}
 Let $M\overset{f_1}\rarrow M\overset{f_2}\rarrow M\overset{f_3}\rarrow
\dotsb$ be a countable inductive system of copies of the object $M$
in~$\sA$.
 Then we have the right exact
sequence~\eqref{categorical-telescope-sequence} in the category $\sA$
and the right exact sequence~\eqref{hom-into-telescope-sequence} in
the category of abelian groups.

 Now, the abelian category $\sB=\boT_M\modl$ is endowed with a faithful
exact forgetful functor $\boT_M\modl\rarrow\Ab$, and the composition
of the functor $\Psi$ with this forgetful functor is isomorphic to
the functor $\Hom_\sA(M,-)$.
 It follows that the image of
the sequence~\eqref{categorical-telescope-sequence} under
the functor $\Psi$ is right exact in~$\sB$.

 The functors $\Psi$ and $\Phi$ restrict to mutually inverse
equivalences between $\Add(M)\subset\sA$ and $\sB_\proj\subset\sB$;
so, in particular, they transform coproducts of objects from $\Add(M)$
in $\sA$ to coproducts of projective objects in $\sB$ and vice versa.
The short sequence
\begin{equation} \label{Psi(M)-telescope}
 \coprod\nolimits_{n=1}^\infty\Psi(M)\lrarrow
 \coprod\nolimits_{n=1}^\infty\Psi(M)\lrarrow
 \varinjlim\nolimits_{n\ge1}\Psi(M)\rarrow0
\end{equation}
is right exact in $\sB$; and the natural morphism from
the sequence~\eqref{Psi(M)-telescope} to the image of
the sequence~\eqref{categorical-telescope-sequence} under the functor
$\Psi$ is an isomorphism at the leftmost and the middle terms.
 Hence it is also an isomorphism at the rightmost terms, that is,
the natural morphism $\varinjlim_n\Psi(M)\rarrow\Psi(\varinjlim_n M)$
is an isomorphism.

 The functor $\Phi$, being a left adjoint, preserves all colimits.
 Since the adjunction morphism $\Phi\Psi(M)\rarrow M$ is
an isomorphism, it follows that the adjunction morphism
$\Phi\Psi(\varinjlim_n M)\rarrow\varinjlim_n M$ is
an isomorphism, too.
 Thus $\varinjlim_n M\in\sG$.

 We have shown that countable direct limits of copies of the object $M$
in $\sA$ belong to $\sG$, and we have also seen that the functor $\Psi$
transforms countable direct limits of copies of $M$ in $\sA$ to
countable direct limits of copies of $P$ in~$\sB$.
 Therefore, countable direct limits of copies of $P$ belong to
$\Psi(\sG)=\sH$.
\end{proof}

\begin{cor} \label{small-THEC-corollary}
 Let\/ $\sA$ be a cocomplete additive category and $M\in\sA$ be
an object satisfying THEC\@.
 Let $\sB=\boT_M\modl$ be the related cocomplete abelian category
with a projective generator $P=\boT_M(*)\in\sB$ corresponding to
the object $M\in\sA$.
 Then the following two conditions are equivalent:
\begin{enumerate}
\item any countable direct limit of copies of $M$ has
an\/ $\Add(M)$\+cover in\/~$\sA$;
\item any countable direct limit of copies of $P$ has
a projective cover in\/~$\sB$.
\end{enumerate}
 The following three conditions are also equivalent to each other:
\begin{enumerate}
\setcounter{enumi}{2}
\item any countable direct limit of copies of $M$ in\/ $\sA$
belongs to\/~$\Add(M)$;
\item any countable direct limit of copies of $P$ in\/ $\sB$
is projective;
\item any countable direct limit of copies of $M$ in\/ $\sA$
belongs to\/ $\Add(M)$, and the related natural epimorphism
$\coprod_{n=1}^\infty M\rarrow\varinjlim_{n\ge1}M$
\eqref{categorical-telescope-sequence} splits.
\end{enumerate}
 All the five conditions~\textup{(1\+-5)} are equivalent when\/
$\sA$ is an abelian category with exact countable direct limits.
 Also, all the five conditions~\textup{(1\+-5)} are equivalent
when\/ $\sB=\R\contra$ is the category of left contramodules over
a complete, separated, right linear topological ring\/~$\R$.
\end{cor}

\begin{proof}
 Both the equivalences (1)\,$\Longleftrightarrow$\,(2) and
(3)\,$\Longleftrightarrow$\,(4) follow from
Proposition~\ref{THEC-Bass-objects-in-A-and-B} and
the equivalence of categories~\eqref{G-H-equivalence}.
 Since any epimorphism onto a projective object splits in $\sB$,
we also obtain the equivalence (4)\,$\Longleftrightarrow$\,(5).
 Alternatively, the equivalence (3)\,$\Longleftrightarrow$\,(5)
follows directly from THEC\@.

 When $\sA$ is an abelian category with exact countable direct
limits, the equivalence (1)\,$\Longleftrightarrow$\,(3) holds by
(the proof of) \cite[Theorem~4.4]{BPS}.
 When $\sB=\R\contra$, the equivalence (2)\,$\Longleftrightarrow$\,(4)
is provided by~\cite[Theorem~2.1 or Corollary~2.10]{BPS}.
\end{proof}

 The following theorem is the main result of this section.
 
\begin{thm} \label{main-conjecture+THEC-imply-AddM-Enochs}
 Let\/ $\sA$ be a locally presentable additive category and $M\in\sA$
be an object satisfying THEC\@.
 Assume that Main Conjecture~\ref{main-conjecture} holds for
the locally presentable abelian category\/ $\sB=\boT_M\modl$.
 Then the following conditions are equivalent:
\begin{enumerate}
\item the class of objects\/ $\Add(M)\subset\sA$ is covering;
\item any direct limit of objects from\/ $\Add(M)$ has an\/
$\Add(M)$\+cover in\/~$\sA$;
\item any countable direct limit of copies of $M$ has
an\/ $\Add(M)$\+cover in\/~$\sA$;
\item any countable direct limit of copies of $M$ in\/ $\sA$
belongs to\/ $\Add(M)$;
\item the class of objects\/ $\Add(M)$ is closed under direct limits
in\/~$\sA$;
\item the class\/ $\sB_\proj$ is covering in\/~$\sB$;
\item any direct limit of projective objects has a projective cover
in\/~$\sB$;
\item any countable direct limit of copies of the projective
generator\/ $P=\boT_M(*)\in\sB$ has a projective cover in\/~$\sB$;
\item any countable direct limit of copies of the projective
generator\/ $P=\boT_M(*)$ is projective in\/~$\sB$;
\item the class\/ $\sB_\proj$ is closed under direct limits in\/~$\sB$.
\end{enumerate}
\end{thm}

\begin{proof}
 The implications (1)\,$\Longrightarrow$\,(2)\,$\Longrightarrow$\,(3)
and (5)\,$\Longrightarrow$\,(4)\,$\Longrightarrow$\,(3) are obvious.
 The implication (5)\,$\Longrightarrow$\,(1) holds by
Example~\ref{Add(M)-precovers} and
Theorem~\ref{direct-limits-imply-covers}.

 Conditions~(6\+-10) are equivalent to each other by assumption.
 The implications (9)\,$\Longrightarrow$\,(4) and
(10)\,$\Longrightarrow$\,(5) are provided by
Lemma~\ref{direct-limit-closure-B-implies-A}.

 Finally, the conditions~(3) and~(8) are equivalent by
Corollary~\ref{small-THEC-corollary}\,(1)\,$\Leftrightarrow$\,(2).
\end{proof}

\Section{Perfect Decompositions} \label{perfect-decomposition-secn}

 The following definitions and terminology can be found in
the manuscript~\cite{Cor}.

 Let $\sA$ be an additive category with set-indexed products and
coproducts.
 Then the category $\sA$ is called \emph{agreeable} if, for every
family of objects $(N_x\in\sA)_{x\in X}$, the natural morphism
from the coproduct to the product
$$
 \coprod\nolimits_{x\in X}N_x\lrarrow
 \prod\nolimits_{x\in X}N_x
$$
is a monomorphism in~$\sA$.

 More generally, let $\sA$ be an additive category with coproducts
(but not necessarily with products).
 Consider an object $M\in\sA$ and a family of objects
$(N_x\in\sA)_{x\in X}$.
 For every index $y\in X$, one has the natural coordinate projection
morphism $\pi_y\:\coprod_{x\in X}N_x\rarrow N_y$.
 Given a morphism $f\:M\rarrow\coprod_{x\in X}N_x$, one can compose it
with the morphism~$\pi_y$, obtaining a morphism $\pi_y\circ f\:
M\rarrow N_y$.
 Consider the map of abelian groups
$$
 \eta\:\Hom_\sA\left(M,\>\coprod\nolimits_{x\in X} N_x\right)
 \lrarrow\prod\nolimits_{x\in X}\Hom_\sA(M,N_x)
$$
assigning to a morphism~$f$ the collection of morphisms
$f_x=\pi_x\circ f$, \ $x\in X$.

 Following~\cite{Cor}, we will say that the category $\sA$ is
\emph{agreeable} if the map~$\eta$ is injective for all objects $M$
and families of objects $N_x\in\sA$.
 When the category $\sA$ has products as well as coproducts,
this definition is clearly equivalent to the previous one.

 We will say that a family of morphisms $(f_x\:M\to N_x)$ in
an agreeable category $\sA$ is \emph{summable} if there exists
a morphism $f\:M\rarrow\coprod_{x\in X}N_x$ such that
$f_x=\pi_x\circ f$ for every $x\in X$.
 When $N_x=N$ is one and the same object for all $x\in X$, one can
construct the \emph{sum} $g=\sum_{x\in X}f_x$ of a summable family
of morphisms $(f_x\:M\to N)_{x\in X}$ as the composition
$g=\Sigma\circ f$ of the morphism $f\:M\rarrow N^{(X)}$ with
the natural summation morphism $\Sigma\:N^{(X)}\rarrow N$.

 In this paper, we will not be dealing with the sums of summable
families of morphisms.
 Instead, we will use the notion of a summable family in order to
extend the classical concept of a module with \emph{perfect
decomposition}~\cite{AS} to the categorical realm.

 Let $\sA$ be an agreeable additive category and
$(M_\xi\in\sA)_{\xi\in\Xi}$ be a family of objects.
 For any sequence of indices $\xi_1$, $\xi_2$, $\xi_3$,~\dots~$\in\Xi$
and any sequence of morphisms $f_i\:M_{\xi_i}\rarrow M_{\xi_{i+1}}$
in $\sA$, we consider the sequence of compositions
$$
 f_nf_{n-1}\dotsm f_1\:M_{\xi_1}\rarrow M_{\xi_{n+1}}, \qquad n\ge1.
$$
 The family of objects $(M_\xi)_{\xi\in\Xi}$ is said to be \emph{locally
T\+nilpotent} if for every sequence of indices~$\xi_i$ and every
sequence of \emph{nonisomorphisms} $f_i\:M_{\xi_i}\rarrow M_{\xi_{i+1}}$,
the family of morphisms $(f_nf_{n-1}\dotsm f_1)_{n\ge1}$ is summable
in~$\sA$.

 In the case of a module category $\sA=A\modl$, this reduces to
the classical definition: a family of modules $(M_\xi)_{\xi\in\Xi}$ is
\emph{locally T\+nilpotent} if for every sequence of indices~$\xi_i$,
every sequence of nonisomorphisms $f_i\:M_{\xi_i}\rarrow M_{\xi_{i+1}}$
in $A\modl$, and every element $m\in M_{\xi_1}$, there exists an integer
$n\ge1$ such that $f_nf_{n-1}\dotsm f_1(m)=0$ in $M_{\xi_{n+1}}$.

 An object $M$ of an agreeable additive category $\sA$ is said to have
a \emph{perfect decomposition} if there exists a locally T\+nilpotent
family of objects $(M_\xi\in\sA)_{\xi\in\Xi}$ such that
$M\cong\coprod_{\xi\in\Xi}M_\xi$.
 More generally, one can (and we will) drop the assumption that $\sA$
is agreeable and just assume that the full subcategory $\Add(M)\subset
\sA$ is agreeable instead.
 Thus, let $\sA$ be an additive category with coproducts
and let $M\in\sA$ be an object.
 We will say that $M$ has a perfect decomposition if the category
$\Add(M)$ is agreeable and there exists a locally T\+nilpotent family
of objects $(M_\xi\in\Add(M))_{\xi\in\Xi}$ such that
$M\cong\coprod_{\xi\in\Xi}M_\xi$.

 The definition of a topologically left perfect topological ring,
which was introduced in~\cite[Section~10]{PS3} and reproduced above in
Section~\ref{seven-conditions-secn}, is the topological ring
counterpart of the notion of an object with perfect decomposition.
 The following result obtained in the paper~\cite{PS3} illustrates
the connection.
 In the case of module categories, the equivalence of conditions~(i)
and~(iii) was established in~\cite[Theorem~1.4]{AS}.

\begin{thm} \label{topological-rings-perfect-decompositions}
 Let\/ $\sA$ be an idempotent-complete additive category with
coproducts and $M\in\sA$ be an object.
 Assume that the monad\/ $\boT_M\:\Sets\rarrow\Sets$ is isomorphic to
the monad\/ $\boT_\R$ for a complete, separated right linear
topological ring\/~$\R$.
 Consider the following three properties:
\begin{enumerate}
\renewcommand{\theenumi}{\roman{enumi}}
\item the object $M\in\sA$ has a perfect decomposition;
\item the topological ring\/ $\R$ is topologically left perfect;
\item for any directed poset\/ $\Theta$ and a diagram $A\:\Theta
\rarrow\Add(M)$, the direct limit
$\varinjlim_{\theta\in\Theta}A(\theta)$ exists in $\sA$, belongs
to\/ $\Add(M)$, and the natural epimorphism\/ $\coprod_{\theta\in\Theta}
A(\theta)\rarrow\varinjlim_{\theta\in\Theta}A(\theta)$ is split.
\end{enumerate}
 Then the implications \textup{(i)} $\Longleftrightarrow$
\textup{(ii)} $\Longrightarrow$ \textup{(iii)} hold.

 If\/ $\sA$ is a cocomplete abelian category with exact direct limits,
then all the three conditions~\textup{(i\+-iii)} are equivalent.
 If\/ $\sA=\R\contra$ is the category of contramodules over a complete,
separated right linear topological ring and $M=\R$ is the free\/
contramodule, then all the three conditions~\textup{(i\+-iii)} are
equivalent as well.
\end{thm}

\begin{proof}
 (i)\,$\Longleftrightarrow$\,(ii)
 By the definition, an object $M\in\sA$ having a perfect decomposition
means that $M$ has a perfect decomposition as an object of the category
$\Add(M)$.
 Following~\eqref{generalized-tilting-equivalence}, the category
$\Add(M)$ is equivalent to $\boT_M\modl_\proj$; so an isomorphism of
monads $\boT_M\cong\boT_\R$ implies that $\Add(M)$ is equivalent
to the category of projective left $\R$\+contramodules
$\R\contra_\proj$.
 This equivalence of categories takes the object $M\in\Add(M)$ to
the object $\R\in\R\contra_\proj$.
 According to~\cite[Remark~3.11]{PS3}, the category $\R\contra_\proj$
is topologically agreeable.
 Now the equivalence of the two conditions
(i)\,$\Longleftrightarrow$\,(ii) is provided
by~\cite[Theorem~10.4]{PS3}.

 (i)\,$\Longrightarrow$\,(iii)
 By~\cite[Theorem~10.2]{PS3}, condition~(i) implies (in fact,
is equivalent to) the category $\Add(M)$ having split direct limits
in the sense of~\cite[Section~9]{PS3}.
 According to~\cite[Lemma~9.2(b)]{PS3}, it follows that the direct
limits of diagrams in $\Add(M)$ exist in $\sA$ and belong to $\Add(M)$;
and by~\cite[Lemma~9.1\,(1)\,$\Rightarrow$\,(3)]{PS3}, the natural
epimorphisms $\coprod_{\theta\in\Theta}A(\theta)\rarrow
\varinjlim_{\theta\in\Theta}A(\theta)$ are split.

 (iii)\,$\Longrightarrow$\,(i)
 By~\cite[Theorem~10.2]{PS3}, in order to prove~(i) it suffices to check
that the category $\Add(M)$ has split direct limits.
 Now in the case of an abelian category $\sA$ with exact direct limits,
the desired implication is provided
by~\cite[Corollary~9.3\,(3)\,$\Rightarrow$\,(0)]{PS3}.
 In the case when $\sA=\R\contra$ and $M=\R$,
\,\cite[Proposition~9.6]{PS3} is applicable.
\end{proof}

 Countable direct limits of copies of the free $\R$\+contramodule
with one generator $\R=\R[[*]]$ in $\R\contra$ are called
\emph{Bass flat\/ $\R$\+contramodules}~\cite[Section~4]{Pproperf}.

\begin{cor} \label{contramodule-THEC-corollary}
 Let\/ $\sA$ be a cocomplete additive category and $M\in\sA$
be an object satisfying THEC\@.
 Assume that the monad\/ $\boT_M\:\Sets\rarrow\Sets$ is isomorphic to
the monad\/ $\boT_\R$ for a complete, separated right linear
topological ring\/~$\R$.
 Consider the following ten properties:
\begin{enumerate}
\item the object $M\in\sA$ has a perfect decomposition;
\item the topological ring\/ $\R$ is topologically left perfect;
\item the class\/ $\R\contra_\proj$ is closed under direct limits
in\/ $\R\contra$;
\item the class of objects\/ $\Add(M)$ is closed under direct limits
in\/~$\sA$;
\item every countable direct limit of copies of $M$ in $\sA$
belongs to\/ $\Add(M)$;
\item all Bass flat left\/ $\R$\+contramodules are projective;
\item all Bass flat left\/ $\R$\+contramodules have projective covers
in\/ $\R\contra$;
\item every countable direct limit of copies of $M$ has
an\/ $\Add(M)$\+cover in\/~$\sA$;
\item all descending chains of cyclic discrete right\/ $\R$\+modules
terminate;
\item all the discrete quotient rings of the topological ring\/ $\R$ are
left perfect.
\end{enumerate}
 Then the following implications hold:
$$
 \mathrm{(1)} \Longleftrightarrow \mathrm{(2)}
 \Longleftrightarrow \mathrm{(3)} \Longrightarrow
 \mathrm{(4)} \Longrightarrow \mathrm{(5)} \Longleftrightarrow
 \mathrm{(6)} \Longleftrightarrow \mathrm{(7)} \Longleftrightarrow
 \mathrm{(8)} \Longrightarrow \mathrm{(9)}
 \Longrightarrow \mathrm{(10)}.
$$
 If the topological ring\/ $\R$ satisfies one of the conditions~(a), (b),
(c), or~(d), then all the conditions~\textup{(1\+-10)} are equivalent.
 If the topological ring\/ $\R$ satisfies one of the conditions~(e),
(f), or~(g), then the nine conditions~\textup{(1\+-9)} are equivalent.
\end{cor}

\begin{proof}
 (1)\,$\Longleftrightarrow$\,(2) is
Theorem~\ref{topological-rings-perfect-decompositions}%
(i)\,$\Leftrightarrow$\,(ii).

 (2)\,$\Longleftrightarrow$\,(3)
is~\cite[Theorem~14.1\,(iv)\,$\Leftrightarrow$\,(iii$'$)]{PS3}.

 (1)\,$\Longrightarrow$\,(4) is
Theorem~\ref{topological-rings-perfect-decompositions}%
(i)\,$\Rightarrow$\,(iii);
(3)\,$\Longrightarrow$\,(4) is
Lemma~\ref{direct-limit-closure-B-implies-A}.

 The implications (4)\,$\Longrightarrow$\,(5)\,$\Longrightarrow$\,(8)
and (3)\,$\Longrightarrow$\,(6)\,$\Longrightarrow$\,(7) are obvious.

 The equivalences (5)\,$\Longleftrightarrow$\,(6)\,%
$\Longleftrightarrow$\,(7)\,$\Longleftrightarrow$\,(8) are provided by
Corollary~\ref{small-THEC-corollary}.

 (7)\,$\Longrightarrow$\,(10) is~\cite[Corollary~4.7]{Pproperf};
 (6)\,$\Longrightarrow$\,(10) is~\cite[Corollary~4.5]{Pproperf}.

 (6)\,$\Longrightarrow$\,(9) is~\cite[Proposition~4.3 and
Lemma~6.3]{Pproperf}; (9)\,$\Longrightarrow$\,(10) is explained
in~\cite[proof of Theorem~10.1]{Pproperf}
(see also~\cite[Theorem~14.4]{PS3}).

 The last two assertions of the corollary follow from
Corollary~\ref{a-b-c-d-THEC-main-corollary} below.
\end{proof}

 Left modules $M$ over an associative ring $A$ for which there exists
a topological ring $\R$ satisfying~(e) such that the monad $\boT_M$
is isomorphic to $\boT_\R$ are discussed under the name of
\emph{weakly countably generated modules} in
the paper~\cite[Section~7.2]{BPS}.

 The next corollary, covering the assertions of
Theorem~\ref{introd-THEC-main-theorem} from the introduction,
is our main result in the setting of
Sections~\ref{enochs-conjecture-secn}\+-\ref{perfect-decomposition-secn}.

\begin{cor} \label{a-b-c-d-THEC-main-corollary}
 Let\/ $\sA$ be a locally presentable additive category and $M\in\sA$
be an object satisfying THEC\@.
 Assume that the monad\/ $\boT_M\:\Sets\rarrow\Sets$ is isomorphic to
the monad\/ $\boT_\R$ for a complete, separated right linear
topological ring\/ $\R$ satisfying one of the conditions~(a), (b),
(c), or~(d).
 Let\/ $\sB=\boT_M\modl\cong\R\contra$ be the related abelian category
of contramodules.
 Then the following conditions are equivalent:
\begin{enumerate}
\item the class of objects\/ $\Add(M)\subset\sA$ is covering;
\item any direct limit of objects from\/ $\Add(M)$ has an\/
$\Add(M)$\+cover in\/~$\sA$;
\item every countable direct limit of copies of $M$ has
an\/ $\Add(M)$\+cover in\/~$\sA$;
\item every countable direct limit of copies of $M$ in\/ $\sA$
belongs to\/ $\Add(M)$;
\item the class of objects\/ $\Add(M)$ is closed under direct limits
in\/~$\sA$;
\item the class\/ $\sB_\proj$ is covering in\/~$\sB$;
\item any direct limit of projective objects has a projective cover
in\/~$\sB$;
\item every countable direct limit of copies of the projective
generator\/ $\R\in\sB$ has a projective cover in\/~$\sB$;
\item every countable direct limit of copies of the projective
generator\/ $\R\in\sB$ is a projective object in\/~$\sB$;
\item the class\/ $\sB_\proj$ is closed under direct limits in\/~$\sB$;
\item the object $M\in\sA$ has a perfect decomposition;
\item the topological ring\/ $\R$ is topologically left perfect;
\item all descending chains of cyclic discrete right\/ $\R$\+modules
terminate;
\item there is a topologically left T\+nilpotent strongly closed
two-sided ideal\/ $\fH\subset\R$ such that the quotient ring\/
$\S=\R/\fH$ is isomorphic, as a topological ring, to a product of simple
Artinian discrete rings endowed with the product topology;
\item all the discrete quotient rings of the topological ring $\R$ are
left perfect.
\end{enumerate}
 Replacing the assumption of one of the conditions~(a\+-d) with that of
one of the conditions (e), (f), or~(g), the thirteen
conditions~\textup{(1\+-13)} are equivalent.
\end{cor}

\begin{proof}
 The conditions~(1\+-10) are equivalent to each other by
Theorem~\ref{main-conjecture+THEC-imply-AddM-Enochs}, whose
applicability follows from any one of the conditions~(a), (b), (c),
(d), (e), (f), or~(g) by Theorem~\ref{a-b-c-d-imply-main-conjecture}.
 Under any one of the conditions~(a), (b), (c), or~(d),
the conditions~(6\+-10) and~(13\+-15) are equivalent to each other
by~\cite[Theorem~12.4]{Pproperf}, while the conditions~(6\+-10)
and~(12\+-13) are equivalent to each other
by~\cite[Corollary~14.7]{PS3}.
 Assuming~(e), (f), or~(g), the conditions~(6\+-10) and~(12\+-13) are
equivalent to each other by Theorems~\ref{a-b-c-d-imply-main-conjecture}
and~\ref{g-implies-as}, and~\cite[Theorem~14.4]{PS3}.
 The equivalence~(11)\,$\Longleftrightarrow$\,(12) holds by
Theorem~\ref{topological-rings-perfect-decompositions}\,%
(i)\,$\Leftrightarrow$\,(ii).

 This suffices to prove the corollary; but alternatively,
here is a direct proof of the equivalence
(12)\,$\Longleftrightarrow$\,(14) under the assumption of
condition~(d).
 By the Artin--Wedderburn classification of simple Artinian rings,
(14) implies~(12) unconditionally.
 In fact, the only difference between (14) and~(12) is that the sets
$\Upsilon_\gamma$ can be infinite in~(12); the class of topological
rings $\S$ in~(14) is obtained by such class in~(12) by imposing
the condition that all the sets $\Upsilon_\gamma$ are finite
(cf.~\cite[Remark~14.6]{PS3}).

 Let $\R$ be a topologically left perfect topological ring
satisfying~(d).
 Let $\fH\subset\R$ be the ideal from the defintion of a topologically
left perfect topological ring, and $\fK\subset\R$ be the ideal from
condition~(d).
 Then the argument from the proof of
Lemma~\ref{topologically-perfect-quotient-by-T-nilpotent}
(based on~\cite[Lemma~10.3]{PS3} and~\cite[Lemma~6.6(a)]{Pproperf})
shows that $\fK\subset\fH$.
 So the topological ring $\S=\R/\fH$ is a quotient ring of
the topological ring $\R/\fK$.
 Hence for every $\gamma\in\Gamma$ the topological ring
$\S_\gamma=\Hom_{D_\gamma}(D_\gamma^{(\Upsilon_\gamma)},
D_\gamma^{(\Upsilon_\gamma)})^\rop$ is also a quotient ring
of the topological ring~$\R/\fK$.

 The ring $\R/\fK\cong\prod_{\delta\in\Delta}\T_\delta=\T$, on the other
hand, is the product of two-sided linear topological rings $\T_\delta$,
so $\T$ is a two-sided linear topological ring, too.
 As any topological quotient ring of a two-sided linear topological ring
is two-sided linear, the ring $\S_\gamma$ must be two-sided linear,
i.~e., it has a base of neighborhoods of zero consisting of two-sided
ideals.
 Since, in fact, there are no nonzero proper open two-sided ideals in
$\S_\gamma$, it follows that $\S_\gamma$ must be discrete, which happens
exactly when the set $\Upsilon_\gamma$ is finite.
\end{proof}

\Section{Functor Purity in Abelian Categories}
\label{functor-purity-secn}

 Let $A$ be an associative ring.
 A short exact sequence of left $A$\+modules $0\rarrow K\rarrow M
\rarrow L\rarrow 0$ is said to be \emph{pure} if the map of abelian
groups $N\ot_A K\rarrow N\ot_A M$ is injective for every right
$A$\+module $N$, or equivalently, if the map $\Hom_A(E,M)\rarrow
\Hom_A(E,L)$ is surjective for every finitely presented left
$A$\+module~$E$.
 A short exact sequence of left $A$\+modules is pure if and only if
it is a direct limit of split short exact
sequences of left $A$\+modules~\cite[Lemma~2.19]{GT}.

 The aim of this section is to suggest a simple way to extend the notion
of purity to arbitrary cocomplete abelian categories.
 We will use it in the next Section~\ref{self-pure-projective-secn}.

 Let $\sA$ be a cocomplete abelian category.
 We will say that a monomorphism $f\:K\rarrow M$ is \emph{pure} (or
\emph{functor pure}) in $\sA$ if for every cocomplete abelian category
$\sV$ with exact direct limit functors, and any additive functor
$F\:\sA\rarrow\sV$ preserving all colimits (that is, a right exact
covariant functor preserving coproducts), the morphism
$F(f)\:F(K)\rarrow F(M)$ is a monomorphism in~$\sV$.
 If this is the case, the object $K$ is said to be a \emph{(functor)}
\emph{pure subobject} of the object $M\in\sA$.

 A short exact sequence $0\rarrow K\rarrow M\rarrow L\rarrow0$
in $\sA$ is called (\emph{functor}) \emph{pure} if the monomorphism
$K\rarrow M$ is pure, or equivalently, if the short sequence
$0\rarrow F(K)\rarrow F(M)\rarrow F(L)\rarrow0$ is exact in $\sV$
for every functor $F\:\sA\rarrow\sV$ as above.
 The morphism $M\rarrow L$ is then said to be a (\emph{functor})
\emph{pure epimorphism}, and the object $L$ a \emph{pure quotient}
of~$M$.
 A long exact sequence $M^\bu$ in $\sA$ is said to be \emph{pure} if
it is obtained by splicing pure short exact sequences in $\sA$, or
equivalently, if the complex $F(M^\bu)$ is exact in $\sV$ for every
abelian category $\sV$ with exact direct limits and any
colimit-preserving functor $F\:\sA\rarrow\sV$.

\begin{lem}
 Let\/ $\sA=A\modl$ be the abelian category of left modules over
an associative ring~$A$.
 Then a monomorphism (or a short exact sequence, or a long exact
sequence) in $A\modl$ is functor pure if and only if it is pure
in the conventional sense of the word (as in~\cite{GT}).
\end{lem}

\begin{proof}
 A functor $A\modl\rarrow\Ab$ from the category of left $A$\+modules to
the category of abelian groups $\Ab$ preserves colimits if and only
if it is isomorphic to the functor of tensor product $M\longmapsto
N\otimes_AM$ with a certain right $A$\+module~$N$ \cite[Theorem~1]{Wat}.
 (Colimit-preserving functors $A\modl\rarrow\sV$ can be similarly
described as the functors of tensor product with an object in $\sV$
endowed with a right action of the ring~$A$.)
 So any functor pure exact sequence in $A\modl$ remains exact after
taking the tensor product with any right $A$\+module $N$, i.~e., it
is pure exact in the conventional sense.

 Conversely, any pure short exact sequence of left $A$\+modules
is a direct limit of split short exact sequences.
 Hence its image under any colimit-preserving functor (and more
generally, under any direct limit-preserving additive functor)
$F\:A\modl\rarrow\sV$, taking values in an abelian category $\sV$
with exact direct limits, is exact.
\end{proof}

\begin{lem}
 In any cocomplete abelian category\/ $\sA$, the class of functor pure
monomorphisms is closed under pushouts and compositions.
 The class of functor pure epimorphisms is closed under pullbacks and
compositions.
\end{lem}

\begin{proof}
 Essentially, the lemma claims that the category $\sA$ with the class of
all pure short exact sequences is a Quillen exact category.
 To prove such an assertion, it suffices to check that the class of
pure monomorphisms is closed under pushouts and compositions, and
the class of pure epimorphisms is closed under pullbacks.
 Closedness of the class of pure epimorphisms with respect to compositions
will then follow~\cite[Section~A.1]{Kel}.

 Let $0\rarrow K\rarrow M\rarrow L\rarrow0$ be a pure short exact
sequence in $\sA$ and $L'\rarrow L$ be a morphism.
 Since $\sA$ is an abelian category, the pullback sequence
$0\rarrow K\rarrow M'\rarrow L'\rarrow0$ is exact.
 To show that the epimorphism $M'\rarrow L'$ is pure, it suffices to
check that the monomorphism $K\rarrow M'$ is pure.
 Indeed, the composition $K\rarrow M'\rarrow M$ is a pure monomorphism.
 Since for any colimit-preserving functor $F\:\sA\rarrow\sV$ the morphism
$F(K)\rarrow F(M)$ is a monomorphism, the morphism $F(K)\rarrow F(M')$
is a monomorphism, too.

 Let $K\rarrow K''$ be a morphism in $\sA$ and $0\rarrow K''\rarrow M''
\rarrow L\rarrow 0$ be the pushout sequence.
 Once again, since $\sA$ is abelian, the pushout sequence is exact.
 Any colimit-preserving functor $F\:\sA\rarrow\sV$ preserves pushouts;
so $F(K)\rarrow F(M)\rarrow F(M'')$, \ $F(K)\rarrow F(K'')\rarrow
F(M'')$ is a pushout square.
 Since the morphism $F(K)\rarrow F(M)$ is a monomorphism, the morphism
$F(K'')\rarrow F(M'')$ is a monomorphism, too.

 The assertion that the composition of any two pure monomorphisms is
a pure monomorphism is obvious.
\end{proof}

\begin{ex} \label{telescope-pure-exact-example}
 Let $\sA$ be a cocomplete abelian category with exact countable direct
limits.
 Then, for any sequence of objects and morphisms $A_1\rarrow A_2\rarrow
A_3\rarrow\dotsb$ in $\sA$, the short sequence
\begin{equation} \label{telescope-exact-sequence}
 0\rarrow\coprod\nolimits_{n=1}^\infty A_n\lrarrow
 \coprod\nolimits_{n=1}^\infty A_n\lrarrow
 \varinjlim\nolimits_{n\ge1}A_n\lrarrow0.
\end{equation}
is pure exact.
 Indeed, the sequence~\eqref{telescope-exact-sequence} is exact as
the countable direct limit of split exact sequences
$0\rarrow\coprod_{i=1}^{n-1}A_i\rarrow\coprod_{i=1}^nA_i\rarrow
A_n\rarrow0$.
 The image of~\eqref{telescope-exact-sequence} under
a colimit-preserving functor $F\:\sA\rarrow\sV$ is the similar short
sequence for the inductive system $F(A_1)\rarrow F(A_2)\rarrow F(A_3)
\rarrow\dotsb$ in the category $\sV$, which is exact whenever countable
direct limits are exact in~$\sV$.
\end{ex}

\begin{ex} \label{bar-complex-pure-exact-example}
 Let $\sA$ be a cocomplete abelian category with exact direct limits.
 Let $\Theta$ be a directed poset and $A\:\Theta\rarrow\sA$ be
a $\Theta$\+indexed diagram in~$\sA$.
 Then the augmented bar-complex
\begin{equation} \label{augmented-bar-complex}
 \dotsb\lrarrow\coprod_{\theta_0\le\theta_1\le\theta_2}A(\theta_0)
 \lrarrow\coprod_{\theta_0\le\theta_1}A(\theta_0)\lrarrow
 \coprod_{\theta_0}A(\theta_0)\lrarrow
 \varinjlim_{\theta\in\Theta}A(\theta)\lrarrow0
\end{equation}
is pure exact in~$\sA$.
 Indeed, the complex~\eqref{augmented-bar-complex} is the direct limit
(over $\delta\in\Theta$) of the similar bar-complexes related to
the subposets $\Theta_\delta=\{\theta\in\Theta\:\theta\le\delta\}
\subset\Theta$ and the subdiagrams $A|_{\Theta_\delta}$ of~$A$.
 The bar-complex of any diagram indexed by a poset with a greatest element
is easily seen to be contractible (by the explicit contracting homotopy
given by the morphisms taking the summand $A(\theta_0)$ indexed by
$\theta_0\le\dotsb\le\theta_n$ to the summand $A(\theta_0)$ indexed by
$\theta_0\le\dotsb\le\theta_n\le\delta$).

 This proves exactness of~\eqref{augmented-bar-complex}.
 To prove the pure exactness, one observes that the image
of~\eqref{augmented-bar-complex} under a colimit-preserving functor
$F\:\sA\rarrow\sV$ is the similar augmented bar-complex for the diagram
$F\circ A\:\Theta\rarrow\sV$ in the category $\sV$, which is exact for
the same reason explained above whenever direct limits are exact
in~$\sV$.

 In addition, we have shown that all the objects of cycles in
the bar-complex~\eqref{augmented-bar-complex} are direct limits (over
the poset~$\Theta$) of direct sums of copies of the objects
$A(\theta)$, \,$\theta\in\Theta$.
 Indeed, one easily observes that all the objects of cycles in
the bar-complexes related to the subposets $\Theta_\delta\subset\Theta$
are direct sums of copies of the objects~$A(\theta)$.
\end{ex}

\Section{Self-Pure-Projective and
$\protect\varinjlim$-Pure-Rigid Objects}
\label{self-pure-projective-secn}

 The aim of this section is to prove the analogues of such results as
Proposition~\ref{THEC-Bass-objects-in-A-and-B},
Corollary~\ref{small-THEC-corollary}, and the related equivalence of
properties in Corollary~\ref{contramodule-THEC-corollary} for
uncountable direct limits, under appropriate assumptions.

 Let $\sA$ be a cocomplete abelian category.
 We use the notion of (functor) purity defined in
Section~\ref{functor-purity-secn}.

 An object $M\in\sA$ is said to be \emph{pure-split} if every pure
monomorphism $K\rarrow M$ is split in~$\sA$.
 One says that an object $T\in\sA$ is \emph{$\Sigma$\+pure-split} if
all the objects $M$ from the class $\Add(T)\subset\sA$ are pure-split
in~$\sA$.

 An object $Q\in\sA$ is said to be \emph{pure-projective} if, for
any pure short exact sequence $0\rarrow K\rarrow M\rarrow L\rarrow0$
in the category $\sA$, the short sequence of abelian groups $0\rarrow
\Hom_\sA(Q,K)\rarrow\Hom_\sA(Q,M)\rarrow\Hom_\sA(Q,L)\rarrow0$ is exact.
 In other words, an object of $\sA$ is pure-projective if it is
projective with respect to the pure exact structure on~$\sA$.

 We will say that an object $M\in\sA$ is \emph{self-pure-projective}
if, for any pure short exact sequence $0\rarrow K\rarrow M'\rarrow L
\rarrow0$ in $\sA$ with $M'\in\Add(M)$, the short sequence of abelian
groups $0\rarrow\Hom_\sA(M,K)\rarrow\Hom_\sA(M,M')\rarrow\Hom_\sA(M,L)
\rarrow0$ is exact.
 The following examples mention classes of objects that are known
to be self-pure-projective, showing that self-pure-projective objects
and, in particular, self-pure-projective modules are not uncommon.

\begin{exs} \label{are-self-pure-projective-easy}
 The following objects in a cocomplete abelian category $\sA$ are
self-pure-projective: \par
(1)~all pure-projective objects; \par
(2)~all $\Sigma$\+pure-split objects; \par
(3)~all the objects belonging to $\Add(M)$, if $M\in\sA$ is
a self-pure-projective object.
\end{exs}

\begin{exs} \label{are-self-pure-projective-harder}
 (1)~Let $\sL$ and $\sE\subset\sA$ be two classes of
objects such that $\Ext^1_\sA(L,E)=0$ for all $L\in\sL$ and
$E\in\sE$ (cf.\ Section~\ref{covers-cotorsion-secn} below).
 Assume that the class $\sE\subset\sA$ is closed under coproducts
and pure subobjects.
 Then all objects in the intersection $\sL\cap\sE\subset\sA$ are
self-pure-projective.

 Indeed, $M\in\sL\cap\sE$ and $M'\in\Add(M)$ implies $M'\in\sE$; and
if a (pure) subobject $K$ of $M'$ also belongs to $\sE$,
then $\Ext^1_\sA(M,K)=0$.
 Consequently $\Hom_\sA(M,M')\rarrow\Hom_\sA(M,M'/K)$
is a surjective map.

\smallskip
 (2)~In particular, let $\sA=A\modl$ be the category of modules over
an associative ring~$A$.
 Then any $n$\+tilting left $A$\+module
(cf.\ Section~\ref{tilting-cotilting-secn} below)
is self-pure-projective.
 Indeed, any $n$\+tilting class $\sE$ in $A\modl$ is
definable, which implies, in particular, that it is closed under direct
sums and pure submodules~\cite[Definition~6.8 and Corollary~13.42]{GT}.
\end{exs}

\begin{rems}
 (1)~A pair of classes of objects $(\sE,\sL)$ in abelian category $\sA$
is said to be a \emph{cotorsion pair} if both the classes $\sL$ and
$\sE$ are maximal with respect to the property that  $\Ext^1_\sA(L,E)=0$
for all $L\in\sL$ and $E\in\sE$ (see
Section~\ref{covers-cotorsion-secn}).
 Notice that if $\sA$ is a complete, cocomplete abelian category with
exact direct limits and $(\sL,\sE)$ is a cotorsion pair in $\sA$ such
that the class $\sE\subset\sA$ is closed under pure subobjects, then
the class $\sE$ is also closed under coproducts in~$\sA$.
 Indeed, the right class $\sE$ in a cotorsion pair $(\sL,\sE)$ is
always closed under products in~$\sA$ \cite[Appendix~A]{CS}.
 For any family of objects $A_\alpha\in\sA$, the natural morphism
$\coprod_\alpha A_\alpha\rarrow\prod_\alpha A_\alpha$ is a direct limit
of split monomorphisms, hence $\coprod_\alpha A_\alpha$ is a pure
subobject of $\prod_\alpha A_\alpha$.
 It follows that all the objects in the class $\sL\cap\sE$ are
self-pure-projective.

\smallskip
 (2)~Let $(\sL,\sE)$ be a cotorsion pair in the category of left modules
over an associative ring~$A$.
 In this context, if the class $\sE$ is closed under direct limits in
$R\modl$, then it is definable~\cite[Theorem~6.1]{Sar}.
 If the cotorsion pair $(\sL,\sE)$ is hereditary and the class $\sE$ is
closed under unions of well-ordered chains in $A\modl$, then the class
$\sE$ is definable as well~\cite[Theorem~3.5]{SS}.
 In both cases, the class $\sE\subset A\modl$ is closed under (direct
sums and) pure submodules, and it follows that all the $A$\+modules in
the class $\sL\cap\sE$ are self-pure-projective.
\end{rems}

 Let $\sA$ be a cocomplete abelian category.
 We will say that an object $M\in\sA$ is \emph{$\varinjlim$\+pure-rigid}
if $\PExt^1_\sA(M,N)=0$ for any object $N\in\varinjlim^\sA\Add(M)$.
 Here $\PExt^1_\sA({-},{-})$ denotes the group $\Ext^1$ in the functor
pure exact structure on the category~$\sA$ (cf.\
Example~\ref{sigma-pure-rigid-example}(2)) and the notation
$\varinjlim^\sA\sM$ was defined in the beginning of
Section~\ref{enochs-conjecture-secn}.
 Any $\varinjlim$\+pure-rigid object is $\Sigma$\+pure-rigid
by definition.

\begin{ex} \label{indlim-pure-rigid-example}
 Let $\sL$ and $\sE\subset\sA$ be two classes of objects such that
$\PExt^1_\sA(L,E)=0$ for all $L\in\sL$ and $E\in\sE$.
 Assume that the class $\sE\subset\sA$ is closed under direct limits.
 Then all objects in the intersection $\sL\cap\sE$ are
$\varinjlim$\+pure-rigid.
\end{ex}

 Let $\sA$ be a cocomplete abelian category and $M\in\sA$ be an object.
 As in Section~\ref{generalized-tilting-subsecn}, we consider
the related abelian category $\sB=\boT_M\modl$ and the pair of adjoint
functors $\Psi\:\sA\rarrow\sB$ and $\Phi\:\sB\rarrow\sA$.
 As in Section~\ref{THEC-secn}, we also consider the related pair of
full subcategories $\sG\subset\sA$ and $\sH\subset\sB$.

 The following proposition is the uncountable version of
Proposition~\ref{THEC-Bass-objects-in-A-and-B}.

\begin{prop} \label{self-pure-projective-direct-limits}
 Let\/ $\sA$ be a cocomplete abelian category with exact direct limits,
$M\in\sA$ be an object that is either self-pure-projective or\/
$\varinjlim$\+pure-rigid, $\sB=\boT_M\modl$ be the related abelian
category, and\/ $\sG\subset\sA$ and\/ $\sH\subset\sB$ be the related
two full subcategories.
 Then one has\/ $\varinjlim^\sA\Add(M)\subset\sG$ and\/
$\varinjlim^\sB\sB_\proj\subset\sH$.
 The functor\/ $\Psi$ preserves direct limits of objects from\/
$\Add(M)$ in\/~$\sA$ (taking them to direct limits of the corresponding
projective objects in\/~$\sB$).
\end{prop}

\begin{proof}
 Let $\Theta$ be a directed poset and $A\:\Theta\rarrow\sA$ be
a diagram in $\sA$ with $A(\theta)\in\Add(M)$ for all $\theta\in\Theta$.
 Then the augmented bar-complex~\eqref{augmented-bar-complex}
is pure exact in~$\sA$, and the objects of cycles in
the complex~\eqref{augmented-bar-complex} are direct limits of objects
from $\Add(M)$ (see Example~\ref{bar-complex-pure-exact-example}).
 As all the terms of this complex, except perhaps the rightmost one,
belong to $\Add(M)$ and the object $M$ is either self-pure-projective
or $\varinjlim$\+pure-rigid, it follows that the functor
$\Hom_\sA(M,{-})$ takes the complex~\eqref{augmented-bar-complex}
to an exact sequence of abelian groups.
 As in the proof of
Proposition~\ref{THEC-Bass-objects-in-A-and-B}, we conclude
that the functor $\Psi$ transforms
the complex~\eqref{augmented-bar-complex} into an exact complex in~$\sB$.

 On the other hand, for any cocomplete abelian category $\sB$,
any poset $\Theta$, and any diagram $B\:\Theta\rarrow\sB$,
the augmented bar-complex
\begin{equation} \label{bar-complex-B}
 \dotsb\lrarrow\coprod_{\theta_0\le\theta_1\le\theta_2}B(\theta_0)
 \lrarrow\coprod_{\theta_0\le\theta_1}B(\theta_0)\lrarrow
 \coprod_{\theta_0}B(\theta_0)\lrarrow
 \varinjlim_{\theta\in\Theta} B(\theta)\lrarrow0
\end{equation}
is exact, at least, at its rightmost term.

 In the situation at hand, put $B=\Psi\circ A\:\Theta\rarrow\sB$.
 Then the natural morphism from the complex~\eqref{bar-complex-B}
to the image of the complex~\eqref{augmented-bar-complex} 
 under $\Psi$ is an isomorphism at all the terms, except
perhaps the rightmost one.
 It follows that this morphism of complexes is an isomorphism at
the rightmost terms, too; that is, the natural morphism
$\varinjlim_{\theta\in\Theta}\Psi(A(\theta))\rarrow
\Psi(\varinjlim_{\theta\in\Theta}A(\theta))$ is an isomorphism.

 The argument finishes in the same way as the proof of
Proposition~\ref{THEC-Bass-objects-in-A-and-B}.
\end{proof}

 The next corollary is an uncountable version of
Corollary~\ref{small-THEC-corollary}.

\begin{cor} \label{self-pure-projective-corollary}
 Let\/ $\sA$ be a cocomplete abelian category with exact direct limits
and $M\in\sA$ be an object that is either self-pure-projective or\/
$\varinjlim$\+pure-rigid.
 Let $\sB=\boT_M\modl$ be the related abelian category.
 Then the following conditions are equivalent:
\begin{enumerate}
\item all the objects from\/ $\varinjlim\Add(M)$ have\/
$\Add(M)$\+covers in\/~$\sA$;
\item all the objects from\/ $\varinjlim\sB_\proj$ have
projective covers in\/~$\sB$;
\item the class of objects\/ $\Add(M)\subset\sA$ is closed under
direct limits;
\item the class of all projective objects in\/ $\sB$ is closed
under direct limits;
\item the object $M\in\sA$ satisfies the condition~(iii) of
Theorem~\ref{topological-rings-perfect-decompositions}.
\end{enumerate}
\end{cor}

\begin{proof}
 Both the equivalences (1)\,$\Longleftrightarrow$\,(2) and
(3)\,$\Longleftrightarrow$\,(4) follow from
Proposition~\ref{self-pure-projective-direct-limits} and
the equivalence of categories~\eqref{G-H-equivalence}.
 Since any epimorphism onto a projective object splits in $\sB$,
we also obtain the equivalence (4)\,$\Longleftrightarrow$\,(5).
 Alternatively, the equivalence (3)\,$\Longleftrightarrow$\,(5) follows
from self-pure-projectivity/$\varinjlim$\+pure-rigidity of $M$ and
the properties of the augmented bar-complex mentioned in
Example~\ref{bar-complex-pure-exact-example}.

 Finally, the equivalence (1)\,$\Longleftrightarrow$\,(3)
is~\cite[Corollary~7.2]{BPS} (for $\varinjlim$\+pure-rigid
objects~$M$) or a particular case of~\cite[Theorem~4.4]{BPS}
(for self-pure-projective objects~$M$).
\end{proof}

 In particular, in the assumptions of
Corollary~\ref{self-pure-projective-corollary}, the two properties
(3) and~(4) in Corollary~\ref{contramodule-THEC-corollary}
are equivalent.

\Section{Covers in Hereditary Cotorsion Pairs}
\label{covers-cotorsion-secn}

 In this section we discuss $\sL$\+covers in an abelian category $\sA$
with a hereditary cotorsion pair~$(\sL,\sE)$, aiming to gradually pass
from Theorem~\ref{introd-THEC-main-theorem} of the introduction to
Theorem~\ref{introd-tilting-main-theorem}.

 Let us recall the relevant definitions.
 Let $\sA$ be an abelian category, and let $\sL$ and $\sE\subset\sA$
be two classes of objects.
 We denote by $\sL^{\perp_1}\subset\sA$ the class of all objects
$X\in\sA$ such that $\Ext_\sA^1(L,X)=0$ for all $L\in\sL$, and by
${}^{\perp_1}\sE \subset\sA$ the class of all objects $Y\in\sA$
such that $\Ext_\sA^1(Y,E)=0$ for all $E\in\sE$.
 The pair of classes of objects $(\sL,\sE)$ in $\sA$ is called
a \emph{cotorsion pair} (or a \emph{cotorsion theory}) if
$\sE=\sL^{\perp_1}$ and $\sL={}^{\perp_1}\sE$.
 A cotorsion pair $(\sL,\sE)$ is called \emph{hereditary} if
$\Ext^n_\sA(L,E)=0$ for all $L\in\sL$, $E\in\sE$, and $n\ge1$.
 These concepts go back to Salce~\cite{Sal}.

 An epimorphism $l\:L\rarrow C$ in $\sA$ is called a \emph{special\/
$\sL$\+precover} if $L\in\sL$ and $\ker(l)\in\sL^{\perp_1}$.
 A monomorphism $b\:B\rarrow E$ in $\sA$ is called a \emph{special\/
$\sE$\+preenvelope} if $E\in\sE$ and $\coker(b)\in{}^{\perp_1}\sE$.
 The following lemma summarizes the properties of precovers, special
precovers, and covers.

\begin{lem} \label{precovers-lemma}
 Let\/ $\sL$ be a class of objects in an abelian category\/~$\sA$.
 Then the following assertions hold true: \par
\textup{(a)} Any special\/ $\sL$\+precover is an\/ $\sL$\+precover. \par
\textup{(b)} If the class\/ $\sL$ is closed under extensions in\/ $\sA$,
then the kernel of any\/ $\sL$\+cover belongs to\/~$\sL^{\perp_1}$.
 In particular, any epic\/ $\sL$\+cover is special in this case. \par
\textup{(c)} Let $l\:L\rarrow C$ be an\/ $\sL$\+cover, and let
$l'\:L'\rarrow C$ be an\/ $\sL$\+precover.
 Then there exists a split epimorphism $f\:L'\rarrow L$ forming
a commutative triangle diagram with the morphisms $l$ and~$l'$.
 The kernel $K$ of the morphism~$f$ is a direct summand of $L'$
contained in\/ $\ker(l')\subset L'$.
 So one has $L'\cong L\oplus K$ and $\ker(l')\cong\ker(l)\oplus K$. \par
\textup{(d)} Assume that an object $C\in\sA$ has an\/ $\sL$\+cover,
and let\/ $l'\:L'\rarrow C$ be an\/ $\sL$\+precover.
 Then the morphism~$l'$ is an\/ $\sL$\+cover if and only if the object
$L'$ has no nonzero direct summands contained in\/ $\ker(l')$.
\end{lem}

\begin{proof}
 Part~(a) is~\cite[Proposition~2.1.3 or~2.1.4]{Xu}.
 Part~(b) is known as Wakamatsu lemma; this
is~\cite[Lemma~2.1.1 or~2.1.2]{Xu}.
 Part~(c) is~\cite[Proposition~1.2.2 or Theorem~1.2.7]{Xu}, and part~(d)
is~\cite[Corollary~1.2.3 or~1.2.8]{Xu}.
\end{proof}

 Let $(\sL,\sE)$ be a cotorsion pair in~$\sA$.
 If $c\:L\rarrow C$ is an epimorphism in $\sA$ with $L\in\sL$ and
the object $\ker(c)\in\sA$ has a special $\sE$\+preenvelope, then
the object $C$ has a special $\sL$\+precover.
 If $b\:B\rarrow E$ is a monomorphism in $\sA$ with $E\in\sE$ and
the object $\coker(b)\in\sA$ has a special $\sL$\+precover, then
the object $B$ has a special $\sE$\+preenvelope.
 In particular, if there are enough injective and projective objects
in $\sA$, then, given a cotorsion pair $(\sL,\sE)$ in $\sA$, every object
of $\sA$ has a special $\sL$\+precover if and only if every object of
$\sA$ has a special $\sE$\+preenvelope.
 These results are known as Salce lemmas~\cite{Sal}.
 A cotorsion pair $(\sL,\sE)$ in $\sA$ is called \emph{complete} if
every object of $\sA$ has a special $\sL$\+precover and a special
$\sE$\+preenvelope.

\begin{lem} \label{cotorsion-kernel-covers-lemma}
 Let $(\sL,\sE)$ be a complete cotorsion pair in an abelian
category\/ $\sA$, and let $E\in\sE\subset\sA$ be an object.
 Then a morphism $l\:L\rarrow E$ in\/ $\sA$ is an\/ $\sL$\+cover
if and only if it is an $\sL\cap\sE$\+cover.
\end{lem}

\begin{proof}
 Since the cotorsion pair $(\sL,\sE)$ is complete in $\sA$, every
object of $\sA$ has a special $\sL$\+precover, which is, in particular,
an epic $\sL$\+precover.
 It follows that all the $\sL$\+precovers in $\sA$ are epic.

 Assume that $l$~is an $\sL$\+cover.
 Then, by Lemma~\ref{precovers-lemma}(b), the morphism~$l$ is
a special $\sL$\+precover; so its kernel belongs to~$\sE$.
 Since the class $\sE$ is closed under extensions in $\sA$, it follows
that $L\in\sL\cap\sE$.
 Therefore, $l$~is an $\sL\cap\sE$\+cover.

 Assume that $l$~is an $\sL\cap\sE$\+cover.
 Let $l'\:L'\rarrow E$ be a special $\sL$\+precover of the object $E$
in~$\sA$.
 Following the above argument, we have $L'\in\sL\cap\sE$; so $l'$~is
also an $\sL\cap\sE$\+precover of~$E$.
 According to Lemma~\ref{precovers-lemma}(c) applied to
the class of objects $\sL\cap\sE\subset\sA$, the object $\ker(l)$ is
a direct summand of $\ker(l')$.
 Hence $\ker(l)\in\sE$.
 So $l$~is a special $\sL$\+precover of $E$ in~$\sA$.
 In particular, by Lemma~\ref{precovers-lemma}(a), $l$~is
an $\sL$\+precover.
 Since $l$~is an $\sL\cap\sE$\+cover, we can conclude that
$l$~is an $\sL$\+cover.
\end{proof}

\begin{lem} \label{hereditary-complete-covers-lemma}
 Let $(\sL,\sE)$ be a hereditary complete cotorsion pair in an abelian
category\/~$\sA$.
 Assume that every object of\/ $\sE$ has an\/ $\sL$\+cover in\/~$\sA$.
 Then every object of\/ $\sA$ has an\/ $\sL$\+cover.
\end{lem}

\begin{proof}
 Let $A$ be an object in~$\sA$.
 By assumption, $A$ has a special $\sE$\+preenvelope $a\:A\rarrow E$.
 Set $L=\coker(a)$; then we have a short exact sequence
$0\rarrow A\rarrow E\rarrow L\rarrow0$ in $\sA$ with $E\in\sE$ and
$L\in\sL$.
 By assumption, the object $E$ has an $\sL$\+cover $m\:M\rarrow E$
in~$\sA$.
 Set $F=\ker(m)$; by Lemma~\ref{precovers-lemma}(b), we have $F\in\sE$.
 Let $K$ denote the kernel of the composition of epimorphisms
$M\rarrow E\rarrow L$; then we have $K\in\sL$, since $M$, $L\in\sL$
and the cotorsion pair $(\sL,\sE)$ is assumed to be hereditary.
 We have constructed a commutative diagram of four short exact sequences
$$
\begin{diagram}\dgARROWLENGTH=1.5em
\node[2]{0}\node{0} \\
\node{0}\arrow{e}\node{A}\arrow{n}\arrow{e,t}{a}\node{E}\arrow{n}
\arrow{e}\node{L}\arrow{e}\arrow{s,=}\node{0} \\
\node{0}\arrow{e}\node{K}\arrow{n,l}{k}\arrow{e}\node{M}\arrow{n,l}{m}
\arrow{e}\node{L}\arrow{e}\node{0} \\
\node[2]{F}\arrow{n}\arrow{e,=}\node{F}\arrow{n}\\
\node[2]{0}\arrow{n}\node{0}\arrow{n}
\end{diagram}
$$

 The morphism $k\:K\rarrow A$ is an epimorphism with the kernel
$F\in\sE$, so it is a special $\sL$\+precover.
 Let us show that it is an $\sL$\+cover.
 Let $h\:K\rarrow K$ be an endomorpism such that $kh=k$.
 Consider a pushout of the short exact sequence $0\rarrow K\rarrow M
\rarrow L\rarrow 0$ by the morphism~$h$ and denote it by
$0\rarrow K\rarrow N\rarrow L\rarrow 0$.
 We have $N\in\sL$, since $K$, $L\in\sL$ and the class $\sL$ is
closed under extensions in~$\sA$.
 In view of the universal property of the pushout, we have
a commutative diagram of two morphisms of short exact sequences
$$
\begin{diagram}\dgARROWLENGTH=1.5em
\node{0}\arrow{e}\node{A}\arrow{e,t}{a}\node{E}\arrow{e}\node{L}
\arrow{e}\arrow{s,=}\node{0} \\
\node{0}\arrow{e}\node{K}\arrow{n,l}{k}\arrow{e}\node{N}\arrow{n,l}{n}
\arrow{e}\node{L}\arrow{e}\arrow{s,=}\node{0} \\
\node{0}\arrow{e}\node{K}\arrow{n,l}{h}\arrow{e}\node{M}\arrow{n,l}{s}
\arrow{e}\node{L}\arrow{e}\node{0}
\end{diagram}
$$
with $kh=k$ and $ns=m$.
 Since the morphism $m\:M\rarrow E$ is an $\sL$\+cover and $N\in\sL$,
there exists a morphism $r'\:N\rarrow M$ such that $mr'=n$.
 Moreover, one has $mr's=ns=m$, hence $r's\:M\rarrow M$ is
automorphism.
 Setting $r=(r's)^{-1}r'\:N\rarrow M$, we have $rs=\id_M$ and
$mr=m(r's)^{-1}r'=mr'=n$.

 It follows from the latter equality that the morphism $r\:N\rarrow M$
forms a commutative triangle diagram with the epimorphisms $N\rarrow L$
and $M\rarrow L$.
 Passing to the kernels of these two epimorphisms, we obtain
a morphism $g\:K\rarrow K$ such that $gh=\id_K$.
 We have constructed a commutative diagram of two morphisms of short
exact sequences
$$
\begin{diagram}\dgARROWLENGTH=1.5em
\node{0}\arrow{e}\node{K}\arrow{e}\node{M}\arrow{e}\node{L}\arrow{e}
\arrow{s,=}\node{0} \\
\node{0}\arrow{e}\node{K}\arrow{n,l}{g}\arrow{e}\node{N}\arrow{n,l}{r}
\arrow{e}\node{L}\arrow{e}\arrow{s,=}\node{0} \\
\node{0}\arrow{e}\node{K}\arrow{n,l}{h}\arrow{e}\node{M}\arrow{n,l}{s}
\arrow{e}\node{L}\arrow{e}\node{0}
\end{diagram}
$$
whose composition is the identity endomorphism of the short exact
sequence $0\rarrow K\rarrow M\rarrow L\rarrow 0$.

 Thus we have shown that any endomorphism $h\:K\rarrow K$ such that
$kh=k$ is a (split) monomorphism.
 Furthermore, there is a commutative diagram of two morphisms of short
exact sequences
$$
\begin{diagram}\dgARROWLENGTH=1.5em
\node{0}\arrow{e}\node{A}\arrow{e,t}{a}\node{E}\arrow{e}\node{L}
\arrow{e}\arrow{s,=}\node{0} \\
\node{0}\arrow{e}\node{K}\arrow{n,l}{k}\arrow{e}\node{M}\arrow{n,l}{m}
\arrow{e}\node{L}\arrow{e}\arrow{s,=}\node{0} \\
\node{0}\arrow{e}\node{K}\arrow{n,l}{g}\arrow{e}\node{N}\arrow{n,l}{r}
\arrow{e}\node{L}\arrow{e}\node{0}
\end{diagram}
$$
where $kg=k$, because $mr=n$ (indeed, since $a$ is a monomorphism, it
suffices to show that $akg=ak$, which follows from the equality $mr=n$
and the commutativity of the left squares of our diagrams).

 Therefore, the morphism $g\:K\rarrow K$ is a (split) monomorphism, too,
and we can conclude that both $g$ and~$h$ are isomorphisms.
\end{proof}

\begin{cor} \label{cotorsion-pair-covers-cor}
 Let $(\sL,\sE)$ be a hereditary complete cotorsion pair in an abelian
category\/~$\sA$.
 Then the following three conditions are equivalent:
\begin{enumerate}
\item every object of\/ $\sA$ has an\/ $\sL$\+cover;
\item every object of\/ $\sE$ has an\/ $\sL$\+cover in\/~$\sA$;
\item every object of\/ $\sE$ has an\/ $\sL\cap\sE$\+cover.
\end{enumerate}
\end{cor}

\begin{proof}
 (1)\,$\Longleftrightarrow$\,(2) is
Lemma~\ref{hereditary-complete-covers-lemma};
 (2)\,$\Longleftrightarrow$\,(3) is
Lemma~\ref{cotorsion-kernel-covers-lemma}.
\end{proof}

\Section{The Tilting-Cotilting Correspondence}
\label{tilting-cotilting-secn}

 Let $\sA$ be a complete, cocomplete abelian category with a fixed
injective cogenerator $J\in\sA$.
 So there are enough injective objects in the category $\sA$, and
the class of all injective objects is $\sA_\inj=\Prod(J)\subset\sA$.

 Let $n\ge0$ be an integer, and let $T\in\sA$ be an object satisfying
the following two conditions:
\begin{enumerate}
\renewcommand{\theenumi}{\roman{enumi}}
\item the projective dimension of $T$ (as an object of $\sA$) does
not exceed~$n$, that is $\Ext^i_\sA(T,A)=0$ for all $A\in\sA$ and
$i>n$; and
\item for any set $X$, one has $\Ext^i_\sA(T,T^{(X)})=0$ for all $i>0$.
\end{enumerate}

 Denote by $\sE\subset\sA$ the class of all objects $E\in\sA$ such that
$\Ext^i_\sA(T,E)=0$ for all $i>0$.
 Notice that, by the definition, one has $\sA_\inj=\Prod_\sA(J)
\subset\sE$ and, by the condition~(ii), $\Add_\sA(T)\subset\sE$.

 Furthermore, for each integer $m\ge0$, denote by $\sL_m\subset\sA$
the class of all objects $L\in\sA$ for which there exists an exact
sequence of the form
$$
 0\lrarrow L\rarrow T^0\lrarrow T^1\lrarrow\dotsb\lrarrow T^m\lrarrow0
$$
in the category $\sA$ with the objects $T^m\in\Add(T)$.
 By the definition, $\Add(T)=\sL_0\subset\sL_1\subset\sL_2
\subset\dotsb\subset\sA$.
 According to~\cite[Lemma~3.2]{PS1}, one has $\sL_n=\sL_{n+1}=\sL_{n+2}
=\dotsb$ (so we set $\sL=\sL_n$) and $\sL\cap\sE=\Add(T)\subset\sA$.

 According to~\cite[Theorem~3.4]{PS1}, every object of $\sE$ is a quotient
of an object from $\Add(T)$ in $\sA$ if and only if every object of $\sA$
is a quotient of an object from~$\sL$.
 If this is the case, we say that the object 
$T\in\sA$ is \emph{$n$\+tilting}.
 For an $n$\+tilting object $T$, the pair of classes of objects
$(\sL,\sE)$ in $\sA$ is a hereditary complete cotorsion pair, called
the \emph{$n$\+tilting cotorsion pair} associated with~$T$.

 Let $\sB$ be a complete, cocomplete abelian category with a fixed
projective generator $P\in\sB$.
 So there are enough projective objects in $\sB$, and one has
$\sB_\proj=\Add(P)\subset\sB$.

 The definition of an \emph{$n$\+cotilting object} $W\in\sB$ is dual
to the above definition of an $n$\+tilting object.
 In other words, an object $W\in\sB$ is said to be \emph{$n$\+cotilting}
if the object $W^\sop$ is $n$\+tilting in the abelian category
$\sB^\sop$ opposite to~$\sB$.

 Specifically, this means, first of all, that the two conditions dual
to~(i) and~(ii) have to be satisfied:
\begin{enumerate}
\renewcommand{\theenumi}{\roman{enumi}*}
\item the injective dimension of $W$ (as an object of $\sB$) does
not exceed~$n$, that is $\Ext^i_\sB(B,W)=0$ for all $B\in\sB$ and
$i>n$; and
\item for any set $X$, one has $\Ext^i_\sB(W^X,W)=0$ for all $i>0$.
\end{enumerate}
 On top of that, denoting by $\sF\subset\sB$ the class of all objects
$F\in\sB$ such that $\Ext^i_\sB(F,W)=0$ for all $i>0$, it is required
that every object of $\sF$ should be a subobject of an object from
$\Prod(W)$ in~$\sB$.

 The following theorem from~\cite{PS1} describes the phenomenon of
\emph{$n$\+tilting-cotilting correspondence}.

\begin{thm} \label{tilting-cotilting-thm}
 There is a bijective correspondence between (the equivalence classes of)
complete, cocomplete abelian categories\/ $\sA$ with an injective
cogenerator $J$ and an $n$\+tilting object $T\in\sA$, and
(the equivalence classes of) complete, cocomplete abelian categories\/
$\sB$ with a projective generator $P$ and an $n$\+cotilting object
$W\in\sB$.
 The abelian categories\/ $\sA$ and\/ $\sB$ corresponding to each other
under this correspondence are connected by the following structures:
\par
\textup{(a)} there is a pair of adjoint functors between\/ $\sA$ and\/
$\sB$, with a left adjoint functor\/ $\Phi\:\sB\rarrow\sA$ and a right
adjoint functor\/ $\Psi\:\sA\rarrow\sB$; \par
\textup{(b)} one has\/ $\Phi(\sF)\subset\sE$ and\/ $\Psi(\sE)\subset\sF$;
the restrictions of the functors\/ $\Phi$ and\/ $\Psi$ are mutually
inverse equivalences between the full subcategories\/ $\sE\subset\sA$
and\/ $\sF\subset\sB$; \par
\textup{(c)} the full subcategory\/ $\sE\subset\sA$ is closed under
extensions and the cokernels of monomorphisms, while the full
subcategory\/ $\sF\subset\sB$ is closed under extensions and the kernels
of epimorphisms; hence they inherit exact category structures (in
Quillen's sense) from their ambient abelian categories; the equivalence
of categories\/ $\sE\cong\sF$ provided by the functors\/ $\Phi$ and\/
$\Psi$ is an equivalence of exact categories; in other words,
the functor\/ $\Phi$ preserves exactness of short exact sequences of
objects from\/ $\sF$, and the functor\/ $\Psi$ preserves exactness of
short exact sequences of objects from\/~$\sE$; \par
\textup{(d)} both the full subcategories\/ $\sE\subset\sA$ and\/
$\sF\subset\sB$ are closed under both the products and coproducts in
their ambient abelian categories; the functor\/ $\Phi\:\sB\rarrow\sA$
preserves the products (and coproducts) of objects from\/ $\sF$, while
the functor\/ $\Psi\:\sA\rarrow\sB$ preserves the (products and)
coproducts of objects from\/~$\sE$; \par
\textup{(e)} under the equivalence of exact categories\/ $\sE\cong\sF$,
the injective cogenerator $J\in\sE\subset\sA$ corresponds to
the $n$\+cotilting object $W\in\sF\subset\sB$, and the $n$\+tilting
object $T\in\sE\subset\sA$ corresponds to the projective generator
$P\in\sF\subset\sB$; \par
\textup{(f)} there are enough projective and injective objects in
the exact category\/ $\sE\cong\sF$; the full subcategories of
projectives and injectives in\/ $\sE$ are\/ $\sE_\proj=\Add(T)$ and\/
$\sE_\inj=\sA_\inj=\Prod(J)$, while the full subcategories of projectives
and injectives in\/ $\sF$ are\/ $\sF_\proj=\sB_\proj=\Add(P)$ and
$\sF_\inj=\Prod(W)$; \par
\textup{(g)} the equivalence of exact categories\/ $\sA\supset\sE\cong
\sF\subset\sB$ can be extended to a triangulated equivalence between
the derived categories\/ $\sD^\star(\sA)\cong\sD^\star(\sB)$, which
exists for any conventional derived category symbol\/
$\star=\bb$, $+$, $-$, or~$\varnothing$.
\end{thm}

\begin{proof}
 The bijective correspondence is constructed
in~\cite[Corollary~4.12]{PS1} (based on~\cite[Theorems~4.10
and~4.11]{PS1}), and the assertions~(e\+-f) are a part of that
construction (cf.~\cite[Proposition~2.6 and Theorem~3.4]{PS1}).
 The adjoint functors $\Phi$ and $\Psi$ are described
in~\cite[beginning of Section~5]{PS1}, and parts~(b\+-c) are also
explained there.
 Part~(d) is~\cite[Lemma~5.3 and Remark~5.4]{PS1}.
 Part~(g) is~\cite[Proposition~4.2 and/or Corollary~5.6]{PS1}.
\end{proof}

 The following characterization of the $n$\+tilting-cotilting
correspondence situations will be useful in Section~\ref{ring-epi-secn}.
 It may also be of an independent interest.
 
\begin{prop} \label{tilting-from-derived-equivalence}
 Let\/ $\sA$ be a complete, cocomplete abelian category with
an injective cogenerator $J$, and let\/ $\sB$ be a complete, cocomplete
abelian category with a projective generator~$P$.
 Suppose that there is a derived equivalence\/ $\sD^\bb(\sA)\cong
\sD^\bb(\sB)$ taking the object $J\in\sA$ to an object
$W\in\sB\subset\sD^\bb(\sB)$ and the object $P\in\sB$ to an object
$T\in\sA\subset\sD^\bb(\sA)$.
 Then, for any integer $n\ge0$, the following conditions are equivalent:
\begin{enumerate}
\renewcommand{\theenumi}{\Roman{enumi}}
\item the projective dimension of the object $T$ in the category\/ $\sA$
does not exceed~$n$;
\item the injective dimension of the object $W$ in the category\/ $\sB$
does not exceed~$n$;
\item the standard t\+structures on the derived categories\/
$\sD^\bb(\sA)$ and\/ $\sD^\bb(\sB)$, viewed as two t\+structures on
the same triangulated category\/ $\sD$ using the triangulated
equivalence\/ $\sD^\bb(\sA)\cong\sD^\bb(\sB)$, satisfy the inclusion\/
$\sD^{\bb,\le0}(\sA)\subset\sD^{\bb,\le n}(\sB)$, or equivalently,
$\sD^{\bb,\ge n}(\sB)\subset\sD^{\bb,\ge0}(\sA)$.
\end{enumerate}
 If any one of these conditions is satisfied, then the object $T\in\sA$
is $n$\+tilting; the object $W\in\sB$ is $n$\+cotilting; and moreover,
the abelian category\/ $\sA$ with the injective cogenerator $J$ and
the $n$\+tilting object $T$ and the abelian category\/ $\sB$ with
the projective generator $P$ and the $n$\+cotilting object $W$ are
connected by the $n$\+tilting-cotilting correspondence.
 The $n$\+tilting class\/ $\sE\subset\sA$ is the intersection
$\sA\cap\sB\subset\sD=\sD^\bb(\sA)=\sD^\bb(\sB)$ viewed as a full
subcategory in\/ $\sA$, and the $n$\+cotilting class\/ $\sF\subset\sB$
is the same intersection $\sB\cap\sA\subset\sD$ viewed as a full
subcategory in\/ $\sB$ (hence the equivalence of exact categories\/
$\sE\cong\sF$).
 The functor\/ $\Psi\:\sA\rarrow\sB$ assigns to an object $A\in\sA$
the degree-zero cohomology of the related complex in\/ $\sD^\bb(\sB)$,
and the functor\/ $\Phi\:\sB\rarrow\sA$ assigns to an object $B\in\sB$
the degree-zero cohomology of the related complex in\/ $\sD^\bb(\sA)$,
that is, $\Psi(A)=H_\sB^0(A)$ and\/ $\Phi(B)=H_\sA^0(B)$.
\end{prop}

\begin{proof}
 This is essentially the material of~\cite[Sections~2 and~4]{PS1}
(the description of the functors $\Phi$ and $\Psi$ can be found
in the beginning of~\cite[Section~5]{PS1}).
 So we only give a brief sketch of the argument.

 Notice, first of all, that the inclusions
$\sD^{\bb,\le0}(\sB)\subset\sD^{\bb,\le0}(\sA)$ and
$\sD^{\bb,\ge0}(\sA)\subset\sD^{\bb,\ge0}(\sB)$ always hold in our
assumptions, because an object $Z\in\sD$ belongs to
$\sD^{\bb,\ge0}(\sB)$ if and only if $\Hom_\sD(P,Z[i])=0$ for all $i<0$,
while one has $\Hom_\sD(S,Z[i])=0$ for all $Z\in\sD^{\bb,\ge0}(\sA)$,
all $i<0$, and all $S\in\sA$ (in particular, for $S=T$).

 In the same way one shows that the two inclusions in~(III) (which are
obviously equivalent to each other) are equivalent to~(I) on
the one hand and to~(II) on the other hand, (I)\,$\Longleftrightarrow$%
\,\allowbreak(III)\,$\Longleftrightarrow$\,(II).
 Indeed, an object $Z\in\sD$ belongs to $\sD^{\bb,\le n}(\sB)$
if and only if $\Hom_\sD(P,Z[i])=0$ for all $i>n$, while the projective
dimension of $T$ in $\sA$ does not exceed~$n$ if and only if
$\Hom_\sD(T,Z[i])=0$ for all $Z\in\sD^{\bb,\le0}(\sA)$ and all $i>n$.
 The argument for $W$ is similar.

 The inclusion $\sA\rarrow\sD^\bb(\sA)$ preserves coproducts, because
the coproduct functors are exact in $\sA$; and the inclusion
$\sB\rarrow\sD^\bb(\sB)$ preserves products, because the product functors
are exact in~$\sB$.
 Furthermore, we have $\sA\cap\sB=\sA\cap\sD^{\bb,\le0}(\sB)\subset\sD$,
since $\sB=\sD^{\bb,\le0}(\sB)\cap\sD^{\bb,\ge0}(\sB)$ and
$\sA\subset\sD^{\bb,\ge0}(\sA)\subset\sD^{\bb,\ge0}(\sB)$.
 The full subcategory $\sD^{\bb,\le0}(\sB)$ is closed under coproducts
in $\sD$ (those coproducts that exist in $\sD$), because the left part
of any t\+structure is closed under coproducts.
 Hence the full subcategory $\sA\cap\sB$ is closed under coproducts
in $\sD$, and consequently in $\sA$ and $\sB$.
 Similary, the full subcategory $\sA\cap\sB$ is closed under products
in $\sD$, and consequently in $\sA$ and~$\sB$.
 So the products and coproducts of objects of $\sE$ computed in $\sA$
agree with the products and coproducts of objects of $\sF$
computed in~$\sB$.
 (Cf.~\cite[Lemma~5.3 and Remark~5.4]{PS1}.)

 Now we can see that $\Ext_\sA^i(T,T^{(X)})=
\Hom_{\sD^\bb(\sA)}(T,T^{(X)}[i])=\Hom_{\sD^\bb(\sB)}(P,P^{(X)}[i])
\allowbreak=0$ for all $i>0$, and similarly $\Ext_\sB^i(W^X,W)=0$
for all $i>0$ and all sets~$X$.
 This proves the $n$\+tilting axiom~(ii) for $T$ and
the $n$\+cotilting axiom~(ii*) for $W$; while the axioms~(i)
and~(i*) are provided by the conditions~(I) and~(II).
 It remains to apply~\cite[Proposition~2.5 and Corollary~4.4(b)]{PS1}
in order to conclude that the object $T\in\sA$ is $n$\+tilting
and the object $W\in\sB$ is $n$\+cotilting.
 It is also clear from the construction of the $n$\+tilting-cotilting
correspondence in~\cite[Theorems~4.10\+-4.11 and Corollary~4.12]{PS1}
that the triples $(\sA,J,T)$ and $(\sB,P,W)$ are connected by such
correspondence.
\end{proof} 

\begin{rem}
 Given a complete, cocomplete abelian category $\sA$ with an injective
cogenerator and an $n$\+tilting object $T$, the related abelian category
$\sB$ can be described as the category $\sB=\boT_T\modl$ of modules over
the monad $\boT_T\:X\longmapsto\Hom_\sA(T,T^{(X)})$.
 The functors $\Phi$ and $\Psi$ from
Section~\ref{generalized-tilting-subsecn}
can be identified with the functors $\Phi$ and $\Psi$ from
Theorem~\ref{tilting-cotilting-thm} in this case~\cite[Remark~6.6]{PS1}.

 Dually, given a complete, cocomplete abelian category $\sB$ with
a projective generator and an $n$\+cotilting object $W$, the related
abelian category $\sA$ can be described as the opposite category
$\sA=\boT_{W^\sop}\modl^\sop$ to the category of modules over the monad 
$\boT_{W^\sop}\:X\longmapsto\Hom_\sB(W^X,W)$ (cf.~\cite[Section~1]{PS2}).
\end{rem}

\begin{exs} \label{tilting-subcategory-in-modules-exs}
 Suppose that there is an associative ring $A$ such that the abelian
category $\sA$ can be embedded into $A\modl$ as a full subcategory
closed under coproducts.
 So, in particular, the $n$\+tilting object $T\in\sA$ can be viewed
as a left $A$\+module.
 Then it follows from~\cite[Theorem~7.1 or~9.9]{PS1} that the abelian
category $\sB$ can be described as the category of left contramodules
$\R\contra$ over the topological ring $\R=\Hom_A(T,T)^\rop$
from Examples~\ref{module-endomorphism-topologies}\,(1), (2) or~(4).
 Further examples of classes of abelian categories $\sA$ for which
the category $\sB$ admits such a description are discussed
in~\cite[Sections~9\+-10]{PS1} and~\cite[Section~3]{PS3} (see
Examples~\ref{further-top-agreeable-examples}).
\end{exs}

\Section{Direct Limits in Categorical Tilting Theory}

 In this section we discuss the properties of direct limits in
the $n$\+tilting-cotilting correspondence context.
 We start with the case of the direct limits indexed by the poset of
natural numbers.

\begin{lem} \label{countable-direct-limits-phi-psi-lemma}
 In the context of the $n$\+tilting-cotilting correspondence, assume
that countable direct limits are exact in the abelian category\/~$\sA$.
 Then both the full subcategories\/ $\sE\subset\sA$ and\/
$\sF\subset\sB$ are closed under countable direct limits in their
ambient abelian categories, and the functor\/ $\Psi\:\sA\rarrow\sB$
preserves countable direct limits of objects from\/~$\sE$.
 Furthermore, for any sequence of objects and morphisms
$F_1\rarrow F_2\rarrow F_3\rarrow\dotsb$ with $F_i\in\sF$, the short
sequence $0\rarrow\coprod_{i=1}^\infty F_i\rarrow
\coprod_{i=1}^\infty F_i\rarrow\varinjlim_{i\ge1}F_i\rarrow0$ with the map
$\id-\mathit{shift}\:\coprod_iF_i\rarrow\coprod_iF_i$ is exact
in\/~$\sB$.
 The functors of countable direct limit are exact in the exact
category~$\sF$.
\end{lem}

\begin{proof}
 The argument resembles the proof of
Proposition~\ref{THEC-Bass-objects-in-A-and-B}.
 For any sequence of objects and morphisms $B_1\rarrow B_2\rarrow B_3
\rarrow\dotsb$ in an abelian category $\sB$ with countable
coproducts, the short sequence $\coprod_{i=1}^\infty B_i\rarrow
\coprod_{i=1}^\infty B_i\rarrow\varinjlim_{i\ge1}B_i\rarrow0$ is right
exact in~$\sB$.
 Moreover, for any sequence of objects and morphisms $A_1\rarrow A_2
\rarrow A_3 \rarrow\dotsb$ in an abelian category $\sA$ with exact
countable direct limits, the short sequence $0\rarrow
\coprod_{i=1}^\infty A_i\rarrow\coprod_{i=1}^\infty A_i\rarrow
\varinjlim_{i\ge1}A_i\rarrow0$ is exact in $\sA$
(see Example~\ref{telescope-pure-exact-example}).
 In particular, for any sequence of objects and morphisms $E_1\rarrow
E_2\rarrow E_3\rarrow\dotsb$ with $E_i\in\sE$, the short sequence
$0\rarrow\coprod_iE_i\rarrow\coprod_iE_i\rarrow\varinjlim_iE_i\rarrow0$
is exact in~$\sA$.
 Hence it follows that $\varinjlim_i E_i\in\sE$, because the full
subcategory $\sE\subset\sA$ is closed under coproducts and
the cokernels of monomorphisms.

 The functor $\Phi$, being a left adjoint, preserves all colimits.
 Thus, for any sequence of objects and morphisms $F_1\rarrow F_2
\rarrow F_3\rarrow\dotsb$ in $\sF$, the short sequence $0\rarrow
\Phi(\coprod_i F_i)\rarrow\Phi(\coprod_i F_i)\rarrow
\Phi(\varinjlim_iF_i)\rarrow0$, being isomorphic to the short sequence
$0\rarrow\coprod_i\Phi(F_i)\rarrow\coprod_i\Phi(F_i)\rarrow
\varinjlim_i\Phi(F_i)\rarrow0$, is exact in~$\sA$.
 This is a short exact sequence in $\sA$ with all the three terms
belonging to~$\sE$, so the functor $\Psi$ transforms it into a short
exact sequence in $\sB$ with all the three terms belonging to~$\sF$.
 We have a natural (adjunction) morphism from the right exact sequence
$\coprod_iF_i\rarrow\coprod_iF_i\rarrow\varinjlim_iF_i\rarrow0$
to the exact sequence $0\rarrow\Psi\Phi(\coprod_i F_i)\rarrow
\Psi\Phi(\coprod_i F_i)\rarrow\Psi\Phi(\varinjlim_iF_i)\rarrow0$,
which is an isomorphism at the first two terms, and therefore at
the third term, too.
 Hence the object $\varinjlim_iF_i\cong\Psi(\varinjlim_i\Phi(F_i))$
belongs to $\sF$ and the short sequence $0\rarrow\coprod_iF_i\rarrow
\coprod_iF_i\rarrow\varinjlim_iF_i\rarrow0$ is exact.
 Since the coproduct functors are exact in $\sF$ (because they are
exact in~$\sE$) and the cokernel of an admissible monomorphism is an
exact functor, it follows that the functors of countable direct limit
are exact in~$\sF$.
 The functor $\Psi|_\sE\:\sE\rarrow\sB$ preserves countable direct
limits, because both the equivalence of categories $\sE\cong\sF$ and
the inclusion functor $\sF\rarrow\sB$ do.
 This proves all the assertions of the lemma.
\end{proof}

\begin{cor} \label{countable-direct-limits-closedness-cor}
 In the context of the $n$\+tilting-cotilting correspondence, assume
that countable direct limits are exact in the abelian category\/~$\sA$.
 Then the following three conditions are equivalent:
\begin{enumerate}
\renewcommand{\theenumi}{\roman{enumi}}
\item the full subcategory\/ $\sL$ is closed under countable
direct limits in\/~$\sA$; \par
\item the class of objects\/ $\Add(T)$ is closed under countable
direct limits in\/~$\sA$; \par
\item the class of all projective objects\/ $\sB_\proj$ is closed
under countable direct limits in\/~$\sB$.
\end{enumerate}
\end{cor}

\begin{proof}
 (i)\,$\Longrightarrow$\,(ii)
 According to Lemma~\ref{countable-direct-limits-phi-psi-lemma},
the class $\sE$ is closed under countable direct limits in~$\sA$.
 Hence, if the class $\sL$ is closed under countable direct limits,
too, then so is the class $\sL\cap\sE=\Add(T)$.

 (ii)\,$\Longleftrightarrow$\,(iii)
 By the same lemma, the equivalence of categories $\sE\cong\sF$
transforms countable direct limits of objects from $\sE$ computed
in $\sA$ to countable direct limits of objects from $\sF$ computed
in~$\sB$.
 Thus the class $\sB_\proj=\Psi(\Add(T))\subset\sF$ is closed under
countable direct limits in $\sB$ if and only if the class
$\Add(T)\subset\sE$ is closed under countable direct limits in~$\sA$.

 (ii)\,$\Longrightarrow$\,(i)
 Given an object $L\in\sL$, an exact sequence $0\rarrow L\rarrow T^0
\rarrow\dotsb\rarrow T^n\rarrow0$ with $T^j\in\Add(T)$ can be
constructed in the following way.
 Let $L\rarrow E$ be a special $\sE$\+preenvelope of~$L$; then we
have a short exact sequence $0\rarrow L\rarrow E\rarrow M\rarrow0$ 
with $E\in\sE$ and $M\in\sL$.
 Since the class $\sL$ is closed under extensions in $\sA$, we have
$E\in\sL\cap\sE=\Add(T)$.
 Set $T^0=E$ and $M^1=M$, and let $M^1\rarrow T^1$ be a special
$\sE$\+preenvelope of $M^1$, etc.
 Proceeding in this way, one obtains an exact sequence $0\rarrow L
\rarrow T^0\rarrow T^1\rarrow\dotsb\rarrow T^{n-1}\rarrow M^n\rarrow0$
with $M^n\in\sL$; and one also has $M^n\in\sE$ by cohomological
dimension shifting, since the projective dimension of $T$ does not
exceed~$n$.
 It remains to set $T^n=M^n$.
 Conversely, in any exact sequence $0\rarrow L\rarrow T^0\rarrow T^1
\rarrow\dotsb\rarrow T^n\rarrow0$ with $L\in\sL$ and $T^j\in\Add(T)$,
the objects of cocycles belong to $\sL$, since the class $\sL$,
being the left class in a hereditary cotorsion pair, is closed under
the kernels of epimorphisms.

 Now, for any two objects $A'$ and $A''\in\sA$, their special
$\sE$\+preenvelopes $A'\rarrow E'$ and $A''\rarrow E''$, and
a morphism $A'\rarrow A''$, there is a morphism $E'\rarrow E''$
forming a commutative triangle diagram with the composition
$A'\rarrow A''\rarrow E''$.
 Using this observation, for any sequence of objects and morphisms
$L_1\rarrow L_2\rarrow L_3\rarrow\dotsb$ in $\sL$ and any exact
sequences $0\rarrow L_i\rarrow T_i^0\rarrow\dotsb\rarrow T_i^n\rarrow0$
with $T_i^j\in\Add(T)$, one can extend the sequence of morphisms
$\dotsb\rarrow L_i\rarrow L_{i+1}\rarrow\dotsb$ to a sequence of
morphisms of exact sequences
$\dotsb\rarrow(0\to L_i\to T_i^0\to\dotsb\to T_i^n\to0)\rarrow
(0\to L_{i+1}\to T_{i+1}^0\to\dotsb\to T_{i+1}^n\to0)\rarrow\dotsb$.
 Passing to the direct limit, we obtain an exact sequence
$$
 0\lrarrow\varinjlim\nolimits_{i\ge1}L_i\lrarrow
 \varinjlim\nolimits_{i\ge1}T_i^0\lrarrow\dotsb\lrarrow
 \varinjlim\nolimits_{i\ge1}T_i^n\lrarrow0
$$
in the abelian category~$\sA$.
 Since $\varinjlim_iT_i^j\in\Add(T)$ for all $j=0$,~\dots, $n$, it
follows that $\varinjlim_iL_i\in\sL$ by the definition.
\end{proof}

 The following proposition provides a generalization to uncountable
direct limits.

\begin{prop} \label{direct-limits-phi-psi-prop}
 In the context of the $n$\+tilting-cotilting correspondence, assume
that direct limits are exact in the abelian category\/~$\sA$.
 Then both the full subcategories\/ $\sE$ and\/ $\sF$ are closed under
direct limits in their ambient abelian categories\/ $\sA$ and\/ $\sB$,
and the functor\/ $\Psi\:\sA\rarrow\sB$ preserves direct limits of
objects from\/~$\sE$.
 The functors of direct limit are exact in the exact category\/~$\sF$.
\end{prop}

\begin{proof}
 The argument resembles the proof of
Proposition~\ref{self-pure-projective-direct-limits}.
 Let $E\:\Theta\rarrow\sE$ be a diagram in the exact category $\sE$
indexed by a directed poset~$\Theta$.
 Then the augmented bar-complex~\eqref{augmented-bar-complex} (from
Example~\ref{bar-complex-pure-exact-example}) for the diagram~$E$
is an unbounded resolution of an object of $\sA$ by objects of $\sE$
(since the full subcategory $\sE\subset\sA$ is closed under coproducts).
 Since the full subcategory $\sE\subset\sA$ is defined as the class of
all objects $E\in\sA$ such that $\Ext^i_\sA(T,E)=0$ for all $i>0$, and
the tilting object $T\in\sA$ has finite projective dimension, a simple
cohomological dimension shifting argument shows that
$\varinjlim_{\theta\in\Theta} E(\theta)\in\sE$.
 Moreover, all the objects of cycles of the exact
complex~\eqref{augmented-bar-complex} for the diagram $E$
also belong to~$\sE$.
 So this complex is exact in the exact category~$\sE$.

 Applying the functor $\Psi$ to the augmented bar-complex for
the diagram~$E$, we get an exact complex in the category $\sF$,
which coincides, except possibly at his rightmost term, with
the augmented bar-complex~\eqref{bar-complex-B} for the diagram
$\Psi\circ E$ in $\sB$ (because both the equivalence of categories
$\sE\cong\sF$ and the inclusion functor $\sF\rarrow\sB$ preserve
coproducts).
 Since the bar-complex of any diagram in a cocomplete abelian category
is exact at its rightmost term, it follows that the natural
morphism $\varinjlim_\theta\Psi(E(\theta))\rarrow
\Psi(\varinjlim_\theta E(\theta))$ is an isomorphism and
$\varinjlim_\theta\Psi(E(\theta))\in\sF$.
 As any diagram in $\sF$ can be obtained by applying the functor
$\Psi$ to a diagram in $\sE$, we can conclude that the full subcategory
$\sF\subset\sB$ is also closed under direct limits, and
the bar-complexes~\eqref{bar-complex-B} computing such direct
limits in $\sF$ are exact.
 Exactness of the direct limit functors in $\sF$ easily follows.
\end{proof}

\begin{cor} \label{uncountable-direct-limit-closedness-cor}
 In the context of the $n$\+tilting-cotilting correspondence, assume
that direct limits are exact in the abelian category\/~$\sA$.
 Then the following three conditions are equivalent:
\begin{enumerate}
\renewcommand{\theenumi}{\roman{enumi}}
\item the full subcategory\/ $\sL$ is closed under direct limits
in\/~$\sA$; \par
\item the class of objects\/ $\Add(T)$ is closed under direct limits
in\/~$\sA$; \par
\item the class of all projective objects\/ $\sB_\proj$ is closed
under direct limits in\/~$\sB$.
\end{enumerate}
\end{cor}

\begin{proof}
 Provable in the same way as
Corollary~\ref{countable-direct-limits-closedness-cor}, using
Proposition~\ref{direct-limits-phi-psi-prop} in place of
Lemma~\ref{countable-direct-limits-phi-psi-lemma}.
 Let us just say a few words about the implication
(ii)\,$\Longrightarrow$\,(i).

 In view of~\cite[Sections~1.6\+-1.7]{AR}, it suffices to show that
$\sL$ is closed under the direct limits of well-ordered chains
in $\sA$ (in fact, it suffices to consider direct limits indexed
by regular cardinals).
 Let us prove that $\sL$ is closed under $\lambda$\+indexed
direct limits for any ordinal~$\lambda$.

 Let $(L_i\to L_j)_{0\le i<j<\lambda}$ be a $\lambda$\+indexed
diagram in~$\sL$.
 Proceeding by transfinite induction in $0\le i<\lambda$, we construct
a $\lambda$\+indexed diagram of exact sequences $0\rarrow L_i\rarrow
T_i^0\rarrow\dotsb\rarrow T_i^n\rarrow0$ with $T_i^k\in\Add(T)$,
connected by morphisms of exact sequences for all $0\le i<j<\lambda$.

 The case $i=0$ is clear.
 Assume that the desired directed diagram of exact sequences has
been constructed for $0\le i<j<\alpha$, where $0<\alpha<\lambda$ is
some ordinal.
 Then we have an exact sequence
\begin{equation} \label{transfinite-induction-exact-sequence}
 0\lrarrow\varinjlim\nolimits_{i<\alpha} L_i
 \lrarrow\varinjlim\nolimits_{i<\alpha}T^0_i\lrarrow\dotsb
 \lrarrow\varinjlim\nolimits_{i<\alpha}T^n_i\lrarrow0
\end{equation}
in $\sA$ with $\varinjlim_{i<\alpha}T^k_i\in\Add(T)$ by~(ii)
for all $0\le k\le n$, hence $\varinjlim_{i<\alpha}L_i\in\sL$.
 Starting from the natural morphism $\varinjlim_{i<\alpha} L_i
\rarrow L_\alpha$ and arguing as in the proof of
Corollary~\ref{countable-direct-limits-closedness-cor}, we
construct a morphism from the exact
sequence~\eqref{transfinite-induction-exact-sequence} to
an exact sequence $0\rarrow L_\alpha\rarrow T^0_\alpha\rarrow
\dotsb\rarrow T^n_\alpha\rarrow0$ with $T^k_\alpha\in\Add(T)$.

 Having obtained the desired $\lambda$\+indexed diagram of
exact sequences, it remains to say that, in the exact sequence
$0\rarrow\varinjlim_{i<\lambda} L_i\rarrow\varinjlim_{i<\lambda}T^0_i
\rarrow\dotsb\rarrow\varinjlim_{i<\lambda}T^n_i\rarrow0$ in $\sA$,
the objects $\varinjlim_{i<\lambda}T^k_i$ belong to $\Add(T)$ by~(ii)
for all $0\le k\le n$.
 Hence $\varinjlim_{i<\lambda}L_i\in\sL$, so (i)~holds.
\end{proof}

\Section{When is the Left Tilting Class Covering?}
\label{tilting-covering-secn}

 In this section we prove
Theorem~\ref{introd-tilting-main-theorem} from the introduction.
 As in the previous sections, we start with weaker assumptions and
then gradually strengthen them.

\begin{prop} \label{L-covering-iff-R-properfect}
 In the context of the $n$\+tilting-cotilting correspondence,
the following four conditions are equivalent:
\begin{enumerate}
\item the class\/ $\sL$ is covering in\/~$\sA$;
\item every object of\/ $\sE$ has an\/ $\sL$\+cover in\/~$\sA$;
\item the class\/ $\Add(T)$ is covering in\/~$\sE$;
\item the class\/ $\sB_\proj$ is covering in\/~$\sF$.
\end{enumerate}

 Furthermore, assume that countable direct limits are exact in
the abelian category\/~$\sA$.
 Then the following six conditions~\textup{(5\+-10)} are equivalent:
\begin{enumerate}
\setcounter{enumi}{4}
\item any countable direct limit of copies of the tilting object $T$
has an\/ $\sL$\+cover in\/~$\sA$;
\item any countable direct limit of copies of the object $T$
has an\/ $\Add(T)$\+cover in\/~$\sA$;
\item any countable direct limit of copies of the projective generator
$P$ has a projective cover in\/~$\sB$;
\item any countable direct limit of copies of the tilting object $T$
in\/ $\sA$ belongs to\/~$\sL$;
\item any countable direct limit of copies of the object $T$ in\/ $\sA$
belongs to\/ $\Add(T)$;
\item any countable direct limit of copies of the projective generator
$P$ in\/ $\sB$ is projective.
\end{enumerate}

 Moreover, let us assume that countable direct limits are exact
in\/~$\sA$ and that\/ $\sB$ is the abelian category of left
contramodules over a complete, separated right linear topological
ring\/~$\R$.
 Consider the following six properties:
\begin{enumerate}
\setcounter{enumi}{10}
\item the object $T\in\sA$ has a perfect decomposition;
\item the topological ring\/ $\R$ is topologically left perfect;
\item the class\/ $\sB_\proj$ is closed under direct limits in\/~$\sB$;
\item the class\/ $\sB_\proj$ is covering in\/~$\sB$;
\item all descending chains of cyclic discrete right\/ $\R$\+modules
terminate;
\item all the discrete quotient rings of\/ $\R$ are left perfect.
\end{enumerate}
 Then the following implications hold:
$$
 \mathrm{(11)} \Longleftrightarrow \mathrm{(12)}
 \Longleftrightarrow \mathrm{(13)} \Longleftrightarrow \mathrm{(14)}
 \Longrightarrow \mathrm{(4)}
 \Longrightarrow \mathrm{(7)} \Longrightarrow \mathrm{(15)}
 \Longrightarrow \mathrm{(16)}.
$$
 If the topological ring\/ $\R$ satisfies one of the conditions~(a),
(b), (c), or~(d), then all the conditions~\textup{(1\+-16)} are
equivalent to each other.
 If the topological ring\/ $\R$ satisfies one of the conditions~(e),
(f), or~(g) of Section~\ref{seven-conditions-secn}, then the fifteen
conditions~\textup{(1\+-15)} are equivalent to each other.
\end{prop}

\begin{proof}
 (1)\,$\Longleftrightarrow$\,(2)\,$\Longleftrightarrow$\,(3) is
Corollary~\ref{cotorsion-pair-covers-cor}.

 (3)\,$\Longleftrightarrow$\,(4) holds in view of the equivalence of
categories $\sE\cong\sF$ taking the class $\Add(T)\subset\sE$ to
the class $\sB_\proj=\sF_\proj\subset\sF$ (see
Theorem~\ref{tilting-cotilting-thm}(b,f)).

 (5)\,$\Longleftrightarrow$\,(6)
 By Lemma~\ref{countable-direct-limits-phi-psi-lemma}, any countable
direct limit of copies of the object $T$ in $\sA$ belongs to~$\sE$.
 So Lemma~\ref{cotorsion-kernel-covers-lemma} applies.

 (6)\,$\Longleftrightarrow$\,(7) 
 The equivalence of categories $\sE\cong\sF$ identifies the class of
objects $\Add(T)\subset\sE$ with the class $\sB_\proj\subset\sF$.
 By Lemma~\ref{countable-direct-limits-phi-psi-lemma}, it also
identifies countable direct limits of copies of the object $T$ in
$\sA$ with countable direct limits of copies of the object $P$ in~$\sB$.

 (8)\,$\Longleftrightarrow$\,(9) holds, since any countable direct
limit of copies of $T$ in $\sA$ belongs to~$\sE$.

 (9)\,$\Longleftrightarrow$\,(10) is similar to
Corollary~\ref{countable-direct-limits-closedness-cor}\,%
(ii)\,$\Leftrightarrow$\,(iii).

 (6)\,$\Longleftrightarrow$\,(7)\,$\Longleftrightarrow$\,(9)\,%
$\Longleftrightarrow$\,(10) An $n$\+tilting object $T\in\sA$
satisfies THEC by Example~\ref{sigma-pure-rigid-example}\,(1),
so Corollary~\ref{small-THEC-corollary} is applicable.

 The implications (14)\,$\Longrightarrow$\,(4)\,$\Longrightarrow$\,(7)
and (13)$\Longrightarrow$\,(10) are obvious.
 So are the implications (2)\,$\Longrightarrow$\,(5) and
(3)\,$\Longrightarrow$\,(6), in view of
Lemma~\ref{countable-direct-limits-phi-psi-lemma}.

 (11)\,$\Longleftrightarrow$\,(12)\,$\Longleftrightarrow$\,(13)
is Corollary~\ref{contramodule-THEC-corollary}\,(1)\,%
$\Leftrightarrow$\,(2)\,$\Leftrightarrow$\,(3).

 (13)\,$\Longleftrightarrow$\,(14)
is~\cite[Theorem~14.1\,(iii$'$)\,$\Leftrightarrow$\,(ii)]{PS3}.

 (10)\,$\Longrightarrow$\,(15)\,$\Longrightarrow$\,(16) is
Corollary~\ref{contramodule-THEC-corollary}\,(6)\,$\Rightarrow$\,(9)\,%
$\Rightarrow$\,(10).

 This proves all the assertions of the proposition except the last two
(in which one of the conditions~(a), (b), (c), (d), (e), (f), or~(g) is
assumed).
 Now we assume~(d) (which is a common generalization of (a), (b),
and~(c)) and prove the related implications.

 (16)\,$\Longrightarrow$\,(14) If all the discrete quotient rings of
$\R$ are left perfect and (d)~is satisfied, then all left
$\R$\+contramodules have projective covers
by Corollary~\ref{a-b-c-d-THEC-main-corollary}\,(15)\,$\Rightarrow$\,(6)
or~\cite[Theorem~12.4\,(vi)\,$\Rightarrow$\,(ii)]{Pproperf}.

 (16)\,$\Longrightarrow$\,(13) Follows from
Corollary~\ref{a-b-c-d-THEC-main-corollary}\,(15)\,$\Rightarrow$\,(10)
or~\cite[Theorem~12.4\,(vi)\,$\Rightarrow$\,\allowbreak(iii)]{Pproperf}
(since the direct limits of projective contramodules are always flat).

 (16)\,$\Longrightarrow$\,(12) is
Corollary~\ref{a-b-c-d-THEC-main-corollary}\,(15)\,$\Rightarrow$\,(12).

 Finally, assuming that one of the conditions~(e), (f), or~(g) holds,
all the conditions~(11\+-15) are equivalent by
Corollary~\ref{a-b-c-d-THEC-main-corollary}\,(6)\,$\Leftrightarrow$\,%
(10)\,$\Leftrightarrow$\,(11)\,$\Leftrightarrow$\,(12)\,%
$\Leftrightarrow$\,(13).
\end{proof}

\begin{thm} \label{tilting-direct-limits-covers-thm}
 In the context of the $n$\+tilting-cotilting correspondence, assume
that\/ $\sA$ is a Grothendieck abelian category.
 Then the following conditions are equivalent:
\begin{enumerate}
\item the class\/ $\sL$ is covering in\/~$\sA$;
\item any direct limit of objects from\/ $\Add(T)$ has
an\/ $\sL$\+cover in\/~$\sA$;
\item the class\/ $\sL$ is closed under direct limits in\/~$\sA$;
\item the class\/ $\Add(T)$ is covering in\/~$\sA$;
\item any direct limit of objects from\/ $\Add(T)$ has
an\/ $\Add(T)$\+cover in\/~$\sA$;
\item the class\/ $\Add(T)$ is closed under direct limits in\/~$\sA$;
\item the class\/ $\sB_\proj$ is covering in\/~$\sB$;
\item any direct limit of projective objects has a projective cover
in\/~$\sB$;
\item the class\/ $\sB_\proj$ is closed under direct limits in\/~$\sB$.
\end{enumerate}

 Furthermore, assume that\/ $\sB$ is the abelian category of left
contramodules over a complete, separated right linear topological
ring\/~$\R$.
 Consider the following four properties:
\begin{enumerate}
\setcounter{enumi}{9}
\item the object $T\in\sA$ has a perfect decomposition;
\item the topological ring\/ $\R$ is topologically left perfect;
\item all descending chains of cyclic discrete right\/ $\R$\+modules
terminate;
\item all the discrete quotient rings of\/ $\R$ are left perfect.
\end{enumerate}
 Then the following implications hold:
$$
 \mathrm{(9)} \Longleftrightarrow \mathrm{(10)} \Longleftrightarrow
 \mathrm{(11)} \Longrightarrow \mathrm{(12)}
 \Longrightarrow \mathrm{(13)}.
$$
 If the topological ring\/ $\R$ satisfies one of
the conditions~(a), (b), (c), or~(d), then all
the conditions~\textup{(1\+-13)} are equivalent to each other.
 If the topological ring\/ $\R$ satisfies one of the conditions~(e),
(f), or~(g) of Section~\ref{seven-conditions-secn}, then the twelve
conditions~\textup{(1\+-12)} are equivalent to each other.
\end{thm}

\begin{proof}
 The implications (1)\,$\Longrightarrow$\,(2),
(4)\,$\Longrightarrow$\,(5), and (7)\,$\Longrightarrow$\,(8)
are obvious (as are the implications (3)\,$\Longrightarrow$\,(2),
(6)\,$\Longrightarrow$\,(5), and (9)\,$\Longrightarrow$\,(8)).

 The equivalences
(3)\,$\Longleftrightarrow$\,(6)\,$\Longleftrightarrow$\,(9) hold by
Corollary~\ref{uncountable-direct-limit-closedness-cor}.

 (3)\,$\Longrightarrow$\,(1) holds by
Theorem~\ref{direct-limits-imply-covers}, since the class $\sL$ is
(special) precovering in~$\sA$.

 (6)\,$\Longrightarrow$\,(4) is Example~\ref{Add(M)-precovers} and
Theorem~\ref{direct-limits-imply-covers}.

 (9)\,$\Longrightarrow$\,(7) is Example~\ref{projective-precovers}
and Theorem~\ref{direct-limits-imply-covers}.
 Notice that Theorem~\ref{direct-limits-imply-covers} requires
the category $\sA$ to be locally presentable for its applicability.
 An abelian category is locally presentable \emph{and} has exact direct
limit functors if and only if it is Grothendieck; that is why we assume
that $\sA$ is a Grothendieck category in the present theorem.

 (5)\,$\Longleftrightarrow$\,(6) is a particular case
of~\cite[Application~7.3]{BPS}.

 (5)\,$\Longleftrightarrow$\,(6)\,$\Longleftrightarrow$\,(8)\,%
$\Longleftrightarrow$\,(9)
 The object $T\in\sA$ is $\varinjlim$\+pure-rigid by
Example~\ref{indlim-pure-rigid-example}, since the $n$\+tilting
class $\sE\subset\sA$ is closed under direct limits by
Proposition~\ref{direct-limits-phi-psi-prop}.
 Therefore, Corollary~\ref{self-pure-projective-corollary}
is applicable.

 (2)\,$\Longleftrightarrow$\,(5) is
Lemma~\ref{cotorsion-kernel-covers-lemma}.

 This proves the first assertion of the theorem
(see also~\cite[Remark~7.4]{BPS} for a brief summary of this argument).
 The remaining implications are provided by
Corollary~\ref{a-b-c-d-THEC-main-corollary}\,(10\+-15) as well as by
Proposition~\ref{L-covering-iff-R-properfect}\,(11\+-16).
\end{proof}

\begin{proof}[{Proof of Theorem~\ref{introd-tilting-main-theorem}}]
 Follows from Proposition~\ref{L-covering-iff-R-properfect} and
Theorem~\ref{tilting-direct-limits-covers-thm}.
\end{proof}

\Section{Injective Ring Epimorphisms of Projective Dimension~$1$}
\label{ring-epi-secn}

 In this section we discuss a certain tilting-cotilting correspondence
situation associated with an injective homological ring epimorphism
satisfying additional conditions on the flat and projective dimension.
 
 We recall that a \emph{ring epimorphism} $u\:R\rarrow U$ is
a homomorphism of associative rings such that the multiplication
map $U\ot_RU\rarrow U$ is an isomorphism of $U$\+$U$\+bimodules.
 We refer to the book~\cite[Section~XI.1]{St} for background information
on ring epimorphisms, and to the paper~\cite{BP2} for more advanced
recent results.
 A ring epimorphism~$u$ is said to be \emph{homological} if
$\Tor^R_i(U,U)=0$ for all $i\ge1$.

 The two-term complex of $R$\+$R$\+bimodules $K^\bu=(R\to U)$ plays
a key role in the theory developed in~\cite{BP2}.
 In the present paper, we deal with \emph{injective} ring epimorphisms,
i.~e., ring epimorphisms~$u$ such that the map~$u$ is injective.
 In this case, the two-term complex of $R$\+$R$\+bimodules $K^\bu$ is
naturally quasi-isomorphic to the quotient bimodule~$U/R$.
 So we set $K=U/R$ and use $K$ in lieu of~$K^\bu$.
 
 We will denote by $\pd{}_RE$ the projective dimension of a left
$R$\+module $E$ and by $\fd E_R$ the flat dimension of a right
$R$\+module~$E$.
 For any injective homological ring epimorphism $u\:R\rarrow U$ 
such that $\pd{}_RU\le1$, the left $R$\+module $U\oplus K$ is
$1$\+tilting~\cite[Theorem~3.5]{ASan}.
 In this section we discuss a different tilting-cotilting correspondence
situation, in which $\sA\subset R\modl$ is a certain abelian subcategory.

 Let $u\:R\rarrow U$ be an injective homological ring epimorphism.
 A left $R$\+module $A$ is said to be \emph{$u$\+torsionfree} if it is
an $R$\+submodule of a left $U$\+module, or equivalently, if
the $R$\+module morphism $u\ot_R\id_A\:A\rarrow U\ot_RA$ is injective.
 The class of $u$\+torsionfree left $R$\+modules is closed under
submodules, direct sums, and products.
 Any left $R$\+module $A$ has a unique maximal $u$\+torsionfree
quotient module, which can be constructed as the image of
the $R$\+module morphism $A\rarrow U\ot_RA$.
 When $\fd U_R\le1$, the class of $u$\+torsionfree $R$\+modules is also
closed under extensions in $R\modl$ \cite[Lemma~2.7(a)]{BP2}.

 A left $R$\+module $B$ is said to be \emph{$u$\+divisible} if it is
a quotient $R$\+module of a left $U$\+module, or equivalently, if
the $R$\+module morphism $\Hom_R(u,\id_B)\:\Hom_R(U,B)\allowbreak
\rarrow B$ is surjective.
 (See~\cite[Remarks~1.2]{BP2} for a terminological discussion.)
 The class of all $u$\+divisible left $R$\+modules is closed under
epimorphic images, direct sums, and products.
 Any left $R$\+module $B$ has a unique maximal $u$\+divisible
submodule, which can be constructed as the image of the $R$\+module
morphism $\Hom_R(U,B)\rarrow B$.
 When $\pd{}_RU\le1$, the class of $u$\+divisible $R$\+modules is
also closed under extensions in $R\modl$ \cite[Lemma~2.7(b)]{BP2}.

 A left $R$\+module $M$ is called a \emph{$u$\+comodule} (or
a \emph{left $u$\+comodule}) if $U\ot_RM=0=\Tor^R_1(U,M)$.
 Assuming that $\fd U_R\le1$, the full subcategory $R\modl_{u\co}$
of left $u$\+comodules is closed under kernels, cokernels, extensions,
and direct sums in $R\modl$ \cite[Proposition~1.1]{GL}; so
$R\modl_{u\co}$ is an abelian category and the embedding functor
$R\modl_{u\co}\rarrow R\modl$ is exact.
 The embedding functor $R\modl_{u\co}\rarrow R\modl$ has a right adjoint
(``coreflector'') $\Gamma_u\:R\modl\rarrow R\modl_{u\co}$,
computable as $\Gamma_u(A)=\Tor_1^R(K,A)$ for all $A\in R\modl$.
 The category $R\modl_{u\co}$ is a Grothendieck abelian category with
an injective cogenerator $\Gamma_u(J)$, where $J$ is any chosen
injective cogenerator of $R\modl$
\cite[Proposition~3.1 and Corollary~3.6]{BP2}.

 A left $R$\+module $C$ is called a \emph{$u$\+contramodule} (or
a \emph{left $u$\+contramodule}) if $\Hom_R(U,C)=0=\Ext_R^1(U,C)$.
 Assuming that $\pd {}_RU\le1$, the full subcategory $R\modl_{u\ctra}$
of left $u$\+contramodules is closed under kernels, cokernels,
extensions, and direct products in $R\modl$ \cite[Proposition~1.1]{GL};
so $R\modl_{u\ctra}$ is an abelian category and the embedding functor
$R\modl_{u\ctra}\rarrow R\modl$ is exact.
 The embedding functor $R\modl_{u\ctra}\rarrow R\modl$ has a left
adjoint (``reflector'') $\Delta_u\:R\modl\rarrow R\modl_{u\ctra}$,
computable as $\Delta_u(B)=\Ext^1_R(K,B)$ for all $B\in R\modl$.
 The category $R\modl_{u\ctra}$ is a locally presentable abelian category
with a projective generator $\Delta_u(R)\in R\modl_{u\ctra}$
\cite[Proposition~3.2 and Lemma~3.7]{BP2}. {\hbadness=1350\par}

 The following two theorems are the main results of this section.
 
\begin{thm} \label{ring-epi-tilting-cotilting}
 Let $u\:R\rarrow U$ be an injective homological ring epimorphism.
 Assume that\/ $\fd U_R\le1$ and\/ $\pd{}_RU\le1$.
 Then the two abelian categories\/ $\sA=R\modl_{u\co}$ and\/
$\sB=R\modl_{u\ctra}$ are connected by the $1$\+tilting-cotilting
correspondence in the following way.
 The injective cogenerator is $J=\Gamma_u(\Hom_\boZ(R,\boQ/\boZ))
\in\sA$, and the $1$\+tilting object is $T=K\in\sA$.
 The projective generator is $P=\Delta_u(R)\in\sB$, and
the $1$\+cotilting object is $W=\Hom_\boZ(K,\boQ/\boZ)\in\sB$.
 The functor\/ $\Psi\:\sA\rarrow\sB$ is\/ $\Psi=\Hom_R(K,{-})$,
and the functor\/ $\Phi\:\sB\rarrow\sA$ is\/ $\Phi=K\ot_R{-}$.
 The $1$\+tilting class\/ $\sE\subset\sA$ is the class of all\/
$u$\+divisible $u$\+comodule left $R$\+modules, and
the $1$\+cotilting class\/ $\sF\subset\sB$ is the class of all
$u$\+torsionfree $u$\+contramodule left $R$\+modules.
 The equivalence of exact categories\/ $\sE\cong\sF$ is the first
Matlis category equivalence of~\cite[Theorem~1.3]{BP2}.
\end{thm}

 Consider the topological ring $\R=\Hom_R(K,K)^\rop$ opposite to
the ring of endomorphisms of the left $R$\+module $K$, and endow it
with the finite topology, as defined in
Example~\ref{module-endomorphism-topologies}\,(1).
 Then the right action of the ring $R$ in the $R$\+$R$\+bimodule $K$
induces a homomorphism of associative rings $R\rarrow\R$.
 We are interested in the composition of the forgetful functor
$\R\contra\rarrow\R\modl$ defined in
Section~\ref{introd-contramodules-subsecn} with the obvious functor
of restriction of scalars $\R\modl\rarrow R\modl$.

\begin{thm} \label{R-contra-and-u-contra}
 Let $u\:R\rarrow U$ be an injective homological ring epimorphism.
 Assume that\/ $\pd{}_RU\le1$.
 Then the forgetful functor\/ $\R\contra\rarrow R\modl$ is fully
faithful, and its essential image coincides with the full subcategory
of $u$\+contramodule left $R$\+modules $R\modl_{u\ctra}\subset R\modl$.
 So we have an equivalence of abelian categories\/
$\R\contra\cong R\modl_{u\ctra}$.
\end{thm}

\begin{proof}[Proof of Theorems~\ref{ring-epi-tilting-cotilting}
and~\ref{R-contra-and-u-contra}]
 We discuss the proofs of the two theorems simultaneously, because they
are closely related (even though the assumptions in
Theorem~\ref{ring-epi-tilting-cotilting} are slightly more restrictive
than in Theorem~\ref{R-contra-and-u-contra}).

 The argument is largely based in the following result, which is
a particular case of~\cite[Corollary~4.4]{CX}
or~\cite[Corollary~7.3]{BP2}.

\begin{thm} \label{derived-matlis-thm}
 Let $u\:R\rarrow U$ be an injective homological ring epimorphism
such that\/ $\fd U_R\le1$ and\/ $\pd{}_RU\le1$.
 Then, for any derived category symbol\/ $\star=\bb$, $+$, $-$,
or~$\varnothing$, there is a triangulated equivalence between
the derived categories of the abelian categories of left
$u$\+comodules and left $u$\+contramodules,
\begin{equation} \label{derived-matlis-eq}
 \sD^\star(R\modl_{u\co})\cong\sD^\star(R\modl_{u\ctra}).
\end{equation}
\end{thm}

\begin{proof}
 The additional assumptions of~\cite[Corollary~4.4]{CX}
or~\cite[Corollary~7.3]{BP2} hold for all injective ring epimorphisms
by~\cite[Example~7.4]{BP2}.
\end{proof}

 Theorem~\ref{ring-epi-tilting-cotilting} is simplest obtained
by applying Proposition~\ref{tilting-from-derived-equivalence} (for
$n=1$) to the derived equivalence~\eqref{derived-matlis-eq}
(for $\star=\bb$).
 To be more precise, one needs to know a bit about how the derived
equivalence~\eqref{derived-matlis-eq} is constructed.
 In the proof of~\cite[Corollary~7.3]{BP2}, the triangulated equivalence
is obtained from the \emph{recollement} of~\cite[Section~6]{BP2},
and it needs to be shifted by~$[1]$ before it becomes a tilting
derived equivalence.
 The triangulated equivalence in~\cite[Corollary~6.2]{BP2}
is provided by the functors $\boR\Hom_R(K^\bu[-1],{-})$ and
$K^\bu[-1]\ot_R^\boL{-}$, while in our present context one has to
consider the equivalence provided by the functors $\boR\Hom_R(K,{-})$
and $K\ot_R^\boL{-}$.

 Now one observes that the $R$\+$R$\+bimodule $K$ is both a left and
a right $u$\+comodule, and consequently $\Hom_\boZ(K,\boQ/\boZ)$ is
a left $u$\+contramodule.
 Furthermore, one can compute that $\boR\Hom_R(K,K)=\Hom_R(K,K)=
\Ext^1_R(K,R)=\Delta_u(R)=P$, since $\Ext^1_R(K,K)=\Ext^2_R(K,R)=0$.
 Similarly, $\boR\Hom_R(K,J)=\Hom_R(K,J)=
\Hom_R(K,\Hom_\boZ(R,\boQ/\boZ))=\Hom_\boZ(K,\boQ/\boZ)=W$, since
$\Ext^1_R(K,J)=\Ext^1_\sA(K,J)\allowbreak=0$ (as $\sA=R\modl_{u\co}
\subset R\modl$ is a full subcategory closed under extensions).
 Finally, any one of the conditions~(I\+-III) of 
Proposition~\ref{tilting-from-derived-equivalence} is easily verified.
 The descriptions of the classes $\sE\subset\sA$ and $\sF\subset\sB$
follow from~\cite[Lemma~2.7]{BP2}.
 This finishes the proof of Theorem~\ref{ring-epi-tilting-cotilting}.

 Alternatively, one can check that $K\in R\modl_{u\co}$ is
a $1$\+tilting object in the way similar to the argument
in~\cite[Example~5.7]{PS1}.
 Following Examples~\ref{tilting-subcategory-in-modules-exs},
the abelian category $\sB$ corresponding to this tilting object in
the abelian category $\sA=R\modl_{u\co}$ can be described
as $\sB=\R\contra$.
 The functor $\Psi$ is then still computed as $\Psi=\Hom_R(K,{-})$
\cite[Remark~6.6]{PS1}, while the left adjoint functor $\Phi$ is
the functor of so-called \emph{contratensor product} $\Phi=K\ocn_\R{-}$
with  the discrete right $\R$\+module~$K$ \cite[formula~(20)]{PS1}
(which is the same thing as the tensor product $K\ot_R{-}$ provided that
the forgetful functor $\R\contra \rarrow R\modl$ is fully faithful,
cf.~\cite[Lemma~7.11]{PS1}).
 Comparing this approach to the previous one yields $\R\contra\cong
\sB\cong R\modl_{u\ctra}$, that is the assertion of
Theorem~\ref{R-contra-and-u-contra} (in the assumpions of
Theorem~\ref{ring-epi-tilting-cotilting}).

 A direct proof of Theorem~\ref{R-contra-and-u-contra} (in full
generality) can be given based on~\cite[Proposition~2.1]{Pper}.
 For any set $X$, we have to construct a natural isomorphism
of left $R$\+modules $\Delta_u(R[X])\cong\R[[X]]$.
 Indeed,
$$
 \Delta_u(R[X])=\Ext^1_R(K,R[X])\cong\Hom_R(K,K[X])
 \cong\R[[X]]
$$
by~\cite[proof of Theorem~7.1]{PS1}.

 Let us spell out this argument a bit more explicitly.
 There are enough projective objects of the form $P=\Delta_u(R[X])$ in
$R\modl_{u\ctra}$, and these are precisely the images of the free
$\R$\+contramodules $\R[[X]]$ under the forgetful functor.
 To show that the whole image of the forgetful functor $\R\contra
\rarrow R\modl$ lies inside $R\modl_{u\ctra}$, observe that the forgetful
functor preserves cokernels, the full subcategory $R\modl_{u\ctra}
\subset R\modl$ is closed under cokernels, and every left
$\R$\+contramodule is the cokernel of a morphism of free left
$\R$\+contramodules.

 As an abelian category with enough projective objects is determined
by its full subcategory of projective objects, in order to prove
that the functor $\R\contra\rarrow R\modl_{u\ctra}$ is an equivalence
of categories it suffices to show that it is an equivalence in
restriction to the full subcategories of projective objects.
 In other words, we have to check that the natural map
$\Hom^\R(\R[[X]],\R[[Y]])\rarrow\Hom_R(\R[[X]],\R[[Y]])$ is isomorphism
for all sets $X$ and~$Y$.
 Indeed, we have
$$
 \Hom^\R(\R[[X]],\R[[Y]])\cong \R[[Y]]^X\cong\Hom_R(\R[[X]],\R[[Y]]),
$$
where the second isomorphism holds because,
by~\cite[Theorem~1.3]{BP2},
$$
 \Hom_R(\R[[X]],\R[[Y]])\cong\Hom_R(K[X],K[Y])\cong
 \Hom_R(K,K[Y])^X\cong\R[[Y]]^X
$$
as $K[X]$ is a $u$\+divisible left $u$\+comodule and
$\Hom_R(K,K[X])\cong\R[[X]]$.

 The proof of Theorems~\ref{ring-epi-tilting-cotilting}
and~\ref{R-contra-and-u-contra} is finished.
\end{proof} 

\begin{rem}
 The above ``alternative'' argument follows the lines of the exposition
in~\cite[Section~8]{PS1} (see, in particular,
\cite[formulas~(21\+-23)]{PS1}).
 However, the assumptions in~\cite{PS1} presume existence of
a left linear topological ring $\fA$ such that $\sA$ is
the category of discrete left $\fA$\+modules, or in other words,
a hereditary pretorsion class in $\fA\modl$.
 In the context of the present section, $\sA$ is the full abelian
subcategory of left $u$\+comodules in $R\modl$, which is not necessarily
a pretorsion class (see the discussion in~\cite[Section~5]{BP2} and
the examples in~\cite[Section~8]{BP2}).

 Nevertheless, the arguments in the beginning of~\cite[Section~8]{PS1}
are still valid in our present context.
 The key observation is that, for any associative ring $S$,
any $R$\+$S$\+bimodule $E$ whose underlying left $R$\+module is
a $u$\+comodule, and any left $S$\+module $C$, the left $R$\+module
$E\ot_SC$ is a left $u$\+comodule.
 This follows easily from the fact that the full subcategory of left
$u$\+comodules is closed under cokernels and direct sums in $R\modl$.
 So the functor $\Phi=K\ot_R{-}\:R\modl_{u\ctra}\rarrow R\modl_{u\co}$
is well-defined.
 A similar observation holds for the contratensor product in place of
the tensor product; so the functor $\Phi=K\ocn_\R{-}\:\R\contra
\rarrow R\modl_{u\co}$ is well-defined, too.
\end{rem}

\Section{Covers and Direct Limits for Injective Ring Epimorphism}

 In this final section, we discuss the covering and direct limit
closedness properties of the tilting objects $U\oplus K\in R\modl$
and $K\in R\modl_{u\co}$ in connection with the perfectness
properties of the related rings.

 Let $u\:R\rarrow U$ be an injective homological ring epimorphism.
 Assuming that $\pd{}_RU\le1$, denote by $(\sN,\sG)$ the $1$\+tilting
cotorsion pair in $R\modl$ associated with the $1$\+tilting left
$R$\+module $U\oplus K$.
 Assuming that $\fd U_R\le1$ and $\pd{}_RU\le1$, we also have
the $1$\+tilting cotorsion pair $(\sL,\sE)$ in the abelian category
$\sA=R\modl_{u\co}$ associated with the $1$\+tilting object~$K$.

\begin{lem}
\textup{(a)} $\sG\subset R\modl$ is the class of all $u$\+divisible
left $R$\+modules. \par
\textup{(b)} $\sE=\sA\cap\sG$ is the class of all $u$\+divisible
left $u$\+comodules. \par
\textup{(c)} One has\/ $\sL=\sA\cap\sN$.
\end{lem}

\begin{proof}
 By the definition, for a $1$\+tilting left $R$\+module $U\oplus K$
we have $\sG=\{U\oplus K\}^{\perp_1}\subset R\modl$, and it is clear
from the short exact sequence of left $R$\+modules
$0\rarrow R\rarrow U\rarrow K\rarrow0$ that $\sG=\{K\}^{\perp_1}
\subset R\modl$.
 Similarly, $\sE\subset\sA$ is the right $\Ext^1_\sA$-orthogonal class
to the $1$\+tilting object $K\in\sA$.
 Now part~(a) is~\cite[Theorem~3.5\,(4)]{ASan}
or~\cite[Lemma~2.7(b)]{BP2}.
 Part~(b) is a part of Theorem~\ref{ring-epi-tilting-cotilting}
(essentially, it holds because the functors $\Ext^1_R$ and $\Ext^1_\sA$
agree).

 To prove part~(c), we observe that the definitions of the class $\sN$
as the left $\Ext^1_R$\+orthogonal class to $\sG$ in $R\modl$ and
the class $\sL$ as the left $\Ext^1_\sA$\+orthogonal class to $\sE$ in
$\sA$ together with the inclusion $\sE\subset\sG$ imply the inclusion
$\sL\supset\sA\cap\sN$.
 On the other hand, the definitions of $\sN$ as the class of all
finitely $\Add(U\oplus K)$\+coresolved objects in $R\modl$ and $\sL$
as the class of all finitely $\Add(K)$\+coresolved objects in $\sA$
(see the beginning of Section~\ref{tilting-cotilting-secn}
or~\cite[Theorem~3.4]{PS1}) imply the inverse inclusion
$\sL\subset\sA\cap\sN$.
\end{proof}

 Let us start with the $1$\+tilting object $K\in R\modl_{u\co}$.
 Recall that $\R$ denotes the topological ring $\Hom_R(K,K)^\rop$
with the finite topology (see Section~\ref{ring-epi-secn}).
 We keep the notation $\sF$ for the $1$\+cotilting class in
the abelian category $R\modl_{u\ctra}=\sB=\R\contra$ (so the exact
category $\sF$ is equivalent to $\sE=\sA\cap\sG$).

\begin{thm} \label{u-co-contra-covers-thm}
 Assume that\/ $\fd U_R\le1$ and\/ $\pd{}_RU\le1$.
 Then the following sixteen conditions are equivalent:
\begin{enumerate}
\item every left $R$\+module has an\/ $\sA\cap\sN$\+cover;
\item every module from\/ $\sG$ has an\/ $\sA\cap\sN$\+cover;
\item every module from\/ $\sA$ has an\/ $\sA\cap\sN$\+cover;
\item every module from\/ $\sA\cap\sG$ has an\/ $\sA\cap\sN$\+cover;
\item any direct limit of modules from\/ $\Add(K)$ has
an\/ $\sA\cap\sN$\+cover;
\item the class of modules\/ $\sA\cap\sN$ is closed under direct limits;
\item every left $R$\+module has an\/ $\Add(K)$\+cover;
\item every module from\/ $\sG$ has an\/ $\Add(K)$\+cover;
\item every module from\/ $\sA$ has an\/ $\Add(K)$\+cover;
\item every module from\/ $\sA\cap\sG$ has an\/ $\Add(K)$\+cover;
\item any direct limit of modules from\/ $\Add(K)$ has
an\/ $\Add(K)$\+cover;
\item the class of modules\/ $\Add(K)$ is closed under direct limits;
\item every object of\/ $\sB$ has a projective cover;
\item the class of projective objects in\/ $\sB$ is closed under
direct limits;
\item the topological ring\/ $\R$ is topologically perfect;
\item the left $R$\+module $K$ has a perfect decomposition.
\end{enumerate}
\end{thm}

\begin{proof}
 Notice first of all that the direct limits in $\sA$ and $R\modl$ agree
(since $\sA$ is closed under direct limits in $R\modl$).
 The implications (1)\,$\Longrightarrow$\,(2)\,$\Longrightarrow$\,(4),
\,(3)\,$\Longrightarrow$\,(4)\,$\Longrightarrow$\,(5) and
(7)\,$\Longrightarrow$\,(8)\,$\Longrightarrow$\,(10),
\,(9)\,$\Longrightarrow$\,(10)\,$\Longrightarrow$\,(11) are obvious. 

 The implication (3)\,$\Longrightarrow$\,(1) holds because
the embedding functor $\sA\rarrow R\modl$ has a right
adjoint~$\Gamma_u$ (in other words, $\sA$ is coreflective in $R\modl$).
 Given a left $R$\+module $C$, let $L\rarrow\Gamma_u(C)$ be
an $\sA\cap\sN$\+cover of the module $\Gamma_u(C)\in\sA$;
then the composition $L\rarrow\Gamma_u(C)\rarrow C$ is
an $\sA\cap\sN$\+cover of~$C$.

 To check the implication (8)\,$\Longrightarrow$\,(7), recall that
$\sG$ is the class of all $u$\+divisible left $R$\+modules and
$\Add(K)\subset\sG$.
 Every left $R$\+module $C$ has a unique maximal $u$\+divisible
$R$\+submodule $h(C)$.
 Let $M\rarrow h(C)$ be an $\Add(K)$\+cover of $h(C)$; then
the composition $M\rarrow h(C)\rarrow C$ is an $\Add(K)$\+cover
of~$C$.

 The implication (10)\,$\Longrightarrow$\,(8) follows
from~\cite[Lemma~3.3(a)]{BP2}.
 Let $C$ be a $u$\+divisible left $R$\+module; then the left
$R$\+module $\Gamma_u(C)$ belongs to $\sA\cap\sG$.
 If $M\rarrow\Gamma_u(C)$ is an $\Add(K)$\+cover of $\Gamma_u(C)$,
then the composition $M\rarrow\Gamma_u(C)\rarrow C$ is
an $\Add(K)$\+cover of~$C$.

 The equivalence of the three conditions (3), (4), and~(10) is
a particular case of the equivalence of conditions~(1\+-3) in
Proposition~\ref{L-covering-iff-R-properfect}.
 Finally, all the conditions (3), (5), (6), (9), and (11\+-16)
are equivalent by Theorem~\ref{tilting-direct-limits-covers-thm}.
 Notice that $\sA=R\modl_{u\co}$ is a Grothendieck abelian category
by~\cite[Corollary~3.6]{BP2}.

 One can also observe that the class $\Add(K)$ is always precovering
in $R\modl$ by Example~\ref{Add(M)-precovers}; and the class
$\sA\cap\sN$ is precovering in $R\modl$ because $\sA$ is coreflective
in $R\modl$ and $\sA\cap\sN$ is special precovering in~$\sA$.
 Hence the implications (6)\,$\Longrightarrow$\,(1) and
(12)\,$\Longrightarrow$\,(7) hold by Enochs' theorem
(see Theorem~\ref{direct-limits-imply-covers}).
\end{proof} 

 We recall from~\cite[Section~10]{Pproperf} that a topological ring $\R$
is said to be \emph{left pro-perfect} if it is separated and complete,
two-sided linear, and all the discrete quotient rings of $\R$ are
left perfect.

\begin{thm} \label{A-cap-N-covering-iff-R-properfect-thm}
 Let $u\:R\rarrow U$ be an injective homological ring epimorphism.
 Assume that\/ $\fd U_R\le1$ and\/ $\pd{}_RU\le1$, and assume further
that the topological ring\/ $\R$ satisfies one of the conditions~(a),
(b), (c), or~(d) of Section~\ref{seven-conditions-secn}.
 Then the conditions in Theorem~\ref{u-co-contra-covers-thm} are
equivalent to the following ones:
\begin{enumerate}
\item any countable direct limit of copies of the left $R$\+module $K$
has an\/ $\sA\cap\sN$\+cover;
\item the class of left $R$\+modules\/ $\sA\cap\sN$ is closed
under countable direct limits; 
\item any countable direct limit of copies of the $R$\+module $K$
has an\/ $\Add(K)$\+cover;
\item the class of left $R$\+modules\/ $\Add(K)$ is closed under
countable direct limits;
\item any countable direct limit of copies of the projective generator
$P=\R$ has a projective cover in\/~$\sB$;
\item the class of objects\/ $\sB_\proj$ is closed under countable
direct limits in\/~$\sB$;
\item all descending chains of cyclic discrete right\/ $\R$\+modules
terminate;
\item all the discrete quotient rings of the topological ring\/ $\R$
are left perfect.
\end{enumerate}
 In particular, if the ring $R$ is commutative and $\pd{}_RU\le1$,
then the eight conditions~\textup{(1\+-8)} are equivalent to each
other and to the conditions in Theorem~\ref{u-co-contra-covers-thm}.
 The condition~\textup{(8)} can be rephrased by saying that
the topological ring\/ $\R$ is pro-perfect in this case.
 Replacing the assumption of one of the conditions~(a\+-d) with that
of one of the conditions~(e), (f), or~(g), the seven
conditions~\textup{(1\+-7)} are equivalent to each other and
to all the conditions in Theorem~\ref{u-co-contra-covers-thm}.
\end{thm}

\begin{proof}
 The conditions~(2), (4), and~(6) are equivalent to each other by
Corollary~\ref{countable-direct-limits-closedness-cor}.

 In the assumption of any one of the conditions~(a\+-d),
all the conditions~(5\+-8) are equivalent to each other and to
the conditions in Theorem~\ref{u-co-contra-covers-thm}\,(13\+-15)
by~\cite[Theorem~12.4]{Pproperf}.
 In the assumption of any one of the conditions~(a\+-g), all
the conditions (1), (3), and (5\+-7) are equivalent to each other and
to the conditions in
Theorem~\ref{u-co-contra-covers-thm}\,(3\+-4,\,10,\,13\+-16)
by Proposition~\ref{L-covering-iff-R-properfect}.

 Alternatively, all the conditions~(3\+-7) are equivalent to each other
and to the conditions in Theorem~\ref{u-co-contra-covers-thm}\,(7\+-16)
by Corollary~\ref{a-b-c-d-THEC-main-corollary}.
 Notice that the left $R$\+module $K$ is always self-pure-projective
by Examples~\ref{are-self-pure-projective-easy}\,(3)
and~\ref{are-self-pure-projective-harder}\,(2), as a direct summand
of a $1$\+tilting left $R$\+module $U\oplus K$.
 Besides, $K$ is also $\Sigma$\+rigid, of course; so it
satisfies THEC by Example~\ref{sigma-pure-rigid-example}\,(1).

 If the ring $R$ is commutative, then so is the ring~$\R$
by~\cite[Lemma~4.1]{BP2}.
 So condition~(a) is satisfied.
 (It is worth recalling that $\pd{}_RU\le1$ implies $\fd{}_RU=0$ for
commutative rings $R$, by~\cite[Theorem~5.2]{BP2}.)
\end{proof}

 Now let us discuss the $1$\+tilting left $R$\+module $U\oplus K$.
 We denote by $\S$ the topological ring
$\Hom_R(U\oplus K,\>U\oplus K)^\rop$ with the finite topology, and
denote by $\sH\subset\S\contra$ the $1$\+cotilting class associated
with the $1$\+cotilting left $\S$\+contramodule
$\Hom_\boZ(U\oplus K,\>\boQ/\boZ)$.
 So the exact category $\sH$ is equivalent to~$\sG$. 

\begin{lem} \label{S-R-and-U-lemma}
\textup{(i)} The topological ring\/ $\S$ is topologically left perfect
if and only if the ring $U$ is left perfect and the topological ring
$\R$ is topologically left perfect. \par
\textup{(ii)} All the discrete quotient rings of the topological ring\/
$\S$ are left perfect if and only if the ring $U$ is left perfect and
all the discrete quotient rings of the topological ring $\R$ are
left perfect. \par
\textup{(iii)} If the topological ring\/ $\R$ satisfies one of
the conditions~(a), (b), (c), or~(d) of
Section~\ref{seven-conditions-secn}, then the topological
ring\/ $\S$ satisfies condition~(d). \par
\textup{(iv)} If the topological ring\/ $\R$ satisfies one of
the conditions~(e), (f), or~(g) of
Section~\ref{seven-conditions-secn}, then the topological
ring $\S$ satisfies condition~(g).
\end{lem}

\begin{proof}
 We have $\Hom_R(U,U)^\rop=U$, \ $\Hom_R(K,K)^\rop=\R$, and
$\Hom_R(U/R,U)=0$.
 So $\S$ is the matrix ring (cf.~\cite[Example~12.1]{Pproperf})
$$
 \begin{pmatrix}
 U & \fK \\
 0 & \R
 \end{pmatrix}
$$
where $\fK=\Hom_R(U,U/R)$ is a nilpotent strongly closed two-sided
ideal in $\S$ (obviously, $\fK^2=0$ in~$\S$).
 Now we have $\S/\fK=U\times\R$, so part~(ii) of the lemma follows
from~\cite[Lemma~12.3]{Pproperf}.
 Similarly, part~(i) follows from
Lemmas~\ref{topologically-perfect-quotient-by-T-nilpotent},
\ref{topologically-perfect-closed-under-products},
and~\ref{topologically-perfect-closed-under-quotients}.
 Furthermore, the discrete ring $U$ trivially satisfies
the condition~(b) of Section~\ref{seven-conditions-secn}.
 Hence it remains to apply~\cite[Lemma~12.6]{Pproperf} in order
to prove part~(iii) of the lemma; and part~(iv) is
a particular case of Lemma~\ref{class-g-closure-properties}.
\end{proof}

\begin{thm} \label{N-covering-iff-R-properfect-and-U-perfect}
 Let $u\:R\rarrow U$ be an injective homological ring epimorphism
such that\/ $\pd{}_RU\le1$.
 Then the following thirteen conditions are equivalent:
\begin{enumerate}
\item all left $R$\+modules have\/ $\sN$\+covers;
\item any countable direct limit of copies of the $R$\+module
$U\oplus K$ has an\/ $\sN$\+cover;
\item the class of left $R$\+modules\/ $\sN$ is closed
under (countable) direct limits; 
\item all left $R$\+modules have\/ $\Add(U\oplus K)$\+covers;
\item any countable direct limit of copies of the $R$\+module
$U\oplus K$ has an\/ $\Add(U\oplus K)$-cover;
\item the class of left $R$\+modules\/ $\Add(U\oplus K)$ is closed under
(countable) direct limits;
\item the left $R$\+module $U\oplus K$ is $\Sigma$\+pure-split;
\item the left $R$\+module $U\oplus K$ has a perfect decomposition;
\item all the objects of\/ $\S\contra$ have projective covers;
\item any countable direct limit of copies of the free left\/
$\S$\+contramodule\/ $\S$ has a projective cover in\/~$\S\contra$;
\item the class of all projective left\/ $\S$\+contramodules is closed
under (countable) direct limits in\/ $\S\contra$;
\item the topological ring\/ $\S$ is topologically left perfect;
\item the ring $U$ is left perfect and the topological ring\/ $\R$
is topologically left perfect.
\end{enumerate}
 Furthermore, consider the next four properties:
\begin{enumerate}
\setcounter{enumi}{13}
\item all descending chains of cyclic discrete right\/ $\S$\+modules
terminate;
\item the ring $U$ is left perfect and all descending chains of
cyclic discrete right\/ $\R$\+modules terminate;
\item all the discrete quotient rings of the topological ring\/ $\S$
are left perfect;
\item the ring $U$ is left perfect and all the discrete quotient rings
of the topological ring\/ $\R$ are left perfect.
\end{enumerate}
 Then the following implications hold:
$$
 \mathrm{(13)} \Longrightarrow
 \mathrm{(14)} \Longrightarrow \mathrm{(15)} \Longrightarrow
 \mathrm{(16)} \Longleftrightarrow \mathrm{(17)}.
$$
 If the topological ring\/ $\R$ satisfies one of the conditions~(e),
(f), or~(g) of Section~\ref{seven-conditions-secn}, then all
the conditions~\textup{(1\+-14)} are equivalent to each other.
 If the topological ring\/ $\R$ satisfies one of the conditions~(a),
(b), (c), or~(d), then all the conditions~\textup{(1\+-17)}
are equivalent to each other.
 In particular, if the ring $R$ is commutative, then the seventeen
conditions~\textup{(1\+-17)} are equivalent.
\end{thm}
 
\begin{proof}
 The condition~(3) (for uncountable direct limits) is equivalent
to~(7) by~\cite[Proposition~13.55]{GT}.
 All the eight conditions~(1\+-8) are equivalent to each other
by~\cite[Theorem~3.6, Theorem~5.2, and Corollary~5.5]{AST}.

 The conditions~(3), (6), and~(11) are equivalent to
each other, for countable direct limits, by
Corollary~\ref{countable-direct-limits-closedness-cor},
and for uncountable ones, by
Corollary~\ref{uncountable-direct-limit-closedness-cor}.
 The conditions~(2), (5), and~(10) are equivalent to each other
by Proposition~\ref{L-covering-iff-R-properfect}\,%
(5)\,$\Leftrightarrow$\,(6)\,$\Leftrightarrow$\,(7).
 All the conditions (1), (4), (8), (9), and~(12), and the uncountable
versions of (3), (6), (11) are equivalent to each other by
Theorem~\ref{tilting-direct-limits-covers-thm}.  {\hbadness=1050\par}

 The implications (12)\,$\Longrightarrow$\,(14)\,$\Longrightarrow$\,(16)
and (13)\,$\Longrightarrow$\,(15)\,$\Longrightarrow$\,(17) hold
by~\cite[Theorem~14.4\,(iv)\,$\Rightarrow$\,(v)%
\,$\Rightarrow$\,(vi)]{PS3}.
 The equivalences (12)\,$\Longleftrightarrow$\,(13) and
(16)\,$\Longleftrightarrow$\,(17) hold
by Lemma~\ref{S-R-and-U-lemma}(i\+-ii).
 The implication (14)\,$\Longrightarrow$\,(15) is easy (cf.\
the discussion in the proof of Theorem~\ref{g-implies-as}, case~(g)).

 If $\R$ satisfies one of the conditions~(a), (b), (c), or~(d), then
all the conditions (9\+-12), (14), and~(16) are equivalent to
each other by Lemma~\ref{S-R-and-U-lemma}(iii)
and~\cite[Theorem~12.4]{Pproperf}.
 In the assumption of any one of the conditions~(a\+-g), all
the conditions (1), (2), (5), (8\+-12), and~(14), are equivalent
to each other by Lemma~\ref{S-R-and-U-lemma}(iv) and
Proposition~\ref{L-covering-iff-R-properfect}.
 This also establishes the equivalence of the countable and
uncountable versions of the condition~(11).

 Alternatively, all the conditions~(4\+-6), (8\+-12), and~(14) are
equivalent to each other by Corollary~\ref{a-b-c-d-THEC-main-corollary}.
 Notice that the left $R$\+module $U\oplus K$
satisfies THEC by Example~\ref{sigma-pure-rigid-example}\,(1)
(it is also self-pure-projective
by Example~\ref{are-self-pure-projective-harder}\,(2)).

 The last assertion of the theorem follows from~\cite[Lemma~4.1]{BP2}.
\end{proof} 

\begin{ex}
 Let $R$ be a commutative ring and $S\subset R$ be a multiplicative
subset consisting of regular elements.
 Denote the multiplicative subset of all regular elements in $R$
by $S\subset S_{\mathit{reg}}\subset R$.
 Set $U=S^{-1}R$; then the localization map $u\:R\rarrow U$ is
an injective flat epimorphism of commutative rings.
 The topological ring $\R=\Hom_R(U/R,U/R)$ is naturally topologically
isomorphic to the $S$\+completion $\varprojlim_{s\in S}R/sR$ of the ring
$R$ (viewed as the topological ring in the projective limit topology),
which was discussed in~\cite[Example~11.2]{Pproperf}.

 Assume that $\pd{}_RS^{-1}R\le1$, and set $K=U/R$.
 Then the homomorphism of commutative rings $R\rarrow S^{-1}R=U$
satisfies the assumptions of
Theorems~\ref{A-cap-N-covering-iff-R-properfect-thm}
and~\ref{N-covering-iff-R-properfect-and-U-perfect}.
 By Theorem~\ref{A-cap-N-covering-iff-R-properfect-thm}, the class of
$R$\+modules $\sA\cap\sN$ is covering (if and only if the class
$\Add(K)\subset R\modl$ is covering and) if and only if the ring $R/sR$
is perfect for every $s\in S$.
 By Theorem~\ref{N-covering-iff-R-properfect-and-U-perfect}, the class
of $R$\+modules $\sN$ is covering (if and only if the class
$\Add(U\oplus K)\subset R\modl$ is covering and) if and only if
two conditions hold: the ring $R/sR$ is perfect for every $s\in S$,
\emph{and} the ring $S^{-1}R$ is perfect.

 The latter two conditions are equivalent to the following two:
one has $S^{-1}R=S_{\mathit{reg}}^{-1}R$, and the ring $R$ is
almost perfect (in the sense of the paper~\cite{FS}).
 It is worth noticing that the condition that all the rings $R/sR$
are perfect already implies $\pd{}_RS^{-1}R\le1$ \cite[Lemma~3.4]{FS},
\cite[Theorem~6.13]{BP1}.

 For example, let $R=\boZ$ be the ring of integers, $p$~be a prime
number, and $S=\{1,p,p^2,p^3,~\dotsc\}\subset R$ be the multiplicative
subset in $\boZ$ generated by~$p$.
 Then the class of abelian groups $\sA\cap\sN\subset\Ab$ is
covering, but the class $\sN\subset\Ab$ is not.
 Alternatively, let $S'\subset\boZ$ be the multiplicative subset of
all integers \emph{not} divisible by~$p$.
 Then, once again, the related class $\sA\cap\sN'\subset\Ab$ is
covering, but the class of abelian groups $\sN'\subset\Ab$ is not.
\end{ex}

\end{document}